\documentclass[a4]{article}
\usepackage{newtxtext}
\usepackage{amsmath,amsthm,amssymb,amscd,stmaryrd}
\usepackage[all]{xy}
\usepackage{tikz-cd}
\usepackage{tikz}
\usetikzlibrary{matrix}
\usepackage{graphicx}
\usepackage{hyperref}
\usepackage{authblk}

\newtheorem{thm}{Theorem}
\newtheorem{lem}[thm]{Lemma}
\newtheorem{prop}[thm]{Proposition}
\newtheorem{cor}[thm]{Corollary}

\newtheorem{prob}[thm]{Problem}

\theoremstyle{definition}
\newtheorem{dfn}[thm]{Definition}
\newtheorem{rem}[thm]{Remark}

\def\Dom{\mathop{\mathrm{Dom}}\nolimits}
\def\Hom{\mathop{\mathrm{Hom}}\nolimits}
\def\End{\mathop{\mathrm{End}}\nolimits}

\def\ad{\mathop{\mathrm{ad}}\nolimits}
\def\id{\mathop{\mathrm{id}}\nolimits}
\def\im{\mathop{\mathrm{im}}\nolimits}
\def\coker{\mathop{\mathrm{coker}}\nolimits}

\def\SL{\mathop{\mathrm{SL}}\nolimits}
\def\PSL{\mathop{\mathrm{PSL}}\nolimits}

\def\PSO{\mathop{\mathrm{PSO}}\nolimits}

\def\sl{\mathop{\mathfrak{sl}}\nolimits}
\def\so{\mathop{\mathfrak{so}}\nolimits}

\def\tr{\mathop{\mathrm{tr}}\nolimits}

\def\Im{\mathop{\mathrm{Im}}\nolimits}

\newcommand{\mf}[1]{{\mathfrak{#1}}}
\newcommand{\bb}[1]{{\mathbb{#1}}}
\newcommand{\mca}[1]{{\mathcal{#1}}}
\newcommand{\ol}[1]{{\overline{#1}}}

\newcommand{\df}{{d_\mathcal{F}}}

\title{De Rham cohomology of the weak stable foliation\\
of the geodesic flow of a hyperbolic surface}
\author[1]{Hirokazu Maruhashi\thanks{maruhashihirokazu@gmail.com}}
\author[2]{Mitsunobu Tsutaya\thanks{tsutaya@math.kyushu-u.ac.jp}}
\affil[1]{Graduate School of Mathematical Sciences, The University of Tokyo, Tokyo 153-8914, Japan}
\affil[2]{Faculty of Mathematics, Kyushu University, Fukuoka 819-0395, Japan}
\date{\empty}

\allowdisplaybreaks

\begin{document}
\maketitle

\begin{abstract}
We compute the de Rham cohomology of the weak stable foliation of the geodesic flow of a connected orientable closed hyperbolic surface with various coefficients. For most of the coefficients, we also give certain ``Hodge decompositions'' of the corresponding de Rham complexes, which are not obtained by the usual Hodge theory of foliations. These results are based on unitary representation theory of $\PSL(2,\bb{R})$. As an application we obtain an answer to a problem considered by Haefliger and Li around 1980. 
\end{abstract}

\tableofcontents

\section{Introduction}
\subsection{What will be done in this paper}
Let $\Sigma$ be a connected orientable closed hyperbolic surface and $\Phi$ be the geodesic flow of $\Sigma$ on the unit tangent bundle $T^1\Sigma$ of $\Sigma$. Then $\Phi$ is known to be an Anosov flow and we can consider the weak stable foliation $\mca{F}$ of $\Phi$, which is a $C^\infty$ foliation of $T^1\Sigma$ in this case. 

We have another well-known description of the foliation $\mca{F}$. Let $\Gamma$ be a torsion free cocompact lattice of $\PSL(2,\bb{R})$. Then $\Sigma=\Gamma\backslash\PSL(2,\bb{R})/\PSO(2)$ is a connected oriented closed hyperbolic surface and $P=\Gamma\backslash\PSL(2,\bb{R})$ is the unit tangent bundle of $\Sigma$. Consider the subgroups 
\begin{gather*}
A=\left\{\pm
\begin{pmatrix}
a&\\
&a^{-1}
\end{pmatrix}
\ \middle|\ a>0\right\},\quad N=\left\{\pm
\begin{pmatrix}
1&b\\
&1
\end{pmatrix}
\ \middle|\ b\in\bb{R}\right\}
\end{gather*}
and 
\begin{equation*}
AN=\left\{\pm
\begin{pmatrix}
a&b\\
&a^{-1}
\end{pmatrix}
\ \middle|\ a>0,\ b\in\bb{R}\right\}
\end{equation*}
of $\PSL(2,\bb{R})$. Then the action $P\curvearrowleft A$ by right multiplication is identified with the geodesic flow of $\Sigma$ and the orbit foliation $\mca{F}$ of the action $P\curvearrowleft AN$ by right multiplication is the weak stable foliation of the geodesic flow of $\Sigma$. 

The objectives of this paper are to compute the de Rham cohomology of this foliation $\mca{F}$ with certain coefficients, and to give a kind of Hodge decompositions to cochain complexes defining the cohomology for most of the coefficients. We do not use the usual Hodge theory of foliations to obtain these Hodge decompositions. (See \'{A}lvarez L\'{o}pez--Tondeur \cite{AT} and \'{A}lvarez L\'{o}pez--Kordyukov \cite{AK}, where the existences of the usual Hodge decompositions are proved for a Riemannian foliation.) Since the foliation $\mca{F}$ is not Riemannian, we cannot apply the theory to $\mca{F}$. See Deninger--Singhof \cite{DS} for a one-dimensional foliation on a three-dimensional Heisenberg nilmanifold, which is not Riemannian and for which the usual Hodge decomposition does not hold. 

Let $T\mca{F}$ be the tangent bundle of $\mca{F}$, ie a subbundle of $TP$ consisting of vectors tangent to leaves of $\mca{F}$. Let $\mf{an}$ be the Lie algebra of $AN$. For a finite dimensional real or complex representation $V$ of $\mf{an}$, the de Rham cohomology $H^*(\mca{F};V)$ of the foliation $\mca{F}$ with coefficient $V$ is the cohomology of the cochain complex $\Gamma\left(\bigwedge^*T^*\mca{F}\otimes V\right)$, the space of $C^\infty$ sections of the bundle $\bigwedge^*T^*\mca{F}\otimes V$. (A detailed definition of the cohomology is given in Section \ref{de Rham cohomology}.) 

In \cite{MM} Matsumoto and Mitsumatsu proved 
\begin{equation*}
H^1(\mca{F};\bb{R})\simeq H^1(P;\bb{R})\oplus H^1(\mf{an};\bb{R}), 
\end{equation*}
where the right hand side is a direct sum of a de Rham cohomology and a Lie algebra cohomology. (To be more precise, they take a (not necessarily torsion free) cocompact lattice of the universal cover $\widetilde{\SL}(2,\bb{R})$ of $\SL(2,\bb{R})$ instead of that in $\PSL(2,\bb{R})$.) 

Consider a basis 
\begin{equation*}
H=\frac{1}{2\sqrt{2}}
\begin{pmatrix}
1&0\\
0&-1
\end{pmatrix}
,\quad E=\frac{1}{2}
\begin{pmatrix}
0&1\\
0&0
\end{pmatrix}
\end{equation*}
of $\mf{an}$. For $\lambda\in\bb{C}$, define a $1$-dimensional complex representation $\bb{C}_\lambda$ by 
\begin{equation*}
\bb{C}_\lambda=\bb{C},\quad H1=\lambda1,\quad E1=0. 
\end{equation*}
A $1$-dimensional real representation $\bb{R}_\lambda$ for $\lambda\in\bb{R}$ is defined similarly. 

In this paper we will compute $H^*(\mca{F};V)$ for 
\begin{equation*}
V=\bb{C}_\lambda,\ \bb{R}_\mu,\ \mf{an}\otimes\bb{C},\ \mf{an},\ \sl(2,\bb{R})\otimes\bb{C},\ \sl(2,\bb{R})
\end{equation*}
for all $\lambda\in\bb{C}$ and $\mu\in\bb{R}$. See Section \ref{comprrrr} for the computation results for $V=\bb{C}_\lambda$, $\bb{R}_\lambda$, Corollary \ref{ancohrrreee} for $V=\mf{an}\otimes\bb{C}$, $\mf{an}$, and Section \ref{slslsls2rr} for $V=\sl(2,\bb{R})\otimes\bb{C}$, $\sl(2,\bb{R})$. The method is completely different from that in Matsumoto--Mitsumatsu \cite{MM}. As far as we know even $H^2(\mca{F};\bb{R})$ was not computed. 

We not only compute the de Rham cohomologies but also construct a kind of Hodge decompositions of $\Gamma\left(\bigwedge^*T^*\mca{F}\otimes V\right)$. This will be done when $V=\mf{an}\otimes\bb{C}$ or $V=\bb{C}_\lambda$ for 
\begin{equation}\label{conditionlll}
\lambda\not\in\left\{\frac{1\pm\sqrt{1-4\nu}}{2\sqrt{2}}\ \middle|\ \nu\in\sigma(\ol{\Delta_\Sigma})\setminus\{0\}\right\}, 
\end{equation}
where $\sigma(\ol{\Delta_\Sigma})$ is the spectrum of the closure $\ol{\Delta_\Sigma}$ of the Laplacian $\Delta_\Sigma$ of the hyperbolic surface $\Sigma$. See Figure \eqref{figure1}. 
\begin{figure}[hbtp]
\begin{center}
\includegraphics[keepaspectratio, scale=0.5]{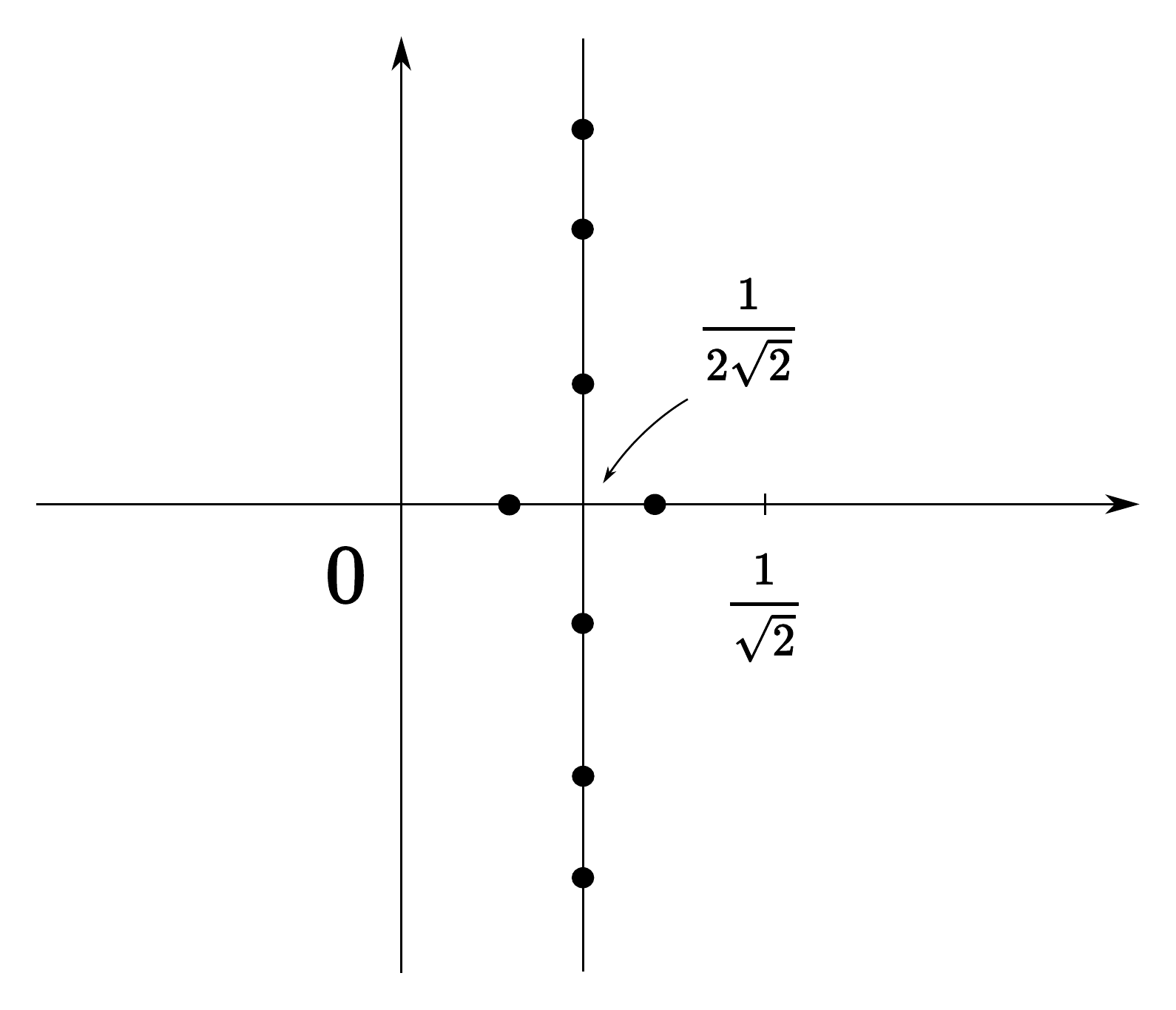}
\caption{$\left\{\frac{1\pm\sqrt{1-4\nu}}{2\sqrt{2}}\ \middle|\ \nu\in\sigma(\ol{\Delta_\Sigma})\setminus\{0\}\right\}$}
\label{figure1}
\end{center}
\end{figure}

To be more precise, we will construct an operator $\delta$ on $\Gamma\left(\bigwedge^*T^*\mca{F}\otimes V\right)$ of degree $-1$ such that $\delta^2=0$ and, defining $\Delta=\df\delta+\delta\df$, where $\df$ is the differential of $\Gamma\left(\bigwedge^*T^*\mca{F}\otimes V\right)$, we have 
\begin{align*}
\Gamma\left(\bigwedge^*T^*\mca{F}\otimes V\right)&=\ker\Delta\oplus\im\df\oplus\im\delta\\
\ker\df&=\ker\Delta\oplus\im\df\\
\ker\delta&=\ker\Delta\oplus\im\delta\\
\ker\Delta&=\ker\df\cap\ker\delta\\
\im\Delta&=\im\df\oplus\im\delta. 
\end{align*}
It follows that $H^*\left(\mca{F};V\right)\simeq\ker\Delta$.

\subsection{How to compute the cohomologies}
Assume that $V$ is a finite dimensional complex representation of $\mf{an}$. Computations of the cohomologies are based on an isomorphism 
\begin{equation}\label{isommmo}
\Gamma\left(\bigwedge^*T^*\mca{F}\otimes V\right)\simeq C^*\left(\mf{an};C^\infty(P,\bb{C})\otimes V\right)
\end{equation}
of cochain complexes, where the right hand side is the Chevalley--Eilenberg complex. See Proposition \ref{drce}. Here $\mf{an}\curvearrowright C^\infty(P,\bb{C})$ is the representation obtained from the action $P\curvearrowleft AN$ and $\mf{an}\curvearrowright C^\infty(P,\bb{C})\otimes V$ is the tensor product of representations. 

From the action $P\curvearrowleft\PSL(2,\bb{R})$ by right multiplication we get a unitary representation $\PSL(2,\bb{R})\curvearrowright L^2(P)$. The representation $L^2(P)$ decomposes into countably infinitely many irreducible unitary representations of $\PSL(2,\bb{R})$. We have a classification of irreducible unitary representations of $\PSL(2,\bb{R})$ and there are formulas for the multiplicities of irreducible unitary representations in $L^2(P)$. Since $C^\infty(P,\bb{C})$ coincides with the space of $C^\infty$ vectors in $L^2(P)$, by taking the spaces of $C^\infty$ vectors, we can obtain a decomposition of $C^\infty(P,\bb{C})$ into representations of $\sl(2,\bb{R})$. Hence the complex $C^*\left(\mf{an};C^\infty(P,\bb{C})\otimes V\right)$ decomposes into complexes of the form $C^*\left(\mf{an};C^\infty(\mca{H})\otimes V\right)$, where $\mca{H}$ is an irreducible unitary representation of $\PSL(2,\bb{R})$ and $C^\infty(\mca{H})$ is the space of $C^\infty$ vectors in $\mca{H}$. (We must be careful with the meanings of the direct sums here, since we are dealing with infinite direct sums. We actually avoid dealing with infinite direct sums when considering the complex $C^*\left(\mf{an};C^\infty(P,\bb{C})\otimes V\right)$. We decompose it into a finite direct sum. See Section \ref{cconsthoxce} below.) 

When $V=\bb{C}_\lambda$ for $\lambda\in\bb{C}$, we compute $H^*\left(\mf{an};C^\infty(\mca{H})\otimes\bb{C}_\lambda\right)$ for an irreducible unitary representation $\mca{H}$ of $\PSL(2,\bb{R})$ by using the Hochschild--Serre spectral sequence for the pair $\mf{n}\subset\mf{an}$, where $\mf{n}$ is the Lie algebra of $N$. During the computation we need to know $H^1(\mf{n};C^\infty(\mca{H}))$, but this cohomology was computed by Flaminio and Forni in \cite{FF}. 

Since $C^*\left(\mf{an};C^\infty(P,\bb{C})\otimes\bb{C}_\lambda\right)$ is an infinite direct sum of complexes of the form $C^*\left(\mf{an};C^\infty(\mca{H})\otimes\bb{C}_\lambda\right)$, to go from the computations of $H^*\left(\mf{an};C^\infty(\mca{H})\otimes\bb{C}_\lambda\right)$ to $H^*\left(\mf{an};C^\infty(P,\bb{C})\otimes\bb{C}_\lambda\right)$, we must consider convergence problems. It is possible to solve them by explicit computations using certain realizations of irreducible unitary representations $\mca{H}$, but we will not explain this in this paper. Instead we will use one of the following Hodge decompositions to deal with the infinite direct sum.

\subsection{Construction of Hodge decompositions}\label{cconsthoxce}
By the isomorphism \eqref{isommmo}, to obtain a Hodge decomposition of $\Gamma\left(\bigwedge^*T^*\mca{F}\otimes V\right)$, we will construct one for $C^*\left(\mf{an};C^\infty(P,\bb{C})\otimes V\right)$. We do it for $V=\bb{C}_\lambda$, where $\lambda\not\in\left\{\frac{1\pm\sqrt{1-4\nu}}{2\sqrt{2}}\ \middle|\ \nu\in\sigma(\ol{\Delta_\Sigma})\setminus\{0\}\right\}$, in Sections \ref{firsthodge99}, \ref{easy} and \ref{discrete series}, and for $V=\mf{an}\otimes\bb{C}$ in Section \ref{hodgean}. Here we consider the case when $V=\bb{C}_0$ for simplicity. We have a decomposition 
\begin{equation*}
L^2(P)=\bb{C}\oplus g_\Sigma D_1^+\oplus g_\Sigma D_1^-\oplus\bigoplus_iW_i
\end{equation*}
of a unitary representation of $\PSL(2,\bb{R})$ into irreducible unitary representations, where $D_1^\pm$ are certain discrete series representations, $W_i$'s are irreducible unitary representations with nonzero Casimir parameter and $g_\Sigma$ is the genus of the surface $\Sigma$. By taking the $C^\infty$ vectors of both sides we get a decomposition 
\begin{equation*}
C^\infty(P,\bb{C})=\bb{C}\oplus g_\Sigma C^\infty(D_1^+)\oplus g_\Sigma C^\infty(D_1^-)\oplus C^\infty\left(\bigoplus_iW_i\right)
\end{equation*}
of a representation of $\sl(2,\bb{R})$, where $C^\infty(\ \cdot\ )$ denotes the space of $C^\infty$ vectors. Hence we obtain a decomposition 
\begin{align*}
C^*\left(\mf{an};C^\infty(P,\bb{C})\right)&\simeq C^*(\mf{an};\bb{C})\oplus g_\Sigma C^*(\mf{an};C^\infty(D_1^+))\oplus g_\Sigma C^*(\mf{an};C^\infty(D_1^-))\\
&\quad\oplus C^*\left(\mf{an};C^\infty\left(\bigoplus_iW_i\right)\right)
\end{align*}
of a cochain complex. We will construct a Hodge decomposition for each summand of this decomposition. In Section \ref{firsthodge99}, we construct an operator $\delta$ on $C^*\left(\mf{an};C^\infty\left(\bigoplus_iW_i\right)\right)$ of degree $-1$ such that $d\delta+\delta d=\id$. Hence this summand does not contribute to the cohomology. The complex $C^*(\mf{an};\bb{C})$ is finite dimensional and it is easy to construct a Hodge decomposition for it (Section \ref{easy}). In Section \ref{discrete series}, we construct an operator $\delta$ on $C^*(\mf{an};C^\infty(D_1^\pm))$ of degree $-1$ such that $d\delta+\delta d=\id-p$, where $p$ is a projection onto the space of ``harmonic elements'', which is a finite dimensional subspace. These Hodge decompositions give a Hodge decomposition of the entire cochain complex $C^*\left(\mf{an};C^\infty(P,\bb{C})\right)$.

\subsection{An application}
By restricting differential forms, we have a homomorphism $H^*(P;\bb{R})\to H^*(\mca{F};\bb{R})$. We will prove in Section \ref{restmapH} that the homomorphism is an isomorphism on the second cohomology: 
\begin{equation*}
H^2(\mca{F};\bb{R})\simeq H^2(P;\bb{R}). 
\end{equation*}
As an application of this isomorphism, we can give an answer to a problem considered by Haefliger and Li in \cite{HL}. Let $\mca{G}$ be a $C^\infty$ foliation of a $C^\infty$ manifold $M$. We say $\mca{G}$ is totally minimizable if any $C^\infty$ Riemannian metric $g$ of $T\mca{G}$ can be extended to a $C^\infty$ Riemannian metric $\tilde{g}$ of $TM$ such that each leaf of $\mca{G}$ is a minimal submanifold of the Riemannian manifold $(M,\tilde{g})$. Haefliger studied this property in \cite{Haef} for general foliations. Haefliger and Li studied it for the weak stable foliation $\mca{F}$ in \cite{HL} but they did not succeed in determining whether $\mca{F}$ has the property or not. This property was later named as the total minimizability by Matsumoto in \cite{Mat}. As stated in \cite{Mat} if the restriction map $H^2(P;\bb{R})\to H^2(\mca{F};\bb{R})$ is surjective, then $\mca{F}$ is totally minimizable. Hence $\mca{F}$ is totally minimizable.

\subsection{Convention and notations}
\begin{itemize}
\item For a vector space $V$, the definition of the wedge product in this paper is 
\begin{align*}
&(\varphi\wedge\psi)(X_1,\ldots,X_{p+q})\\
&:=\frac{1}{p!q!}\sum_{\sigma\in\mf{S}_{p+q}}(-1)^\sigma\varphi(X_{\sigma(1)},\ldots,X_{\sigma(p)})\psi(X_{\sigma(p+1)},\ldots,X_{\sigma(p+q)}), 
\end{align*}
where $\varphi\in\bigwedge^pV^*$, $\psi\in\bigwedge^qV^*$ and $X_1,\ldots,X_{p+q}\in V$. 
\item For a $C^\infty$ vector bundle $E$, $\Gamma(E)$ denotes the space of all $C^\infty$ sections of $E$. 
\item For manifolds $M$ and $N$, the set of all $C^\infty$ maps from $M$ to $N$ is denoted by $C^\infty(M,N)$. 
\end{itemize}

\section{De Rham cohomology of a foliation and Lie algebra cohomology}
In this section we review the definition of the de Rham cohomology of a foliation with coefficients in a vector bundle with a flat partial connection, a construction of a flat partial connection from a locally free action and a representation and an isomorphism between the de Rham cohomology and a Lie algebra cohomology.

\subsection{De Rham cohomology of a foliation}\label{de Rham  cohomology}
Let $\mca{F}$ be a $C^\infty$ foliation on a $C^\infty$ manifold $M$. Recall that the tangent bundle $T\mca{F}$ of the foliation $\mca{F}$ is a $C^\infty$ subbundle of $TM$ defined by 
\begin{equation*}
T_x\mca{F}=T_xL\subset T_xM
\end{equation*}
for $x\in M$, where $L$ is the leaf of $\mca{F}$ through $x$. An element of $\Gamma\left(\bigwedge^pT^*\mca{F}\right)$ is called a leafwise $p$-form of $\mca{F}$. We have the exterior derivative 
\begin{equation*}
d_\mca{F}\colon\Gamma\left(\bigwedge^pT^*\mca{F}\right)\to\Gamma\left(\bigwedge^{p+1}T^*\mca{F}\right)
\end{equation*}
for each $p$, defined by the formula 
\begin{align*}
\left(d_\mca{F}\omega\right)(X_1,\ldots,X_{p+1})&=\sum_{i=1}^{p+1}(-1)^{i+1}X_i\omega\left(X_1,\ldots,\widehat{X_i},\ldots,X_{p+1}\right)\\
&\quad+\sum_{i<j}(-1)^{i+j}\omega\left([X_i,X_j],X_1,\ldots,\widehat{X_i},\ldots,\widehat{X_j},\ldots,X_{p+1}\right)
\end{align*}
for $X_1,\ldots,X_{p+1}\in\Gamma(T\mca{F})$. Note that $[X_i,X_j]\in\Gamma(T\mca{F})$ because $\mca{F}$ is a foliation. Since we have $d_\mca{F}^2=0$, we can define the cohomology of the cochain complex $\left(\Gamma\left(\bigwedge^*T^*\mca{F}\right),d_\mca{F}\right)$. This cohomology is the de Rham cohomology of $\mca{F}$ (with the trivial coefficient $\bb{R}$) and denoted by $H^*(\mca{F})$ or $H^*(\mca{F};\bb{R})$. 

Next we will define the de Rham cohomology with coefficients. Let $\bb{F}$ be $\bb{R}$ or $\bb{C}$ and $E$ be a finite dimensional $C^\infty$ $\bb{F}$-vector bundle over $M$. An element of $\Gamma\left(\bigwedge^pT^*\mca{F}\otimes E\right)$ is called an $E$-valued leafwise $p$-form of $\mca{F}$. 

\begin{dfn}
An $\mca{F}$-partial connection of the vector bundle $E$ is an $\bb{F}$-linear map $\nabla\colon\Gamma(E)\to\Gamma(T^*\mca{F}\otimes E)$ such that the Leibniz rule 
\begin{equation*}
\nabla(f\eta)=d_\mca{F}f\otimes\eta+f\nabla\eta
\end{equation*}
holds for all $f\in C^\infty(M,\bb{F})$ and $\eta\in\Gamma(E)$. 
\end{dfn}

For an $\mca{F}$-partial connection $\nabla$ of $E$, the curvature 
\begin{equation*}
R\colon\Gamma\left(T\mca{F}\otimes\bb{F}\right)\times\Gamma\left(T\mca{F}\otimes\bb{F}\right)\times\Gamma(E)\to\Gamma(E)
\end{equation*}
of $\nabla$ is defined by 
\begin{equation*}
R(X,Y)\eta=\nabla_X\nabla_Y\eta-\nabla_Y\nabla_X\eta-\nabla_{[X,Y]}\eta
\end{equation*}
for $X,Y\in\Gamma\left(T\mca{F}\otimes\bb{F}\right)$ and $\eta\in\Gamma(E)$. We have 
\begin{equation*}
R\left(fX,Y\right)\eta=R\left(X,fY\right)\eta=R(X,Y)f\eta=fR(X,Y)\eta
\end{equation*}
and 
\begin{equation*}
R(X,Y)\eta=-R(Y,X)\eta
\end{equation*}
for $f\in C^\infty(M,\bb{F})$, $X,Y\in\Gamma\left(T\mca{F}\otimes\bb{F}\right)$ and $\eta\in\Gamma(E)$, hence $R$ can be regarded as an element of $\Gamma\left(\bigwedge^2T^*\mca{F}\otimes\End_\bb{F}(E)\right)$. An $\mca{F}$-partial connection $\nabla$ is said to be flat if $R=0$. 

For an $\mca{F}$-partial connection $\nabla$, we can define the covariant exterior derivative 
\begin{equation*}
d_\mca{F}^\nabla\colon\Gamma\left(\bigwedge^pT^*\mca{F}\otimes E\right)\to\Gamma\left(\bigwedge^{p+1}T^*\mca{F}\otimes E\right)
\end{equation*}
by 
\begin{align*}
\left(d_\mca{F}^\nabla\omega\right)(X_1,\ldots,X_{p+1})&=\sum_{i=1}^{p+1}(-1)^{i+1}\nabla_{X_i}\omega\left(X_1,\ldots,\widehat{X_i},\ldots,X_{p+1}\right)\\
&+\sum_{i<j}(-1)^{i+j}\omega\left([X_i,X_j],X_1,\ldots,\widehat{X_i},\ldots,\widehat{X_j},\ldots,X_{p+1}\right)
\end{align*}
for $\omega\in\Gamma\left(\bigwedge^pT^*\mca{F}\otimes E\right)$ and $X_1,\ldots,X_{p+1}\in\Gamma(T\mca{F})$. It follows that 
\begin{equation*}
d_\mca{F}^\nabla\left(\omega\otimes\eta\right)=d_\mca{F}\omega\otimes\eta+(-1)^p\omega\wedge\nabla\eta
\end{equation*}
for $\omega\in\Gamma\left(\bigwedge^pT^*\mca{F}\right)$ and $\eta\in\Gamma(E)$. We have $\left(d_\mca{F}^\nabla\right)^2=R\wedge$. Hence for a flat $\mca{F}$-partial connection $\nabla$, the pair $\left(\Gamma\left(\bigwedge^*T^*\mca{F}\otimes E\right),d_\mca{F}^\nabla\right)$ is a cochain complex and its cohomology is denoted by $H^*(\mca{F};E,\nabla)$. 

If $E$ is a real vector bundle with a flat $\mca{F}$-partial connection $\nabla$, then $E\otimes\bb{C}$ is a complex vector bundle with the $\mca{F}$-partial connection $\nabla\otimes\bb{C}$. Then $\nabla\otimes\bb{C}$ is flat and we have an isomorphism 
\begin{equation}\label{rrrccc}
\left(\Gamma\left(\bigwedge^*T^*\mca{F}\otimes E\otimes\bb{C}\right),d_\mca{F}^{\nabla\otimes\bb{C}}\right)\simeq\left(\Gamma\left(\bigwedge^*T^*\mca{F}\otimes E\right)\otimes\bb{C},d_\mca{F}^\nabla\otimes\bb{C}\right)
\end{equation}
of complex cochain complexes, hence 
\begin{equation*}
H^*\left(\mca{F};E\otimes\bb{C}\right)\simeq H^*(\mca{F};E)\otimes\bb{C}. 
\end{equation*}

\subsection{A construction of a flat $\mca{F}$-partial connection from a representation}\label{acofpcf}
Let $G$ be a connected Lie group with the Lie algebra $\mf{g}$, $M$ be a $C^\infty$ manifold, $M\stackrel{\rho}{\curvearrowleft}G$ be a $C^\infty$ locally free action with the orbit foliation $\mca{F}$. Let $\bb{F}$ be $\bb{R}$ or $\bb{C}$ and $\mf{g}\stackrel{\pi}{\curvearrowright}V$ be a representation on a finite dimensional vector space over $\bb{F}$, ie a real Lie algebra homomorphism $\pi\colon\mf{g}\to\End_\bb{F}(V)$. We define a flat $\mca{F}$-partial connection $\nabla^\pi$ of the trivial $\bb{F}$-vector bundle $M\times V\to M$. 

First we have a trivialization $T\mca{F}\simeq M\times\mf{g}$ defined by 
\begin{equation*}
\left.\frac{d}{dt}\rho\left(x,e^{tX}\right)\right|_{t=0}\mapsfrom(x,X). 
\end{equation*}
Let $\omega_0\in\Gamma\left(T^*\mca{F}\otimes\mf{g}\right)$ be the composition of the above isomorphism $T\mca{F}\simeq M\times\mf{g}$ and the projection $M\times\mf{g}\to\mf{g}$. The $\mf{g}$-valued leafwise $1$-form $\omega_0$ is called the canonical $1$-form of $\rho$. It satisfies the Maurer--Cartan equation 
\begin{equation}\label{Maurer}
\df\omega_0+\frac{1}{2}[\omega_0,\omega_0]=0, 
\end{equation}
where the $\mf{g}$-valued leafwise $2$-form $[\omega_0,\omega_0]\in\Gamma\left(\bigwedge^2T^*\mca{F}\otimes\mf{g}\right)$ is defined by 
\begin{equation*}
[\omega_0,\omega_0](X,Y)=[\omega_0(X),\omega_0(Y)]-[\omega_0(Y),\omega_0(X)]=2[\omega_0(X),\omega_0(Y)]
\end{equation*}
for $X$, $Y\in\Gamma(T\mca{F})$. 

By composing the representation $\pi\colon\mf{g}\to\End_\bb{F}(V)$, we get 
\begin{equation*}
\pi\omega_0\in\Gamma\left(T^*\mca{F}\otimes\End_\bb{F}(V)\right), 
\end{equation*}
which satisfies 
\begin{equation}\label{curvature}
\df(\pi\omega_0)+\frac{1}{2}\left[\pi\omega_0,\pi\omega_0\right]=0
\end{equation}
by \eqref{Maurer}. 

We define an $\mca{F}$-partial connection $\nabla^\pi\colon\Gamma(V)\to\Gamma\left(T^*\mca{F}\otimes V\right)$ of the trivial $\bb{F}$-vector bundle $M\times V\to M$ by 
\begin{equation*}
\nabla^\pi\xi=\df\xi+(\pi\omega_0)\xi. 
\end{equation*}

Let $v_1,\ldots,v_n$ be an $\bb{F}$-basis of $V$. Regard $v_1,\ldots,v_n$ as a global frame of the trivial bundle $M\times V\to M$. Then the connection form of $\nabla^\pi$ with respect to the global frame is the matrix representation of $\pi\omega_0$ with respect to $v_1,\ldots,v_n$. Since the curvature form of $\nabla^\pi$ with respect to the global frame is the matrix representation of $\df(\pi\omega_0)+\frac{1}{2}\left[\pi\omega_0,\pi\omega_0\right]$ with respect to $v_1,\ldots,v_n$, it vanishes by \eqref{curvature}. Hence the $\mca{F}$-partial connection $\nabla^\pi$ is flat. 

\begin{dfn}
We have a de Rham complex $\left(\Gamma\left(\bigwedge^*T^*\mca{F}\otimes V\right),d_\mca{F}^{\nabla^\pi}\right)$ and its cohomology is denoted by $H^*(\mca{F};V)$, $H^*(\mca{F};\pi)$ or $H^*\left(\mca{F};\mf{g}\stackrel{\pi}{\curvearrowright}V\right)$. 
\end{dfn}

If $\mf{g}\stackrel{\pi}{\curvearrowright}V$ is a real representation, then $\mf{g}\stackrel{\pi\otimes\bb{C}}{\curvearrowright}V\otimes\bb{C}$ is a complex representation and $\nabla^\pi\otimes\bb{C}=\nabla^{\pi\otimes\bb{C}}$. So we have an isomorphism 
\begin{equation*}
\left(\Gamma\left(\bigwedge^*T^*\mca{F}\otimes V\otimes\bb{C}\right),d_\mca{F}^{\nabla^{\pi\otimes\bb{C}}}\right)\simeq\left(\Gamma\left(\bigwedge^*T^*\mca{F}\otimes V\right)\otimes\bb{C},d_\mca{F}^{\nabla^\pi}\otimes\bb{C}\right)
\end{equation*}
of complex cochain complexes by \eqref{rrrccc}, hence 
\begin{equation}\label{realcomplex}
H^*\left(\mca{F};\mf{g}\stackrel{\pi\otimes\bb{C}}{\curvearrowright}V\otimes\bb{C}\right)\simeq H^*\left(\mca{F};\mf{g}\stackrel{\pi}{\curvearrowright}V\right)\otimes\bb{C}. 
\end{equation}

\subsection{An isomorphism between the de Rham cohomology and a Lie algebra cohomology}
Let $G$ be a connected Lie group with the Lie algebra $\mf{g}$, $M$ be a $C^\infty$ manifold, $M\stackrel{\rho}{\curvearrowleft}G$ be a $C^\infty$ locally free action with the orbit foliation $\mca{F}$. Let $\bb{F}$ be $\bb{R}$ or $\bb{C}$ and $\mf{g}\stackrel{\pi}{\curvearrowright}V$ be a representation on a finite dimensional vector space over $\bb{F}$. In the previous section we defined a flat $\mca{F}$-partial connection $\nabla^\pi$ of the trivial vector bundle $M\times V\to M$. Hence we have a de Rham complex $\left(\Gamma\left(\bigwedge^*T^*\mca{F}\otimes V\right),d_\mca{F}^{\nabla^\pi}\right)$. 

On the other hand, we have a representation $\mf{g}\curvearrowright C^\infty(M,\bb{R})$ defined by 
\begin{equation*}
(Xf)(x)=\left.\frac{d}{dt}f\left(\rho\left(x,e^{tX}\right)\right)\right|_{t=0}
\end{equation*}
for $X\in\mf{g}$, $f\in C^\infty(M,\bb{R})$ and $x\in M$. Hence we can form the tensor product $\mf{g}\curvearrowright C^\infty(M,\bb{R})\otimes V$ of the representations and consider the Chevalley--Eilenberg complex $\left(C^*\left(\mf{g};C^\infty(M,\bb{R})\otimes V\right),d\right)$. 

\begin{prop}\label{drce}
We have an isomorphism 
\begin{equation*}
\left(\Gamma\left(\bigwedge^*T^*\mca{F}\otimes V\right),d_\mca{F}^{\nabla^\pi}\right)\simeq\left(C^*\left(\mf{g};C^\infty(M,\bb{R})\otimes V\right),d\right)
\end{equation*}
of cochain complexes. Hence
\begin{equation*}
H^*(\mca{F};V)\simeq H^*\left(\mf{g};C^\infty(M,\bb{R})\otimes V\right). 
\end{equation*}
\end{prop}

\begin{proof}
Define a representation $\mf{g}\curvearrowright C^\infty(M,V)$ over $\bb{F}$ by 
\begin{equation*}
(X\xi)(x)=\left.\frac{d}{dt}\xi\left(\rho\left(x,e^{tX}\right)\right)\right|_{t=0}+\pi(X)\xi(x)
\end{equation*}
for $X\in\mf{g}$, $\xi\in C^\infty(M,V)$ and $x\in M$. We have an isomorphism 
\begin{align*}
C^\infty(M,\bb{R})\otimes V&\simeq C^\infty(M,V)\\
f\otimes v&\mapsto fv
\end{align*}
of $\mf{g}$-modules. So it suffices to construct an isomorphism 
\begin{equation*}
\Phi\colon\Gamma\left(\bigwedge^*T^*\mca{F}\otimes V\right)\simeq C^*\left(\mf{g};C^\infty(M,V)\right). 
\end{equation*}
First note that the isomorphism $T\mca{F}\simeq M\times\mf{g}$ induces an isomorphism 
\begin{equation*}
\bigwedge^pT^*\mca{F}\otimes V\simeq M\times\left(\bigwedge^p\mf{g}^*\otimes V\right)
\end{equation*}
of $\bb{F}$-vector bundles. Then $\Phi$ is defined to be the following composition: 
\begin{align*}
\Gamma\left(\bigwedge^pT^*\mca{F}\otimes V\right)&\simeq C^\infty\left( M,\bigwedge^p\mf{g}^*\otimes V\right)\simeq\bigwedge^p\mf{g}^*\otimes C^\infty(M,V)\\
&=C^p\left(\mf{g};C^\infty(M,V)\right). 
\end{align*}
Observe that we have
\begin{equation*}
(\Phi\omega)\left(X_1,\ldots,X_p\right)=\omega\left(\omega_0^{-1}(X_1),\ldots,\omega_0^{-1}(X_p)\right)
\end{equation*}
for $\omega\in\Gamma\left(\bigwedge^pT^*\mca{F}\otimes V\right)$ and $X_1,\ldots,X_p\in\mf{g}$. 

To show that $\Phi$ is a cochain map, first note that 
\begin{align*}
\nabla^\pi_{\omega_0^{-1}(X)}\xi&=\left(\df\xi+(\pi\omega_0)\xi\right)\left(\omega_0^{-1}(X)\right)\\
&=\omega_0^{-1}(X)\xi+\pi(X)\xi\\
&=X\xi
\end{align*}
for $X\in\mf{g}$ and $\xi\in C^\infty(M,V)$. Take $\omega\in\Gamma\left(\bigwedge^pT^*\mca{F}\otimes V\right)$ and $X_1,\ldots,X_{p+1}\in\mf{g}$. Then we have 
\begin{align*}
&\left(\Phi d_\mca{F}^{\nabla^\pi}\omega\right)\left(X_1,\ldots,X_{p+1}\right)\\
&=\left(d_\mca{F}^{\nabla^\pi}\omega\right)\left(\omega_0^{-1}(X_1),\ldots,\omega_0^{-1}(X_{p+1})\right)\\
&=\sum_{i=1}^{p+1}(-1)^{i+1}\nabla^\pi_{\omega_0^{-1}(X_i)}\omega\left(\omega_0^{-1}(X_1),\ldots,\widehat{\omega_0^{-1}(X_i)},\ldots,\omega_0^{-1}(X_{p+1})\right)\\
&\quad+\sum_{i<j}(-1)^{i+j}\omega\left(\left[\omega_0^{-1}(X_i),\omega_0^{-1}(X_j)\right],\omega_0^{-1}(X_1),\ldots,\widehat{\omega_0^{-1}(X_i)},\right.\\
&\qquad\qquad\qquad\qquad\qquad\qquad\qquad\quad\left.\ldots,\widehat{\omega_0^{-1}(X_j)},\ldots,\omega_0^{-1}(X_{p+1})\right)\\
&=\sum_{i=1}^{p+1}(-1)^{i+1}X_i\omega\left(\omega_0^{-1}(X_1),\ldots,\widehat{\omega_0^{-1}(X_i)},\ldots,\omega_0^{-1}(X_{p+1})\right)\\
&\quad+\sum_{i<j}(-1)^{i+j}\omega\left(\omega_0^{-1}\left([X_i,X_j]\right),\omega_0^{-1}(X_1),\ldots,\widehat{\omega_0^{-1}(X_i)},\right.\\
&\qquad\qquad\qquad\qquad\qquad\qquad\qquad\quad\left.\ldots,\widehat{\omega_0^{-1}(X_j)},\ldots,\omega_0^{-1}(X_{p+1})\right)\\
&=\sum_{i=1}^{p+1}(-1)^{i+1}X_i(\Phi\omega)\left(X_1,\ldots,\widehat{X_i},\ldots,X_{p+1}\right)\\
&\quad+\sum_{i<j}(-1)^{i+j}(\Phi\omega)\left([X_i,X_j],X_1,\ldots,\widehat{X_i},\ldots,\widehat{X_j},\ldots,X_{p+1}\right)\\
&=\left(d\Phi\omega\right)\left(X_1,\ldots,X_{p+1}\right). 
\end{align*}
\end{proof}

\section{Notations around $\mf{sl}(2,\bb{R})$}\label{notation}
In this paper we will use the following basis of $\mf{sl}(2,\bb{R})$: 
\begin{equation*}
F=\frac{1}{2}
\begin{pmatrix}
0&0\\
1&0
\end{pmatrix}
,\quad H=\frac{1}{2\sqrt{2}}
\begin{pmatrix}
1&0\\
0&-1
\end{pmatrix}
,\quad E=\frac{1}{2}
\begin{pmatrix}
0&1\\
0&0
\end{pmatrix}
. 
\end{equation*}
We use this basis because it is an orthonormal basis with respect to an inner product $B_\theta$ described below. The commutation relations are 
\begin{equation*}
[H,E]=\frac{1}{\sqrt{2}}E,\quad[H,F]=-\frac{1}{\sqrt{2}}F,\quad[E,F]=\frac{1}{\sqrt{2}}H. 
\end{equation*}
Let $B\colon\mf{sl}(2,\bb{R})\times\mf{sl}(2,\bb{R})\to\bb{R}$ be the Killing form of $\mf{sl}(2,\bb{R})$. This is represented by 
\begin{equation*}
B=
\begin{pmatrix}
0&0&1\\
0&1&0\\
1&0&0
\end{pmatrix}
\end{equation*}
with respect to the basis $F$, $H$, $E$. 

Let $U(\mf{sl}(2,\bb{R}))$ be the universal enveloping algebra of $\mf{sl}(2,\bb{R})$. (This is an algebra over $\bb{R}$.) Let $\Omega\in U(\mf{sl}(2,\bb{R}))$ be the Casimir element of $\mf{sl}(2,\bb{R})$. Since the dual basis of $F$, $H$, $E$ with respect to $B$ is $E$, $H$, $F$, we have 
\begin{align*}
\Omega&=FE+H^2+EF\\
&=2FE+H^2+\frac{1}{\sqrt{2}}H\\
&=2EF+H^2-\frac{1}{\sqrt{2}}H. 
\end{align*}

\begin{rem}
Many authors take $-2\Omega$ as the Casimir element when they deal with $\PSL(2,\bb{R})$. This choice corresponds to the the nondegenerate invariant symmetric bilinear form 
\begin{align*}
B^\prime\colon\sl(2,\bb{R})\times\sl(2,\bb{R})&\to\bb{R}\\
(X,Y)&\mapsto-2\tr(XY). 
\end{align*}
In fact we have $B^\prime=-\frac{1}{2}B$, hence the corresponding Casimir element is $-2\Omega$. 
\end{rem}

Let $\theta\colon\mf{sl}(2,\bb{R})\to\mf{sl}(2,\bb{R})$ be the Cartan involution of $\mf{sl}(2,\bb{R})$ defined by $\theta(X)=-X^\top$. By definition the form $B_\theta\colon\sl(2,\bb{R})\times\sl(2,\bb{R})\to\bb{R}$ defined by 
\begin{equation*}
B_\theta(X,Y)=-B(X,\theta Y)=-B(\theta X,Y)
\end{equation*}
is an inner product of $\sl(2,\bb{R})$. Since 
\begin{equation*}
\theta H=-H,\quad\theta E=-F,\quad\theta F=-E, 
\end{equation*}
the inner product $B_\theta$ is represented by 
\begin{equation*}
B_\theta=
\begin{pmatrix}
1&&\\
&1&\\
&&1
\end{pmatrix}
\end{equation*}
with respect to $F$, $H$, $E$. Hence the basis $F$, $H$, $E$ is an orthonormal basis of $\sl(2,\bb{R})$ with respect to $B_\theta$.

\section{De Rham cohomology of the weak stable foliation of the geodesic flow of a hyperbolic surface}\label{drws}
Let $\Gamma$ be a torsion free cocompact lattice of $\PSL(2,\bb{R})$ and put $P=\Gamma\backslash\PSL(2,\bb{R})$ and $\Sigma=\Gamma\backslash\PSL(2,\bb{R})/\PSO(2)$. The manifold $P$ is naturally identified with the unit tangent bundle of the hyperbolic surface $\Sigma$. (Note that we can obtain any connected oriented closed hyperbolic surface in this way.) Consider the subgroups 
\begin{gather*}
A=\left\{\pm
\begin{pmatrix}
a&\\
&a^{-1}
\end{pmatrix}
\ \middle|\ a>0\right\}\quad\text{and}\quad N=\left\{\pm
\begin{pmatrix}
1&b\\
&1
\end{pmatrix}
\ \middle|\ b\in\bb{R}\right\}
\end{gather*}
of $\PSL(2,\bb{R})$. Then 
\begin{equation*}
AN=\left\{\pm
\begin{pmatrix}
a&b\\
&a^{-1}
\end{pmatrix}
\ \middle|\ a>0,\ b\in\bb{R}\right\}
\end{equation*}
is a $2$-dimensional solvable Lie group and $\PSL(2,\bb{R})=\PSO(2)AN$ is an Iwasawa decomposition of $\PSL(2,\bb{R})$. Let $\mca{F}$ be the orbit foliation of the action $P\curvearrowleft AN$ by right multiplication. The flow 
\begin{align*}
P\times\bb{R}&\to P\\
(x,t)&\mapsto x
\begin{pmatrix}
e^\frac{t}{2}&\\
&e^{-\frac{t}{2}}
\end{pmatrix}
\end{align*}
is identified with the geodesic flow of $\Sigma$, which is an Anosov flow, and $\mca{F}$ is its weak stable foliation. The flow 
\begin{align*}
P\times\bb{R}&\to P\\
(x,t)&\mapsto x
\begin{pmatrix}
1&t\\
&1
\end{pmatrix}
\end{align*}
is identified with (one of) the (two) horocycle flow(s) of $\Sigma$. 

Let $\mf{a}$, $\mf{n}$ and $\mf{an}$ be the Lie algebras of $A$, $N$ and $AN$. Recall the elements $H$ and $E$ defined in Section \ref{notation}. We have 
\begin{gather*}
\mf{a}=\bb{R}H,\quad\mf{n}=\bb{R}E,\quad\mf{an}=\bb{R}H\oplus\bb{R}E, 
\end{gather*}
and $[H,E]=\frac{1}{\sqrt{2}}E$. 

If $\mf{an}\stackrel{\pi}{\curvearrowright}V$ is a finite dimensional representation on a real vector space, then we have 
\begin{equation*}
H^*\left(\mca{F};\mf{an}\stackrel{\pi}{\curvearrowright}V\right)\otimes\bb{C}\simeq H^*\left(\mca{F};\mf{an}\stackrel{\pi\otimes\bb{C}}{\curvearrowright}V\otimes\bb{C}\right)
\end{equation*}
by \eqref{realcomplex} in Section \ref{acofpcf}. Hence the computation of $H^*(\mca{F};\pi)$ reduces to the computation of $H^*\left(\mca{F};\pi\otimes\bb{C}\right)$. 

If $\mf{an}\stackrel{\pi}{\curvearrowright}V$ is a finite dimensional representation on a complex vector space, then we will use the isomorphism 
\begin{equation*}
H^*\left(\mca{F};\mf{an}\stackrel{\pi}{\curvearrowright}V\right)\simeq H^*\left(\mf{an};C^\infty(P,\bb{C})\otimes_\bb{C}V\right)
\end{equation*}
in Proposition \ref{drce} to compute $H^*(\mca{F};\pi)$. 

In this paper we consider the following representations of $\mf{an}$ as coefficients of the de Rham cohomology of $\mca{F}$. For any $\lambda\in\bb{R}$ set $\bb{R}_\lambda=\bb{R}$ and define a representation $\mf{an}\curvearrowright\bb{R}_\lambda$ by 
\begin{gather*}
H1=\lambda1\quad\text{and}\quad E1=0. 
\end{gather*}
Similarly, for any $\lambda\in\bb{C}$ set $\bb{C}_\lambda=\bb{C}$ and define $\mf{an}\curvearrowright\bb{C}_\lambda$ by 
\begin{gather*}
H1=\lambda1\quad\text{and}\quad E1=0. 
\end{gather*}
Since $\mf{n}=\left[\mf{an},\mf{an}\right]$, these are the only $1$-dimensional real or complex representations of $\mf{an}$. Note that $\bb{R}_0$ and $\bb{C}_0$ are the trivial representations $\bb{R}$ and $\bb{C}$, and we have 
\begin{gather*}
H^*\left(\mca{F};\bb{R}_0\right)=H^*(\mca{F};\bb{R})\quad\text{and}\quad H^*\left(\mca{F};\bb{C}_0\right)=H^*(\mca{F};\bb{C}). 
\end{gather*}
We also consider the adjoint representations $\mf{an}\stackrel{\ad}{\curvearrowright}\mf{n}$, $\mf{an}\stackrel{\ad}{\curvearrowright}\mf{an}$ and $\mf{an}\stackrel{\ad}{\curvearrowright}\sl(2,\bb{R})$. We have short exact sequences 
\begin{equation*}
0\to\mf{n}\to\mf{an}\to\mf{an}/\mf{n}\to0
\end{equation*}
and 
\begin{equation*}
0\to\mf{an}\to\sl(2,\bb{R})\to\sl(2,\bb{R})/\mf{an}\to0. 
\end{equation*}
Note that 
\begin{gather*}
\mf{n}\simeq\bb{R}_\frac{1}{\sqrt{2}},\quad\mf{an}/\mf{n}\simeq\bb{R}_0\quad\text{and}\quad\sl(2,\bb{R})/\mf{an}\simeq\bb{R}_{-\frac{1}{\sqrt{2}}}. 
\end{gather*}

\section{Preparation from the unitary representation theory}
Here we will review the unitary representation theory of $\PSL(2,\bb{R})$ we need in this paper. It includes the classification of irreducible unitary representations, the decomposition of $L^2\left(\Gamma\backslash\PSL(2,\bb{R})\right)$ into irreducible representations and a description of $C^\infty$ vectors of $L^2\left(\Gamma\backslash\PSL(2,\bb{R})\right)$. All these things are standard but we think it is useful to gather these results here because it was not easy for us to learn these things through the existing literature.

\subsection{Generalities}\label{gener}
For $n\in\bb{Z}$, let $\PSO(2)\stackrel{r_n}{\curvearrowright}\bb{C}$ be the irreducible unitary representation defined by 
\begin{equation*}
\exp\theta
\begin{pmatrix}
&\frac{1}{2}\\
-\frac{1}{2}&
\end{pmatrix}
=\pm
\begin{pmatrix}
\cos\frac{\theta}{2}&\sin\frac{\theta}{2}\\
-\sin\frac{\theta}{2}&\cos\frac{\theta}{2}
\end{pmatrix}
\stackrel{r_n}{\mapsto}e^{in\theta}. 
\end{equation*}
Note that the element $\pm
\begin{pmatrix}
\cos\frac{\theta}{2}&\sin\frac{\theta}{2}\\
-\sin\frac{\theta}{2}&\cos\frac{\theta}{2}
\end{pmatrix}
$ acts, as a M\"{o}bius transformation on the upper half plane $\bb{H}^2$, as the rotation by $\theta$ about $i$. This follows from the formula $\left(\frac{az+b}{cz+d}\right)^\prime=\frac{1}{(cz+d)^2}$ for $
\begin{pmatrix}
a&b\\
c&d
\end{pmatrix}
\in\SL(2,\bb{R})$. Let $\so(2)\stackrel{r_n}{\curvearrowright}\bb{C}$ denote the corresponding representation of the Lie algebra. Then we have $r_n
\begin{pmatrix}
&\frac{1}{2}\\
-\frac{1}{2}&
\end{pmatrix}
=in$, hence if we let 
\begin{equation*}
H_0=i
\begin{pmatrix}
&-\frac{1}{2}\\
\frac{1}{2}&
\end{pmatrix}
\in\so(2)\otimes\bb{C}, 
\end{equation*}
we have $r_n(H_0)=n$. 

Let $\PSL(2,\bb{R})\stackrel{\pi}{\curvearrowright}\mca{H}$ be a unitary representation on a complex Hilbert space $\mca{H}$. The restriction $\PSO(2)\curvearrowright\mca{H}$ of $\pi$ decomposes as 
\begin{equation}\label{macom}
\mca{H}=\bigoplus_{n\in\bb{Z}}\mca{H}_{(n)}, 
\end{equation}
where $\mca{H}_{(n)}$ is the isotypic component of $\mca{H}$ corresponding to $\PSO(2)\stackrel{r_n}{\curvearrowright}\bb{C}$. (See 1.4.7 of Wallach \cite{Wallach1}.)

An element $v\in\mca{H}$ is called a $C^\infty$ vector of $\mca{H}$ if the map 
\begin{equation*}
\PSL(2,\bb{R})\to\mca{H},\quad g\mapsto\pi(g)v
\end{equation*}
is a $C^\infty$ map. Let $C^\infty(\mca{H})$ be the space of $C^\infty$ vectors in $\mca{H}$. The space $C^\infty(\mca{H})$ is a dense subspace of $\mca{H}$. See 1.6.2 of Wallach \cite{Wallach1} or Proposition 4.4.1.1 of Warner \cite{Warner1}. We have an injective map 
\begin{align*}
C^\infty(\mca{H})&\hookrightarrow C^\infty(\PSL(2,\bb{R}),\mca{H})\\
v&\mapsto(g\mapsto\pi(g)v). 
\end{align*}
The space $C^\infty(\PSL(2,\bb{R}),\mca{H})$ is a Fr\'{e}chet space with respect to the topology of uniform convergence on compact subsets of a function and each of its derivatives. See Example 1 in 2.2 of Appendices of Warner \cite{Warner1}. The image of $C^\infty(\mca{H})$ in $C^\infty(\PSL(2,\bb{R}),\mca{H})$ is closed and we equip $C^\infty(\mca{H})$ with the relative topology, which makes it a Fr\'{e}chet space. 

Let $\Gamma$ be a cocompact lattice of $\PSL(2,\bb{R})$ and $P=\Gamma\backslash\PSL(2,\bb{R})$. Let $L^2(P)$ be the $L^2$ space of $\bb{C}$-valued functions on $P$ with respect to the unique $\PSL(2,\bb{R})$-invariant probability measure of $P$. Then we have a unitary representation $\PSL(2,\bb{R})\curvearrowright L^2(P)$ arising from the action $P\curvearrowleft\PSL(2,\bb{R})$ by right multiplication. 

\begin{prop}\label{cinfinityvect}
We have 
\begin{equation*}
C^\infty(L^2(P))=C^\infty(P,\bb{C}). 
\end{equation*}
\end{prop}

\begin{proof}
See Chapter III. 7. of Borel--Wallach \cite{BW}. 
\end{proof}

\subsection{Irreducible unitary representations of $\PSL(2,\bb{R})$}\label{iurops}
Let $\PSL(2,\bb{R})\curvearrowright\mca{H}$ be an irreducible unitary representation on a complex Hilbert space $\mca{H}$. The Casimir element $\Omega$ acts as a scalar on $C^\infty(\mca{H})$. See Proposition in 1.2.2 of Wallach \cite{Wallach1} or Proposition 4.4.1.4 of Warner \cite{Warner1}. We call the value the Casimir parameter of $\mca{H}$. The following table is a complete list of irreducible unitary representations of $\PSL(2,\bb{R})$ with their Casimir parameters. Any two of them are not unitarily equivalent to each other. 

\renewcommand{\arraystretch}{1.3}
\begin{table}[htbp]
\centering
\begin{tabular}{rlcc}
&&{\scriptsize\shortstack{Range of\\parameter}}&{\scriptsize\shortstack{Casimir\\parameter}}\\\hline
The trivial representation&$\bb{C}$&&$0$\\
Holomorphic discrete series representations&$D_n^+$&$n\in\bb{Z}_{\geq1}$&$\frac{n(n-1)}{2}$\\
Antiholomorphic discrete series representations&$D_n^-$&$n\in\bb{Z}_{\geq1}$&$\frac{n(n-1)}{2}$\\
Principal series representations&$\mca{H}_\mu$&$-\infty<\mu\leq-\frac{1}{8}$&$\mu$\\
Complementary series representations&$\mca{H}_\mu$&$-\frac{1}{8}<\mu<0$&$\mu$\\\hline
\end{tabular}
\end{table}
\renewcommand{\arraystretch}{1.0}

\begin{rem}
\begin{itemize}
\item There are many literatures dealing with the classification of irreducible unitary representations of $\SL(2,\bb{R})$. We can obtain the corresponding results for $\PSL(2,\bb{R})$ by identifying representations for which $-I\in\SL(2,\bb{R})$ acts as the identity. 
\item The set of possible Casimir parameter is 
\begin{equation*}
(-\infty,0]\cup\left\{\frac{n(n-1)}{2}\middle|n\in\bb{Z}_{\geq2}\right\}. 
\end{equation*}
Note that many authors use $-2\Omega$ as the Casimir element. 
\item Let $\mca{R}_\mu$ be the set of the unitary equivalence classes of the irreducible unitary representations of $\PSL(2,\bb{R})$ with Casimir parameter $\mu$. Then we have 
\begin{equation*}
\mca{R}_\mu=
\begin{cases}
\{\mca{H}_\mu\}&\mu<0\\
\{\bb{C},D_1^+,D_1^-\}&\mu=0\\
\{D_n^+,D_n^-\}&\mu=\frac{n(n-1)}{2}\ \ \text{for $n\in\bb{Z}_{\geq2}$}\\
\varnothing&\text{otherwise. }
\end{cases}
\end{equation*}
\end{itemize}
\end{rem}

The dimensions of $\PSO(2)$-isotypic components are as follows: 
\begin{equation*}
\dim\bb{C}_{(i)}=
\begin{cases}
1&i=0\\
0&i\neq0, 
\end{cases}
\qquad\dim(\mca{H}_\mu)_{(i)}=1, 
\end{equation*}
\begin{equation}\label{dimmm}
\dim(D_n^+)_{(i)}=
\begin{cases}
1&i\geq n\\
0&i<n, 
\end{cases}
\qquad\dim(D_n^-)_{(i)}=
\begin{cases}
1&i\leq-n\\
0&i>-n. 
\end{cases}
\end{equation}

\subsection{Decomposition of $L^2(\Gamma\backslash\PSL(2,\bb{R}))$}\label{decompl2}
Let $\Gamma$ be a torsion free cocompact lattice of $\PSL(2,\bb{R})$ and put $P=\Gamma\backslash\PSL(2,\bb{R})$. Let $L^2(P)$ be the $L^2$ space of $\bb{C}$-valued functions on $P$ with respect to the $\PSL(2,\bb{R})$-invariant probability measure of $P$. Then we get a unitary representation $\PSL(2,\bb{R})\stackrel{\pi}{\curvearrowright}L^2(P)$. Since $P$ is compact, $L^2(P)$ is an orthogonal direct sum of irreducible unitary representations each appearing with finite multiplicity. See Example following Proposition 4.3.1.8 of Warner \cite{Warner1} (or use 1.4.1 of Wallach \cite{Wallach1}). Let $W_\mu\subset L^2(P)$ be the sum of the irreducible unitary subrepresentations of $L^2(P)$ with Casimir parameter $\mu$. By the classification of irreducible unitary representations of $\PSL(2,\bb{R})$, 
\begin{equation}\label{decompooo}
L^2(P)=\bigoplus_{\mu<0}W_\mu\oplus W_0\oplus\bigoplus_{n\in\bb{Z}_{\geq2}}W_{\frac{n(n-1)}{2}}. 
\end{equation}
Here we have 
\begin{gather*}
W_0\simeq\bb{C}\oplus*D_1^+\oplus*D_1^-,\quad W_{\frac{n(n-1)}{2}}\simeq*D_n^+\oplus*D_n^-\quad(n\geq2),\\
W_\mu\simeq*\mca{H}_\mu\quad(\mu<0)
\end{gather*}
for some integers $*$. We will determine these integers.

\subsubsection{Contribution of principal and complementary series}
Let $\bb{H}^2$ be the upper half plane. We have a natural action $\PSL(2,\bb{R})\curvearrowright\bb{H}^2$ by M\"{o}bius transformations and an identification $\bb{H}^2=\PSL(2,\bb{R})/\PSO(2)$ given by 
\begin{align*}
\PSL(2,\bb{R})/\PSO(2)&\simeq\bb{H}^2\\
g\PSO(2)&\mapsto gi. 
\end{align*}
Put $\Sigma=\Gamma\backslash\PSL(2,\bb{R})/\PSO(2)$. Then $\Sigma$ is a closed oriented connected hyperbolic surface. Let $p$, $q$ be the following projections: 
\begin{equation*}
\begin{tikzcd}
\PSL(2,\bb{R})\ar[r,"q"]\ar[d,"p"]&P\ar[d,"p"]\\
\bb{H}^2\ar[r,"q"]&\Sigma. 
\end{tikzcd}
\end{equation*}
Let $\Delta_{\bb{H}^2}\colon C^\infty(\bb{H}^2,\bb{C})\to C^\infty(\bb{H}^2,\bb{C})$ be the Laplacian of $\bb{H}^2$. We have 
\begin{equation*}
\Delta_{\bb{H}^2}=-y^2\left(\frac{\partial^2}{\partial x^2}+\frac{\partial^2}{\partial y^2}\right). 
\end{equation*}
Let $\PSL(2,\bb{R})\stackrel{\pi^\prime}{\curvearrowright}C^\infty(\PSL(2,\bb{R}),\bb{C})$ be the representation defined by 
\begin{equation*}
(\pi^\prime(g)f)(g^\prime)=f(g^\prime g). 
\end{equation*}
Let 
\begin{equation*}
H_0=i
\begin{pmatrix}
&-\frac{1}{2}\\
\frac{1}{2}&
\end{pmatrix}
\in\so(2)\otimes\bb{C}. 
\end{equation*}
For a given $f\in C^\infty(\PSL(2,\bb{R}),\bb{C})$, note that $f\left(g\exp\theta
\begin{pmatrix}
&\frac{1}{2}\\
-\frac{1}{2}&
\end{pmatrix}
\right)=f(g)$ for all $g\in\PSL(2,\bb{R})$ and $\theta\in\bb{R}$ if and only if $\pi^\prime
\begin{pmatrix}
&\frac{1}{2}\\
-\frac{1}{2}&
\end{pmatrix}
f=0$. Hence we have 
\begin{equation}\label{hcpsl}
p^*\colon C^\infty(\bb{H}^2,\bb{C})\stackrel{\sim}{\to}\left\{f\in C^\infty(\PSL(2,\bb{R}),\bb{C})\ \middle|\ \pi^\prime(H_0)f=0\right\}, 
\end{equation}
where $(p^*h)(g)=h(gi)$. There are actions of $\PSL(2,\bb{R})$, one on $C^\infty(\bb{H}^2,\bb{C})$ defined by $(gh)(z)=h(g^{-1}z)$ and another one on $\left\{f\in C^\infty(\PSL(2,\bb{R}),\bb{C})\ \middle|\ \pi^\prime(H_0)f=0\right\}$ defined by $(gf)(g^\prime)=f(g^{-1}g^\prime)$, with respect to which $p^*$ is $\PSL(2,\bb{R})$-equivariant. Note that $\pi^\prime(\Omega)$ preserves $\left\{f\in C^\infty(\PSL(2,\bb{R}),\bb{C})\ \middle|\ \pi^\prime(H_0)f=0\right\}$ since $\Omega$ commutes with $H_0$. 

\begin{lem}\label{omegalaplacian}
We have 
\begin{equation*}
(p^*)^{-1}\pi^\prime(\Omega)p^*=-\frac{1}{2}\Delta_{\bb{H}^2}. 
\end{equation*}
\end{lem}

\begin{proof}
Let 
\begin{equation*}
n_x=\pm
\begin{pmatrix}
1&x\\
&1
\end{pmatrix}
,\quad a_y=\pm
\begin{pmatrix}
\sqrt{y}&\\
&\frac{1}{\sqrt{y}}
\end{pmatrix}
\in\PSL(2,\bb{R})
\end{equation*}
for $x\in\bb{R}$ and $y>0$. Then a point $z=x+iy$ of $\bb{H}^2$ is written as $z=n_xa_yi$. Note that $a_yn_x=n_{xy}a_y$. Since 
\begin{align*}
\Omega&=H^2-\frac{1}{\sqrt{2}}H+2EF\\
&=H^2-\frac{1}{\sqrt{2}}H+2E^2-2E(E-F), 
\end{align*}
first we have, for $h\in C^\infty(\bb{H}^2,\bb{C})$, 
\begin{align*}
\left((p^*)^{-1}\pi^\prime(\Omega)p^*h\right)(z)&=\left(\pi^\prime(\Omega)p^*h\right)(n_xa_y)\\
&=\left(\left(\pi^\prime(H)^2-\frac{1}{\sqrt{2}}\pi^\prime(H)+2\pi^\prime(E)^2\right)p^*h\right)(n_xa_y). 
\end{align*}
To compute this, we have, for $f\in C^\infty(\PSL(2,\bb{R}),\bb{C})$, 
\begin{equation*}
\left(\pi^\prime(H)f\right)(n_xa_y)=\left.\frac{d}{dt}f\left(n_xa_ya_{\exp\frac{t}{\sqrt{2}}}\right)\right|_{t=0}=\left.\frac{d}{dt}f\left(n_xa_{y\exp\frac{t}{\sqrt{2}}}\right)\right|_{t=0}
\end{equation*}
and 
\begin{align*}
\left(\pi^\prime(E)f\right)(n_xa_y)&=\left.\frac{d}{dt}f\left(n_xa_yn_\frac{t}{2}\right)\right|_{t=0}=\left.\frac{d}{dt}f\left(n_xn_\frac{ty}{2}a_y\right)\right|_{t=0}\\
&=\left.\frac{d}{dt}f\left(n_{x+\frac{ty}{2}}a_y\right)\right|_{t=0}. 
\end{align*}
Hence 
\begin{align*}
&\left((p^*)^{-1}\pi^\prime(\Omega)p^*h\right)(z)\\
&=\left.\left.\frac{d}{dt}\frac{d}{ds}p^*h\left(n_xa_{y\exp\frac{t+s}{\sqrt{2}}}\right)\right|_{s=0}\right|_{t=0}-\frac{1}{\sqrt{2}}\left.\frac{d}{dt}p^*h\left(n_xa_{y\exp\frac{t}{\sqrt{2}}}\right)\right|_{t=0}\\
&\quad+2\left.\left.\frac{d}{dt}\frac{d}{ds}p^*h\left(n_{x+\frac{(t+s)y}{2}}a_y\right)\right|_{s=0}\right|_{t=0}\\
&=\left.\left.\frac{d}{dt}\frac{d}{ds}h\left(x+iye^\frac{t+s}{\sqrt{2}}\right)\right|_{s=0}\right|_{t=0}-\frac{1}{\sqrt{2}}\left.\frac{d}{dt}h\left(x+iye^\frac{t}{\sqrt{2}}\right)\right|_{t=0}\\
&\quad+2\left.\left.\frac{d}{dt}\frac{d}{ds}h\left(x+\frac{(t+s)y}{2}+iy\right)\right|_{s=0}\right|_{t=0}\\
&=\left.\frac{d}{dt}\frac{1}{\sqrt{2}}ye^\frac{t}{\sqrt{2}}\frac{\partial h}{\partial y}\left(x+iye^\frac{t}{\sqrt{2}}\right)\right|_{t=0}-\frac{1}{2}y\frac{\partial h}{\partial y}(z)+y\left.\frac{d}{dt}\frac{\partial h}{\partial x}\left(x+\frac{ty}{2}+iy\right)\right|_{t=0}\\
&=\frac{1}{2}y\frac{\partial h}{\partial y}(z)+\frac{1}{2}y^2\frac{\partial^2h}{\partial y^2}(z)-\frac{1}{2}y\frac{\partial h}{\partial y}(z)+\frac{1}{2}y^2\frac{\partial^2h}{\partial x^2}(z)\\
&=\left(-\frac{1}{2}\Delta_{\bb{H}^2}h\right)(z). 
\end{align*}
\end{proof}

Let $\Delta_\Sigma\colon C^\infty(\Sigma,\bb{C})\to C^\infty(\Sigma,\bb{C})$ be the Laplacian of $\Sigma$. By taking the $\Gamma$-invariant parts of both sides of the isomorphism \eqref{hcpsl}, we get 
\begin{equation*}
p^*\colon C^\infty(\Sigma,\bb{C})\stackrel{\sim}{\to}\left\{f\in C^\infty(P,\bb{C})\ \middle|\ \pi(H_0)f=0\right\}. 
\end{equation*}

\begin{cor}\label{laplacianomega}
We have 
\begin{equation*}
(p^*)^{-1}\pi(\Omega)p^*=-\frac{1}{2}\Delta_\Sigma. 
\end{equation*}
\end{cor}

\begin{proof}
We have the following commutative diagram 
\begin{equation*}
\begin{tikzcd}
C^\infty(\bb{H}^2,\bb{C})^\Gamma\ar[r,"\sim","p^*"']&\left\{f\in C^\infty(\PSL(2,\bb{R}),\bb{C})\ \middle|\ \pi^\prime(H_0)f=0\right\}^\Gamma\\
C^\infty(\Sigma,\bb{C})\ar[u,"\wr","q^*"']\ar[r,"\sim","p^*"']&\left\{f\in C^\infty(P,\bb{C})\ \middle|\ \pi(H_0)f=0\right\}. \ar[u,"\wr","q^*"']
\end{tikzcd}
\end{equation*}
Obviously, $q^*$ on the left intertwines $\Delta_\Sigma$ and $\Delta_{\bb{H}^2}$ and $q^*$ on the right intertwines $\pi(\Omega)$ and $\pi^\prime(\Omega)$. Hence this lemma follows from Lemma \ref{omegalaplacian}. 
\end{proof}

Let $\mu<0$. We have 
\begin{equation*}
p^*\colon\left\{h\in C^\infty(\Sigma,\bb{C})\ \middle|\ \Delta_\Sigma h=-2\mu h\right\}\stackrel{\sim}{\to}\left\{f\in C^\infty(P,\bb{C})\ \middle|\ \pi(H_0)f=0, \pi(\Omega)f=\mu f\right\}
\end{equation*}
by Corollary \ref{laplacianomega}. Let $u_0$ be a nonzero element of $(\mca{H}_\mu)_{(0)}$. We have $(\mca{H}_\mu)_{(0)}=\bb{C}u_0$. 

\begin{lem}\label{homiso}
We have 
\begin{align*}
\Hom_{\PSL(2,\bb{R})}\left(\mca{H}_\mu,L^2(P)\right)&\stackrel{\sim}{\to}\left\{f\in C^\infty(P,\bb{C})\ \middle|\ \pi(H_0)f=0, \pi(\Omega)f=\mu f\right\}\\
\varphi&\mapsto\varphi(u_0), 
\end{align*}
where $\Hom_{\PSL(2,\bb{R})}\left(\mca{H}_\mu,L^2(P)\right)$ is the set of $\PSL(2,\bb{R})$-equivariant bounded operators $\mca{H}_\mu\to L^2(P)$. 
\end{lem}

\begin{proof}
Take $\varphi\in\Hom_{\PSL(2,\bb{R})}\left(\mca{H}_\mu,L^2(P)\right)$. Since $(\mca{H}_\mu)_{(0)}$ is finite dimensional, we have $(\mca{H}_\mu)_{(0)}\subset C^\infty(\mca{H}_\mu)$. See Corollary 4.4.3.3 of Warner \cite{Warner1}. Hence $u_0\in C^\infty(\mca{H}_\mu)$ and then $\varphi(u_0)\in C^\infty(L^2(P))=C^\infty(P,\bb{C})$. The map $\varphi\mapsto\varphi(u_0)$ is defined and it is injective by the irreducibility of $\mca{H}_\mu$. 

To show the surjectivity, first fix $\varphi_i\in\Hom_{\PSL(2,\bb{R})}\left(\mca{H}_\mu,L^2(P)\right)$ such that $\varphi_i\colon\mca{H}_\mu\to\im\varphi_i$ is a unitary equivalence, $\im\varphi_i\perp\im\varphi_j$ if $i\neq j$ and $W_\mu=\bigoplus_{i=1}^m\im\varphi_i$. Take any $f\in C^\infty(P,\bb{C})$ such that $\pi(H_0)f=0$ and $\pi(\Omega)f=\mu f$. By the latter condition, we have $f\in W_\mu$. Write $f=\sum_{i=1}^mf_i$ for some $f_i\in\im\varphi_i$. Since the projection $L^2(P)\to\im\varphi_i$ is bounded, we have $f_i\in C^\infty(\im\varphi_i)$. It follows that $\pi(H_0)f_i=0$ for all $i$. Hence $f_i=a_i\varphi_i(u_0)$ for some $a_i\in\bb{C}$. Now let $\varphi=\sum_{i=1}^ma_i\varphi_i\in\Hom_{\PSL(2,\bb{R})}\left(\mca{H}_\mu,L^2(P)\right)$. Then $\varphi(u_0)=f$ and the map is surjective. 
\end{proof}

\begin{cor}
We have 
\begin{equation*}
\Hom_{\PSL(2,\bb{R})}\left(\mca{H}_\mu,L^2(P)\right)\simeq\left\{h\in C^\infty(\Sigma,\bb{C})\ \middle|\ \Delta_\Sigma h=-2\mu h\right\}. 
\end{equation*}
\end{cor}

Now we recall some properties of the spectrum of the Laplacian. 
See for example the appendix of Kodaira \cite{Kodaira}. The operator $\Delta_\Sigma$ with domain $C^\infty(\Sigma,\bb{C})$, is essentially self-adjoint on $L^2(\Sigma)$. Let $\ol{\Delta_\Sigma}$ be the closure of $\Delta_\Sigma$ and $\sigma\left(\ol{\Delta_\Sigma}\right)$ be the spectrum of $\ol{\Delta_\Sigma}$. We have $\sigma\left(\ol{\Delta_\Sigma}\right)\subset\bb{R}_{\geq0}$. The spectrum $\sigma\left(\ol{\Delta_\Sigma}\right)$ consists of eigenvalues, has no accumulation points in $\bb{R}_{\geq0}$ and is infinite. Moreover, 
\begin{equation*}
\left\{h\in\Dom\ol{\Delta_\Sigma}\ \middle|\ \ol{\Delta_\Sigma}h=\nu h\right\}=\left\{h\in C^\infty(\Sigma,\bb{C})\ \middle|\ \Delta_\Sigma h=\nu h\right\}
\end{equation*}
and $m_\nu=\dim\left\{h\in C^\infty(\Sigma,\bb{C})\ \middle|\ \Delta_\Sigma h=\nu h\right\}$ is finite. 

\begin{cor}
For $\mu<0$, we have 
\begin{equation*}
W_\mu\simeq m_{-2\mu}\mca{H}_\mu. 
\end{equation*}
\end{cor}

\subsubsection{Contribution of holomorphic discrete series}
Define $\alpha\colon\PSL(2,\bb{R})\times\bb{H}^2\to\bb{C}^\times$ by 
\begin{equation*}
\alpha(g,z)=(cz+d)^2, 
\end{equation*}
where $g=\pm
\begin{pmatrix}
a&b\\
c&d
\end{pmatrix}
\in\PSL(2,\bb{R})$. Then $\alpha$ is a $C^\infty$ cocycle over the action $\PSL(2,\bb{R})\curvearrowright\bb{H}^2$, ie it satisfies 
\begin{equation*}
\alpha(gg^\prime,z)=\alpha(g,g^\prime z)\alpha(g^\prime,z)
\end{equation*}
for all $g$, $g^\prime$ and $z$. Define an action $\PSL(2,\bb{R})\curvearrowright\bb{H}^2\times\bb{C}$ by 
\begin{equation*}
g(z,w)=(gz,\alpha(g,z)w). 
\end{equation*}
This makes the trivial bundle $\bb{H}^2\times\bb{C}\to\bb{H}^2$ a $\PSL(2,\bb{R})$-vector bundle over the natural action $\PSL(2,\bb{R})\curvearrowright\bb{H}^2$. Since the surface $\Sigma=\Gamma\backslash\bb{H}^2$ has a natural complex structure coming from $\bb{H}^2$, we have the decomposition $T\Sigma\otimes\bb{C}=T^{1,0}\Sigma\oplus T^{0,1}\Sigma$ of the complexified tangent bundle. We have $T\Sigma\simeq T^{1,0}\Sigma$ as complex vector bundles. Let $L=\Hom_\bb{C}\left(T^{1,0}\Sigma,\bb{C}\right)$ be the canonical line bundle of $\Sigma$. The complex vector bundle $L$ is a holomorphic vector bundle. 

\begin{lem}\label{veciso}
We have $L\simeq\bb{H}^2\times_\Gamma\bb{C}$ as holomorphic vector bundles, where the right hand side is the quotient $\Gamma\backslash(\bb{H}^2\times\bb{C})$ by the restriction of the action $\PSL(2,\bb{R})\curvearrowright\bb{H}^2\times\bb{C}$ defined above. 
\end{lem}

\begin{proof}
Let $L_{\bb{H}^2}=\Hom_\bb{C}\left(T^{1,0}\bb{H}^2,\bb{C}\right)$ be the canonical bundle of $\bb{H}^2$. We have a holomorphic trivialization $L_{\bb{H}^2}\simeq\bb{H}^2\times\bb{C}$ given by the section $dz=dx+idy$. Since the action $\PSL(2,\bb{R})\curvearrowright\bb{H}^2$ is holomorphic, we get actions $\PSL(2,\bb{R})\curvearrowright T^{1,0}\bb{H}^2$ and $\PSL(2,\bb{R})\curvearrowright L_{\bb{H}^2}$. Explicitly the latter action is given by 
\begin{align*}
(L_{\bb{H}^2})_z&\to(L_{\bb{H}^2})_{gz}\\
\varphi&\mapsto\varphi\circ(g^{-1})_*
\end{align*}
for $g\in\PSL(2,\bb{R})$ and $z\in\bb{H}^2$. We compute $(dz)_{z_0}\circ(g^{-1})_*$ for a point $z_0\in\bb{H}^2$. Let $g^{-1}=\pm
\begin{pmatrix}
a&b\\
c&d
\end{pmatrix}
$. Then 
\begin{align*}
(dz)_{z_0}\circ(g^{-1})_*&=\left((g^{-1})^*dz\right)_{gz_0}=\left(d\left(\frac{az+b}{cz+d}\right)\right)_{gz_0}\\
&=\left(\frac{1}{(cz+d)^2}dz\right)_{gz_0}=\left(\alpha\left(g^{-1},z\right)^{-1}dz\right)_{gz_0}\\
&=\alpha\left(g^{-1},gz_0\right)^{-1}(dz)_{gz_0}=\alpha(g,z_0)(dz)_{gz_0}. 
\end{align*}
We used the relation $1=\alpha(g^{-1}g,z_0)=\alpha(g^{-1},gz_0)\alpha(g,z_0)$ at the last line. This computation shows the isomorphism $L_{\bb{H}^2}\simeq\bb{H}^2\times\bb{C}$ is $\PSL(2,\bb{R})$-equivariant. Since $\Gamma\backslash L_{\bb{H}^2}\simeq L$, we get $L\simeq\bb{H}^2\times_\Gamma\bb{C}$. 
\end{proof}

Let $q\colon\bb{H}^2\to\Sigma$ be the natural projection. 

\begin{cor}\label{gammah}
For $n\in\bb{Z}_{\geq1}$, we have an isomorphism 
\begin{align*}
\Gamma(L^{\otimes n})&\simeq\left\{h\in C^\infty(\bb{H}^2,\bb{C})\ \middle|\ h(\gamma z)=\alpha(\gamma,z)^nh(z)\ \text{{\rm for all $\gamma\in\Gamma$ and $z\in\bb{H}^2$}}\right\}\\
\xi&\mapsto h
\end{align*}
given by $q^*\xi=hdz^{\otimes n}$. Moreover we have $q^*\bar{\partial}\xi=\frac{\partial h}{\partial\bar{z}}d\bar{z}\otimes dz^{\otimes n}$ for all $\xi\in\Gamma(L^{\otimes n})$, where $\bar{\partial}$ is the holomorphic structure of $L^{\otimes n}$. 
\end{cor}

\begin{proof}
Consider the action $\PSL(2,\bb{R})\curvearrowright\bb{H}^2\times\bb{C}$ defined by $g(z,w)=(gz,\alpha(g,z)^nw)$. Then we have $L^{\otimes n}\simeq\bb{H}^2\times_\Gamma\bb{C}$ as holomorphic vector bundles, where the action $\Gamma\curvearrowright\bb{H}^2\times\bb{C}$ is the restriction of the above action. The action $\PSL(2,\bb{R})\curvearrowright\bb{H}^2\times\bb{C}$ gives rise to a representation $\PSL(2,\bb{R})\curvearrowright C^\infty(\bb{H}^2,\bb{C})$ defined by 
\begin{equation*}
(gh)(z)=\alpha(g,g^{-1}z)^nh(g^{-1}z)=\alpha(g^{-1},z)^{-n}h(g^{-1}z). 
\end{equation*}
Here we used the relation $1=\alpha(gg^{-1},z)=\alpha(g,g^{-1}z)\alpha(g^{-1},z)$. Since 
\begin{equation*}
\Gamma(L^{\otimes n})\simeq\Gamma\left(\bb{H}^2\times_\Gamma\bb{C}\right)\simeq C^\infty\left(\bb{H}^2,\bb{C}\right)^\Gamma, 
\end{equation*}
the stated isomorphism holds. 
\end{proof}

\begin{lem}\label{phiprime}
For $n\in\bb{Z}_{\geq1}$, we have an isomorphism 
\begin{equation*}
\Phi_n^\prime\colon C^\infty(\bb{H}^2,\bb{C})\stackrel{\sim}{\to}\left\{f\in C^\infty(\PSL(2,\bb{R}),\bb{C})\ \middle|\ \pi^\prime(H_0)f=nf\right\}
\end{equation*}
defined by 
\begin{equation*}
\Phi_n^\prime(h)(g)=\alpha(g,i)^{-n}h(gi). 
\end{equation*}
The isomorphism $\Phi_n^\prime$ is $\PSL(2,\bb{R})$-equivariant, where $\PSL(2,\bb{R})$ acts on $C^\infty(\bb{H}^2,\bb{C})$ by 
\begin{equation*}
(gh)(z)=\alpha(g,g^{-1}z)^nh(g^{-1}z)=\alpha(g^{-1},z)^{-n}h(g^{-1}z)
\end{equation*}
and on $\left\{f\in C^\infty(\PSL(2,\bb{R}),\bb{C})\ \middle|\ \pi^\prime(H_0)f=nf\right\}$ by 
\begin{equation*}
(gf)(g^\prime)=f(g^{-1}g^\prime). 
\end{equation*}
\end{lem}

\begin{proof}
Put $r_\theta=\exp\theta
\begin{pmatrix}
&\frac{1}{2}\\
-\frac{1}{2}&
\end{pmatrix}
=\pm
\begin{pmatrix}
\cos\frac{\theta}{2}&\sin\frac{\theta}{2}\\
-\sin\frac{\theta}{2}&\cos\frac{\theta}{2}
\end{pmatrix}
$. Since $H_0=i
\begin{pmatrix}
&-\frac{1}{2}\\
\frac{1}{2}&
\end{pmatrix}
$, we have 
\begin{align*}
\left(\pi^\prime(H_0)\Phi_n^\prime(h)\right)(g)&=i\left.\frac{d}{dt}\Phi_n^\prime(h)\left(g\exp t
\begin{pmatrix}
&-\frac{1}{2}\\
\frac{1}{2}&
\end{pmatrix}
\right)\right|_{t=0}=i\left.\frac{d}{dt}\Phi_n^\prime(h)(gr_{-t})\right|_{t=0}\\
&=i\left.\frac{d}{dt}\alpha(gr_{-t},i)^{-n}h(gi)\right|_{t=0}\\
&=i\left.\frac{d}{dt}\alpha(g,i)^{-n}\alpha(r_{-t},i)^{-n}h(gi)\right|_{t=0}\\
&=i\left.\frac{d}{dt}\alpha(g,i)^{-n}e^{-int}h(gi)\right|_{t=0}=n\Phi_n^\prime(h)(g). 
\end{align*}
Hence the map $\Phi_n^\prime$ is defined and it is obviously injective. To show the surjectivity, take any $f\in C^\infty(\PSL(2,\bb{R}),\bb{C})$ such that $\pi^\prime(H_0)f=nf$. Since 
\begin{equation*}
\frac{d}{d\theta}f(gr_\theta)=i\left(\pi^\prime(H_0)f\right)(gr_\theta)=inf(gr_\theta), 
\end{equation*}
we get $f(gr_\theta)=e^{in\theta}f(g)$. Since 
\begin{align*}
\alpha(gr_\theta,i)^nf(gr_\theta)=\alpha(g,i)^n\alpha(r_\theta,i)^ne^{in\theta}f(g)=\alpha(g,i)^nf(g), 
\end{align*}
there exists $h\colon\bb{H}^2\to\bb{C}$ such that $h(gi)=\alpha(g,i)^nf(g)$. Thus $\Phi_n^\prime(h)=f$, which shows the surjectivity. Finally the equivariance follows from the following computation: 
\begin{align*}
\Phi_n^\prime(gh)(g^\prime)&=\alpha(g^\prime,i)^{-n}(gh)(g^\prime i)=\alpha(g^\prime,i)^{-n}\alpha\left(g^{-1},g^\prime i\right)^{-n}h\left(g^{-1}g^\prime i\right)\\
&=\alpha\left(g^{-1}g^\prime,i\right)^{-n}h\left(g^{-1}g^\prime i\right)=\Phi_n^\prime(h)(g^{-1}g^\prime)=\left(g\Phi_n^\prime(h)\right)(g^\prime). 
\end{align*}
\end{proof}

\begin{cor}
For $n\in\bb{Z}_{\geq1}$, we have an isomorphism 
\begin{align*}
\Phi_n&\colon\left\{h\in C^\infty(\bb{H}^2,\bb{C})\ \middle|\ h(\gamma z)=\alpha(\gamma,z)^nh(z)\ \text{{\rm for all $\gamma\in\Gamma$ and $z\in\bb{H}^2$}}\right\}\\
&\stackrel{\sim}{\to}\left\{f\in C^\infty(P,\bb{C})\ \middle|\ \pi(H_0)f=nf\right\}
\end{align*}
defined by 
\begin{equation*}
\Phi_n(h)(\Gamma g)=\alpha(g,i)^{-n}h(gi). 
\end{equation*}
Note that the former set is isomorphic to $\Gamma(L^{\otimes n})$ by Corollary \ref{gammah}. 
\end{cor}

\begin{proof}
Take $\Gamma$-invariant parts of both sides of the isomorphism in Lemma \ref{phiprime}. 
\end{proof}

Let $H_-=i
\begin{pmatrix}
\frac{1}{2}&\\
&-\frac{1}{2}
\end{pmatrix}
+
\begin{pmatrix}
&\frac{1}{2}\\
\frac{1}{2}&
\end{pmatrix}
\in\sl(2,\bb{R})\otimes\bb{C}$. 

\begin{lem}\label{formula}
We have 
\begin{equation*}
\left(\pi^\prime(H_-)\Phi_n^\prime(h)\right)(g)=2\alpha(g,i)^{-n+1}(\Im gi)^2\frac{\partial h}{\partial\bar{z}}(gi)
\end{equation*}
for $n\in\bb{Z}_{\geq1}$, $h\in C^\infty(\bb{H}^2,\bb{C})$ and $g\in\PSL(2,\bb{R})$. 
\end{lem}

\begin{proof}
Since $H_0=i
\begin{pmatrix}
&-\frac{1}{2}\\
\frac{1}{2}&
\end{pmatrix}
$, we have $H_-=i
\begin{pmatrix}
\frac{1}{2}&\\
&-\frac{1}{2}
\end{pmatrix}
+
\begin{pmatrix}
&1\\
&
\end{pmatrix}
-iH_0$, hence 
\begin{equation*}
\pi^\prime(H_-)\Phi_n^\prime(h)=\left(i\pi^\prime
\begin{pmatrix}
\frac{1}{2}&\\
&-\frac{1}{2}
\end{pmatrix}
+\pi^\prime
\begin{pmatrix}
&1\\
&
\end{pmatrix}
\right)\Phi_n^\prime(h)-in\Phi_n^\prime(h). 
\end{equation*}
We have 
\begin{align*}
\left(\pi^\prime
\begin{pmatrix}
\frac{1}{2}&\\
&-\frac{1}{2}
\end{pmatrix}
\Phi_n^\prime(h)\right)(g)&=\left.\frac{d}{dt}\Phi_n^\prime(h)\left(g
\begin{pmatrix}
e^\frac{t}{2}&\\
&e^{-\frac{t}{2}}
\end{pmatrix}
\right)\right|_{t=0}\\
&=\left.\frac{d}{dt}\alpha\left(g
\begin{pmatrix}
e^\frac{t}{2}&\\
&e^{-\frac{t}{2}}
\end{pmatrix}
,i\right)^{-n}h\left(g
\begin{pmatrix}
e^\frac{t}{2}&\\
&e^{-\frac{t}{2}}
\end{pmatrix}
i\right)\right|_{t=0}\\
&=\left.\frac{d}{dt}\alpha\left(g,e^ti\right)^{-n}e^{nt}h\left(ge^ti\right)\right|_{t=0}. 
\end{align*}
Let $g=\pm
\begin{pmatrix}
a&b\\
c&d
\end{pmatrix}
$ and then 
\begin{equation*}
\left.\frac{d}{dt}\alpha\left(g,e^ti\right)^{-n}\right|_{t=0}=\left.\frac{d}{dt}\left(ce^ti+d\right)^{-2n}\right|_{t=0}=-2nci(ci+d)^{-2n-1}. 
\end{equation*}
Recall that for a $C^\infty$ function $k$ on an open set $U$ of $\bb{C}$ and a $C^\infty$ curve $c(t)$ in $U$, we have 
\begin{equation*}
\frac{d}{dt}k(c(t))=\frac{\partial k}{\partial z}(c(t))\frac{dc}{dt}(t)+\frac{\partial k}{\partial\bar{z}}(c(t))\ol{\frac{dc}{dt}(t)}. 
\end{equation*}
Using this and the formula $\frac{d}{dz}\frac{az+b}{cz+d}=\frac{1}{(cz+d)^2}$, we have 
\begin{align*}
\left.\frac{d}{dt}h\left(ge^ti\right)\right|_{t=0}&=\frac{\partial h}{\partial z}(gi)\left.\frac{d}{dt}ge^ti\right|_{t=0}+\frac{\partial h}{\partial\bar{z}}(gi)\ol{\left.\frac{d}{dt}ge^ti\right|_{t=0}}\\
&=\frac{\partial h}{\partial z}(gi)\frac{\partial g}{\partial z}(i)\left.\frac{d}{dt}e^ti\right|_{t=0}+\frac{\partial h}{\partial\bar{z}}(gi)\ol{\frac{\partial g}{\partial z}(i)\left.\frac{d}{dt}e^ti\right|_{t=0}}\\
&=\frac{\partial h}{\partial z}(gi)(ci+d)^{-2}i-\frac{\partial h}{\partial\bar{z}}(gi)(-ci+d)^{-2}i. 
\end{align*}
Therefore 
\begin{align*}
&\left(\pi^\prime
\begin{pmatrix}
\frac{1}{2}&\\
&-\frac{1}{2}
\end{pmatrix}
\Phi_n^\prime(h)\right)(g)\\
&=-2nci(ci+d)^{-2n-1}h(gi)+(ci+d)^{-2n}nh(gi)\\
&\quad+(ci+d)^{-2n}\left(\frac{\partial h}{\partial z}(gi)(ci+d)^{-2}i-\frac{\partial h}{\partial\bar{z}}(gi)(-ci+d)^{-2}i\right). 
\end{align*}
To compute the next part, first we have 
\begin{align*}
\left(\pi^\prime
\begin{pmatrix}
&1\\
&
\end{pmatrix}
\Phi_n^\prime(h)\right)(g)&=\left.\frac{d}{dt}\Phi_n^\prime(h)\left(g
\begin{pmatrix}
1&t\\
&1
\end{pmatrix}
\right)\right|_{t=0}\\
&=\left.\frac{d}{dt}\alpha\left(g
\begin{pmatrix}
1&t\\
&1
\end{pmatrix}
,i\right)^{-n}h\left(g
\begin{pmatrix}
1&t\\
&1
\end{pmatrix}
i\right)\right|_{t=0}\\
&=\left.\frac{d}{dt}\alpha\left(g,(i+t)\right)^{-n}h\left(g(i+t)\right)\right|_{t=0}. 
\end{align*}
Since 
\begin{equation*}
\left.\frac{d}{dt}\alpha\left(g,(i+t)\right)^{-n}\right|_{t=0}=\left.\frac{d}{dt}\left(c(i+t)+d\right)^{-2n}\right|_{t=0}=-2nc(ci+d)^{-2n-1}
\end{equation*}
and 
\begin{align*}
\left.\frac{d}{dt}h\left(g(i+t)\right)\right|_{t=0}&=\frac{\partial h}{\partial z}(gi)\left.\frac{d}{dt}g(i+t)\right|_{t=0}+\frac{\partial h}{\partial\bar{z}}(gi)\ol{\left.\frac{d}{dt}g(i+t)\right|_{t=0}}\\
&=\frac{\partial h}{\partial z}(gi)\frac{\partial g}{\partial z}(i)+\frac{\partial h}{\partial\bar{z}}(gi)\ol{\frac{\partial g}{\partial z}(i)}\\
&=\frac{\partial h}{\partial z}(gi)(ci+d)^{-2}+\frac{\partial h}{\partial\bar{z}}(gi)(-ci+d)^{-2}, 
\end{align*}
we get 
\begin{align*}
&\left(\pi^\prime
\begin{pmatrix}
&1\\
&
\end{pmatrix}
\Phi_n^\prime(h)\right)(g)\\
&=-2nc(ci+d)^{-2n-1}h(gi)\\
&\quad+(ci+d)^{-2n}\left(\frac{\partial h}{\partial z}(gi)(ci+d)^{-2}+\frac{\partial h}{\partial\bar{z}}(gi)(-ci+d)^{-2}\right). 
\end{align*}
Using $\Im gi=\Im\frac{ai+b}{ci+d}=\Im\frac{(ai+b)(-ci+d)}{c^2+d^2}=\frac{1}{c^2+d^2}$, 
\begin{align*}
&\left(\pi^\prime(H_-)\Phi_n^\prime(h)\right)(g)\\
&=2nc(ci+d)^{-2n-1}h(gi)+(ci+d)^{-2n}inh(gi)\\
&\quad+(ci+d)^{-2n}\left(-\frac{\partial h}{\partial z}(gi)(ci+d)^{-2}+\frac{\partial h}{\partial\bar{z}}(gi)(-ci+d)^{-2}\right)\\
&\quad-2nc(ci+d)^{-2n-1}h(gi)\\
&\quad+(ci+d)^{-2n}\left(\frac{\partial h}{\partial z}(gi)(ci+d)^{-2}+\frac{\partial h}{\partial\bar{z}}(gi)(-ci+d)^{-2}\right)-in\Phi_n^\prime(h)(g)\\
&=2(ci+d)^{-2n}(-ci+d)^{-2}\frac{\partial h}{\partial\bar{z}}(gi)=2(ci+d)^{-2n+2}(c^2+d^2)^{-2}\frac{\partial h}{\partial\bar{z}}(gi)\\
&=2\alpha(g,i)^{-n+1}(\Im gi)^2\frac{\partial h}{\partial\bar{z}}(gi). 
\end{align*}
\end{proof}

Let $H^0(\Sigma;L^{\otimes n})=\left\{\xi\in\Gamma(L^{\otimes n})\ \middle|\ \bar{\partial}\xi=0\right\}$ be the set of holomorphic sections of $L^{\otimes n}$. 

\begin{cor}
For $n\in\bb{Z}_{\geq1}$, we have 
\begin{equation*}
H^0(\Sigma;L^{\otimes n})\simeq\left\{f\in C^\infty(P,\bb{C})\ \middle|\ \pi(H_0)f=nf,\ \pi(H_-)f=0\right\}. 
\end{equation*}
\end{cor}

\begin{proof}
By Lemma \ref{formula} we have 
\begin{equation*}
\left(\pi(H_-)\Phi_n(h)\right)(\Gamma g)=2\alpha(g,i)^{-n+1}(\Im gi)^2\frac{\partial h}{\partial\bar{z}}(gi)
\end{equation*}
for all $h\in C^\infty(\bb{H}^2,\bb{C})^\Gamma$ and $g\in\PSL(2,\bb{R})$. Hence the restriction of $\Phi_n$ gives an isomorphism 
\begin{equation*}
\left\{h\in C^\infty(\bb{H}^2,\bb{C})^\Gamma\ \middle|\ \frac{\partial h}{\partial\bar{z}}=0\right\}\simeq\left\{f\in C^\infty(P,\bb{C})\ \middle|\ \pi(H_0)f=nf,\ \pi(H_-)f=0\right\}. 
\end{equation*}
The former set is isomorphic to $H^0(\Sigma;L^{\otimes n})$ by Corollary \ref{gammah}. 
\end{proof}

Let $u_n$ be a nonzero element of $(D_n^+)_{(n)}$. Hence $(D_n^+)_{(n)}=\bb{C}u_n$. 

\begin{lem}
For $n\in\bb{Z}_{\geq1}$, we have an isomorphism 
\begin{align*}
\Hom_{\PSL(2,\bb{R})}\left(D_n^+,L^2(P)\right)&\simeq\left\{f\in C^\infty(P,\bb{C})\ \middle|\ \pi(H_0)f=nf,\ \pi(H_-)f=0\right\}\\
\varphi&\mapsto\varphi(u_n). 
\end{align*}
\end{lem}

\begin{proof}
The map is defined and injective as in the proof of Lemma \ref{homiso}. To show the surjectivity, take $f\in C^\infty(P,\bb{C})$ such that $\pi(H_0)f=nf$ and $\pi(H_-)f=0$. Let $W_n^\prime$ be the isotypic component of $D_n^+$ in $L^2(P)$. Fix an orthogonal direct sum decomposition $L^2(P)=\bigoplus_iV_i$ into irreducible unitary subrepresentations. Write $f=\sum_if_i$ according to the decomposition. By the condition $\pi(H_0)f=nf$, we have $f\in L^2(P)_{(n)}=\bigoplus_i(V_i)_{(n)}$ and hence $f_i\in(V_i)_{(n)}$. 

First note that $f_i=0$ when $(V_i)_{(n)}=0$, ie when $V_i$ is isomorphic to $D_m^+$ for $m>n$, $D_m^-$ for all $m\geq1$ or $\bb{C}$. Since $0=\pi(H_-)f=\sum_i\pi(H_-)f_i$, we have $\pi(H_-)f_i=0$. When $V_i$ is isomorphic to principal, complementary series or $D_m^+$ for $m=1,\ldots,n-1$, we have $\pi(H_-)\colon(V_i)_{(n)}\stackrel{\sim}{\to}(V_i)_{(n-1)}$. (See for example Section 2 of Chapter 8 of Taylor \cite{Taylor}.) Hence $f_i=0$. This shows $f\in W_n^\prime$. 

The rest of the proof is similar to the proof of Lemma \ref{homiso}. 
\end{proof}

\begin{cor}
For $n\in\bb{Z}_{\geq1}$, we have 
\begin{equation*}
\Hom_{\PSL(2,\bb{R})}\left(D_n^+,L^2(P)\right)\simeq H^0(\Sigma;L^{\otimes n}). 
\end{equation*}
\end{cor}

Let $g_\Sigma$ be the genus of $\Sigma$. We have $g_\Sigma\geq2$. By the Riemann--Roch theorem, 
\begin{equation*}
\dim H^0(\Sigma;L^{\otimes n})=
\begin{cases}
g_\Sigma&n=1\\
(2n-1)(g_\Sigma-1)&n\geq2. 
\end{cases}
\end{equation*}

\subsubsection{Contribution of antiholomorphic discrete series}
For a complex vector space $V$, let $\ol{V}$ be the complex conjugate of $V$. Let $\ol{v}$ denote the element corresponding to an element $v\in V$. For complex vector spaces $V$, $W$ and a linear map $f\colon V\to W$, define $\ol{f}\colon\ol{V}\to\ol{W}$ by $\ol{f}(\ol{v})=\ol{f(v)}$. Then $\ol{f}$ is a linear map. For a complex Hilbert space $\mca{H}$, the complex conjugate $\ol{\mca{H}}$ is naturally a Hilbert space with the inner product $\langle\ol{v},\ol{w}\rangle=\ol{\langle v,w\rangle}$ for $\ol{v},\ol{w}\in\ol{\mca{H}}$. For a unitary representation $\PSL(2,\bb{R})\stackrel{\pi}{\curvearrowright}\mca{H}$, we have a unitary representation $\PSL(2,\bb{R})\stackrel{\ol{\pi}}{\curvearrowright}\ol{\mca{H}}$ defined by $\ol{\pi}(g)=\ol{\pi(g)}$. 

By looking at realizations of $D_n^+$ and $D_n^-$, we see that $\ol{D_n^+}$ is unitarily equivalent to $D_n^-$. 

Let $W\subset L^2(P)$ be a closed invariant subspace unitarily equivalent to $D_n^\pm$. The complex conjugate $\ol{W}$ is unitarily equivalent to $\left\{\ol{f}\in L^2(P)\ \middle|\ f\in W\right\}$ by the map $\ol{f}\mapsto\ol{f}$. Hence $\left\{\ol{f}\in L^2(P)\ \middle|\ f\in W\right\}$ is unitarily equivalent to $D_n^\mp$. This shows the multiplicity of $D_n^+$ in $L^2(P)$ is the same as that of $D_n^-$.

\subsubsection{Conclusion}\label{conclu}
By the previous two sections, we have 
\begin{equation*}
W_0\simeq\bb{C}\oplus g_\Sigma(D_1^+\oplus D_1^-)
\end{equation*}
and 
\begin{equation*}
W_{\frac{n(n-1)}{2}}\simeq(2n-1)(g_\Sigma-1)(D_n^+\oplus D_n^-)
\end{equation*}
for $n\in\bb{Z}_{\geq2}$. 

\begin{thm}\label{unidecompkk}
We have 
\begin{align*}
L^2(P)\simeq&\bigoplus_{-2\mu\in\sigma\left(\ol{\Delta_\Sigma}\right)\setminus\{0\}}m_{-2\mu}\mca{H}_\mu\oplus\bb{C}\oplus g_\Sigma(D_1^+\oplus D_1^-)\\
&\qquad\oplus\bigoplus_{n=2}^\infty(2n-1)(g_\Sigma-1)(D_n^+\oplus D_n^-), 
\end{align*}
where $g_\Sigma$ is the genus of the surface $\Sigma=P/\PSO(2)$ and $m_\nu$ is the multiplicity of the eigenvalue $\nu$ of the closure $\ol{\Delta_\Sigma}$ of the Laplacian of $\Sigma$. 
\end{thm}

\section{Computation of $H^*\left(\mf{an};C^\infty(\mca{H})\otimes\bb{C}_\lambda\right)$ without a Hodge decomposition}
In this section we will compute $H^*\left(\mf{an};C^\infty(\mca{H})\otimes\bb{C}_\lambda\right)$ for an irreducible unitary representation $\PSL(2,\bb{R})\curvearrowright\mca{H}$ and $\lambda\in\bb{C}$, by using the Hochschild--Serre spectral sequence: 
\begin{equation*}
E_2^{p,q}\simeq H^p\left(\mf{an}/\mf{n};H^q(\mf{n};C^\infty(\mca{H})\otimes\bb{C}_\lambda)\right)\Rightarrow H^{p+q}\left(\mf{an};C^\infty(\mca{H})\otimes\bb{C}_\lambda\right). 
\end{equation*}

\subsection{Hochschild--Serre spectral sequence}\label{hochschildserre}
We review the Hochschild--Serre spectral sequence \cite{HS} in this section. Let $K\subset L$ be a field extension. Let $\mf{g}$ be a Lie algebra over $K$, $\mf{h}$ be a subalgebra of $\mf{g}$ and $\pi\colon\mf{g}\to\End_L(V)$ be a representation on a vector space $V$ over $L$. The Chevalley--Eilenberg complex $C^*(\mf{g};V)=\bigwedge^*\mf{g}^*\otimes V$ is a vector space over $L$. 

The Hochschild--Serre spectral sequence is a spectral sequence arising form a filtration of $C^*(\mf{g};V)$, which computes the cohomology $H^*(\mf{g};V)$. Set 
\begin{equation*}
F^pC^{p+q}(\mf{g};V)=\left\{\varphi\in C^{p+q}(\mf{g};V)\ \middle|\ \iota_{X_1}\cdots\iota_{X_{q+1}}\varphi=0\ \text{if $X_1,\ldots,X_{q+1}\in\mf{h}$}\right\}. 
\end{equation*}
Then we have 
\begin{equation*}
F^pC^{p+q}(\mf{g};V)\supset F^{p+1}C^{p+q}(\mf{g};V)
\end{equation*}
and 
\begin{equation*}
C^n(\mf{g};V)=F^0C^n(\mf{g};V)\supset F^1C^n(\mf{g};V)\supset\cdots\supset F^nC^n(\mf{g};V)\supset F^{n+1}C^n(\mf{g};V)=0. 
\end{equation*}
The differential preserves the filtration: 
\begin{equation*}
dF^pC^{p+q}(\mf{g};V)\subset F^pC^{p+q+1}(\mf{g};V). 
\end{equation*}

We get a spectral sequence $E^{p,q}_r$ from the filtered complex $F^*C^*(\mf{g};V)$. The spectral sequence converges to $H^*(\mf{g};V)$. This spectral sequence is defined as follows. Put $C^*=C^*(\mf{g};V)$ and define 
\begin{align*}
Z^{p,q}_r&=\ker\left(d\colon F^pC^{p+q}\to C^{p+q+1}/F^{p+r}C^{p+q+1}\right), \\
B^{p,q}_r&=\im\left(d\colon F^{p-r+1}C^{p+q-1}\to C^{p+q}\right)\cap F^pC^{p+q}, \\
E^{p,q}_r&=Z^{p,q}_r\left/\middle(B^{p,q}_r+Z^{p+1,q-1}_{r-1}\right). 
\end{align*}
The differential $d\colon E^{p,q}_r\to E^{p+r,q-r+1}_r$ is induced from $d\colon Z^{p,q}_r\to Z^{p+r,q-r+1}_r$ and 
\begin{equation*}
H^*(E_r)\simeq E_{r+1}. 
\end{equation*}

Assume that $\mf{h}$ is an ideal of $\mf{g}$. Define 
\begin{equation*}
(L_X\varphi)(X_1,\ldots,X_p)=\pi(X)\left(\varphi(X_1,\ldots,X_p)\right)-\sum_{i=1}^p\varphi(X_1,\ldots,[X,X_i],\ldots,X_p)
\end{equation*}
for $X\in\mf{g}$, $\varphi\in C^p(\mf{h};V)$ and $X_1,\ldots,X_p\in\mf{h}$. Since $L_{[X,Y]}=[L_X,L_Y]$ for $X,Y\in\mf{g}$, this gives a representation $\mf{g}\curvearrowright C^*(\mf{h};V)$. The differential $d\colon C^p(\mf{h};V)\to C^{p+1}(\mf{h};V)$ commutes with $L_X$ for all $X\in\mf{g}$. Hence we get $\mf{g}\curvearrowright H^*(\mf{h};V)$. By the formula $L_X=d\iota(X)+\iota(X)d$ for $X\in\mf{h}$, we actually have a representation $\mf{g}/\mf{h}\curvearrowright H^*(\mf{h};V)$. 

\begin{prop}[Hochschild--Serre \cite{HS}]
If $\mf{h}$ is an ideal of $\mf{g}$, we have 
\begin{equation*}
E^{p,q}_2\simeq H^p\left(\mf{g}/\mf{h};H^q(\mf{h};V)\right). 
\end{equation*}
\end{prop}

\begin{proof}
See Corollary to Theorem 4 in \cite{HS}. 
\end{proof}

\subsection{A general consideration}\label{gencons}
Let $\sl(2,\bb{R})\stackrel{\pi}{\curvearrowright}\mca{H}$ be a complex representation and $\mf{an}\stackrel{\xi}{\curvearrowright}V$ be a real or complex representation such that $\xi(\mf{n})=0$. The Hochschild--Serre spectral sequence gives 
\begin{equation*}
E_2^{p,q}\simeq H^p\left(\mf{an}/\mf{n};H^q(\mf{n};\mca{H}\otimes V)\right)\Rightarrow H^{p+q}(\mf{an};\mca{H}\otimes V). 
\end{equation*}
The tensor product $\mca{H}\otimes V$ is taken over $\bb{C}$ when $V$ is complex. The $E_2$ sheet is 
\begin{equation*}
\begin{tikzpicture}
\matrix(m)[matrix of math nodes,
nodes in empty cells,
nodes={minimum width=5ex,minimum height=5ex,outer sep=-5pt},
column sep=1ex,row sep=1ex]{
H^0\left(\mf{an}/\mf{n};H^1(\mf{n};\mca{H}\otimes V)\right)&H^1\left(\mf{an}/\mf{n};H^1(\mf{n};\mca{H}\otimes V)\right)\\
H^0\left(\mf{an}/\mf{n};H^0(\mf{n};\mca{H}\otimes V)\right)&H^1\left(\mf{an}/\mf{n};H^0(\mf{n};\mca{H}\otimes V)\right). \\
};
\end{tikzpicture}
\end{equation*}
Since there are no nontrivial differentials at $E_2$, the spectral sequence degenerates at $E_2$ and therefore $E_2\simeq E_\infty$. Hence we get 
\begin{align}
H^0(\mf{an};\mca{H}\otimes V)&\simeq H^0\left(\mf{an}/\mf{n};H^0(\mf{n};\mca{H}\otimes V)\right), \label{666}\\
H^1(\mf{an};\mca{H}\otimes V)&\simeq H^0\left(\mf{an}/\mf{n};H^1(\mf{n};\mca{H}\otimes V)\right)\oplus H^1\left(\mf{an}/\mf{n};H^0(\mf{n};\mca{H}\otimes V)\right), \label{777}\\
H^2(\mf{an};\mca{H}\otimes V)&\simeq H^1\left(\mf{an}/\mf{n};H^1(\mf{n};\mca{H}\otimes V)\right). \label{888}
\end{align}
There exists an exact sequence 
\begin{align*}
0&\to H^0\left(\mf{an}/\mf{n};H^q(\mf{n};\mca{H}\otimes V)\right)\\
&\to H^q(\mf{n};\mca{H}\otimes V)\stackrel{L_H}{\to}H^q(\mf{n};\mca{H}\otimes V)\\
&\to H^1\left(\mf{an}/\mf{n};H^q(\mf{n};\mca{H}\otimes V)\right)\to0
\end{align*}
for $q=0,1$. So we need to compute the kernel and cokernel of the map 
\begin{equation}\label{liederilh}
H^q(\mf{n};\mca{H}\otimes V)\stackrel{L_H}{\to}H^q(\mf{n};\mca{H}\otimes V)
\end{equation}
to compute the cohomology $H^*(\mf{an};\mca{H}\otimes V)$. We have an isomorphism 
\begin{equation}\label{isoco}
\left(C^*(\mf{n};\mca{H}\otimes V),d\right)\simeq\left(C^*(\mf{n};\mca{H})\otimes V,d\otimes\id\right)
\end{equation}
of cochain complexes since $\xi(\mf{n})=0$. Therefore 
\begin{equation*}
H^*(\mf{n};\mca{H}\otimes V)\simeq H^*(\mf{n};\mca{H})\otimes V. 
\end{equation*}
We have 
\begin{equation*}
L_H=L_H\otimes\id+\id\otimes\xi(H)
\end{equation*}
through the isomorphism \eqref{isoco}, hence on the cohomologies. Assume that $\xi(H)$ is a scalar multiplication. The map \eqref{liederilh} becomes 
\begin{equation*}
\left(L_H+\xi(H)\right)\otimes\id\colon H^q(\mf{n};\mca{H})\otimes V\to H^q(\mf{n};\mca{H})\otimes V. 
\end{equation*}
Hence 
\begin{align}
H^0\left(\mf{an}/\mf{n};H^q(\mf{n};\mca{H}\otimes V)\right)&\simeq\ker\left(H^q(\mf{n};\mca{H})\stackrel{L_H+\xi(H)}{\longrightarrow}H^q(\mf{n};\mca{H})\right)\otimes V, \label{111111}\\
H^1\left(\mf{an}/\mf{n};H^q(\mf{n};\mca{H}\otimes V)\right)&\simeq\coker\left(H^q(\mf{n};\mca{H})\stackrel{L_H+\xi(H)}{\longrightarrow}H^q(\mf{n};\mca{H})\right)\otimes V. \label{111222}
\end{align}
So it suffices to compute the kernel and cokernel of the map 
\begin{equation*}
L_H+\xi(H)\colon H^q(\mf{n};\mca{H})\to H^q(\mf{n};\mca{H})
\end{equation*}
to compute the $E_2$ of the spectral sequence. 

When $H^q(\mf{n};\mca{H})$ is finite dimensional, the kernel and cokernel of $L_H+\xi(H)$ on $H^q(\mf{n};\mca{H})$ have the same dimension. Therefore $H^0\left(\mf{an}/\mf{n};H^q(\mf{n};\mca{H}\otimes V)\right)$ and $H^1\left(\mf{an}/\mf{n};H^q(\mf{n};\mca{H}\otimes V)\right)$ are isomorphic, but not canonically.

\subsection{The $\mf{n}$-invariant distributions of irreducible unitary representations}\label{tnidoiu}
Let $\PSL(2,\bb{R})\stackrel{\pi}{\curvearrowright}\mca{H}$ be an irreducible unitary representation, $\sl(2,\bb{R})\stackrel{\pi}{\curvearrowright}C^\infty(\mca{H})$ be its derivative representation on the space of $C^\infty$ vectors and $\mf{an}\curvearrowright\bb{C}_\lambda$ for $\lambda\in\bb{C}$ be the representation defined in Section \ref{drws}. We will compute $H^*\left(\mf{an};C^\infty(\mca{H})\otimes\bb{C}_\lambda\right)$. By \eqref{111111} and \eqref{111222} we need to compute the kernel and cokernel of 
\begin{equation*}
L_H+\lambda\colon H^q(\mf{n};C^\infty(\mca{H}))\to H^q(\mf{n};C^\infty(\mca{H})). 
\end{equation*}
So we need to compute the $\mf{an}$-module $H^*(\mf{n};C^\infty(\mca{H}))$. 

First we have 
\begin{align}\label{vvvaaa}
H^0\left(\mf{n};C^\infty(\mca{H})\right)&=C^\infty(\mca{H})^{\pi(\mf{n})}\nonumber\\
&=
\begin{cases}
\bb{C}_0&\mca{H}=\bb{C}\\
0&\text{otherwise}. 
\end{cases}
\end{align}
This can be proved by using a realization of the representation $\mca{H}$. 

To compute $H^1(\mf{n};C^\infty(\mca{H}))$ we use the following computation of the $\mf{n}$-invariant distributions. 

\begin{prop}[Flaminio--Forni \cite{FF}]\label{ffoorr}
We have 
\begin{equation*}
\left(C^\infty(\mca{H})/\pi(\mf{n})C^\infty(\mca{H})\right)^*\simeq
\begin{cases}
\bb{C}_0&\mca{H}=\bb{C}\\
\bb{C}_{-\frac{n}{\sqrt{2}}}&\mca{H}=D_n^\pm\\
\bb{C}_{-\frac{1+\sqrt{1+8\mu}}{2\sqrt{2}}}\oplus\bb{C}_{-\frac{1-\sqrt{1+8\mu}}{2\sqrt{2}}}&\mca{H}=\mca{H}_\mu\quad\left(\mu\neq-\frac{1}{8}\right)\\
\bb{C}^2_\sharp&\mca{H}=\mca{H}_{-\frac{1}{8}}
\end{cases}
\end{equation*}
as $\mf{an}$-modules, where the left hand side denotes the space of all linear maps from $C^\infty(\mca{H})/\pi(\mf{n})C^\infty(\mca{H})$ to $\bb{C}$, and $\bb{C}^2_\sharp=\bb{C}^2$ is a two dimensional representation of $\mf{an}$ defined by 
\begin{equation*}
H\mapsto-\frac{1}{2\sqrt{2}}
\begin{pmatrix}
1&\\
1&1
\end{pmatrix}
, \quad E\mapsto0. 
\end{equation*}
\end{prop}

\begin{rem}
Equip $C^\infty(\mca{H})$ with the $C^\infty$ topology. Consider the space 
\begin{equation*}
\mca{D}(\mca{H})=\left\{\varphi\colon C^\infty(\mca{H})\to\bb{C}\ \middle|\ \text{$\varphi$ is linear and continuous}\right\}
\end{equation*}
of distributions and the space 
\begin{equation*}
\mca{D}_E(\mca{H})=\left\{\varphi\in\mca{D}(\mca{H})\ \middle|\ E\varphi=0\right\}
\end{equation*}
of $E$-invariant distributions, where $E\varphi\in\mca{D}(\mca{H})$ is defined by $(E\varphi)(f)=-\varphi\left(\pi(E)f\right)$ for $f\in C^\infty(\mca{H})$. Then we have 
\begin{equation*}
\mca{D}_E(\mca{H})\simeq\left(C^\infty(\mca{H})/\pi(\mf{n})C^\infty(\mca{H})\right)^*
\end{equation*}
since the computation in Section 3.1 and 3.2 of Flaminio--Forni \cite{FF} shows any linear $\varphi\colon C^\infty(\mca{H})\to\bb{C}$ such that $E\varphi=0$ is automatically continuous. 
\end{rem}

It follows that 
\begin{align*}
H_0\left(\mf{n};C^\infty(\mca{H})\right)&=
C^\infty(\mca{H})/\pi(\mf{n})C^\infty(\mca{H})\\
&\simeq
\begin{cases}
\bb{C}_0&\mca{H}=\bb{C}\\
\bb{C}_\frac{n}{\sqrt{2}}&\mca{H}=D_n^\pm\\
\bb{C}_\frac{1+\sqrt{1+8\mu}}{2\sqrt{2}}\oplus\bb{C}_\frac{1-\sqrt{1+8\mu}}{2\sqrt{2}}&\mca{H}=\mca{H}_\mu\quad\left(\mu\neq-\frac{1}{8}\right)\\
\bb{C}^2_\flat&\mca{H}=\mca{H}_{-\frac{1}{8}}, 
\end{cases}
\end{align*}
where $\bb{C}^2_\flat=\bb{C}^2$ is a two dimensional representation of $\mf{an}$ defined by 
\begin{equation*}
H\mapsto\frac{1}{2\sqrt{2}}
\begin{pmatrix}
1&1\\
&1
\end{pmatrix}
, \quad E\mapsto0. 
\end{equation*}

To go from homology to cohomology, we use the following isomorphism, which is a special case of a more general isomorphism. 

\begin{lem}
We have 
\begin{equation*}
\mf{n}\otimes H^1\left(\mf{n};C^\infty(\mca{H})\right)\simeq H_0\left(\mf{n};C^\infty(\mca{H})\right)
\end{equation*}
as $\mf{an}$-modules. 
\end{lem}

\begin{proof}
Recall that the Chevalley--Eilenberg complex for Lie algebra homology is defined by 
\begin{equation*}
C_*\left(\mf{n};C^\infty(\mca{H})\right)=\bigwedge^*\mf{n}\otimes C^\infty(\mca{H})
\end{equation*}
and 
\begin{align*}
&\partial\left(X_1\wedge\cdots\wedge X_p\otimes f\right)=\sum_{i=1}^p(-1)^{i+1}X_1\wedge\cdots\wedge\widehat{X_i}\wedge\cdots\wedge X_p\otimes X_if\\
&\qquad\qquad+\sum_{i<j}(-1)^{i+j}[X_i,X_j]\wedge X_1\wedge\cdots\wedge\widehat{X_i}\wedge\cdots\wedge\widehat{X_j}\wedge\cdots\wedge X_p\otimes f. 
\end{align*}
Consider the diagram 
\begin{equation*}
\begin{tikzcd}
0&0\\
\mf{n}\otimes C^1\left(\mf{n};C^\infty(\mca{H})\right)\ar[u]\ar[r]&C_0\left(\mf{n};C^\infty(\mca{H})\right)\ar[u]\\
\mf{n}\otimes C^0\left(\mf{n};C^\infty(\mca{H})\right)\ar[u,"\id\otimes d"]\ar[r]&C_1\left(\mf{n};C^\infty(\mca{H})\right),\ar[u,"\partial"]
\end{tikzcd}
\end{equation*}
where the upper horizontal arrow is defined by 
\begin{align*}
\mf{n}\otimes\mf{n}^*\otimes C^\infty(\mca{H})&\stackrel{\sim}{\to}C^\infty(\mca{H})\\
X\otimes\varphi\otimes f&\mapsto\varphi(X)f
\end{align*}
and the lower horizontal arrow is 
\begin{equation*}
\id\colon\mf{n}\otimes C^\infty(\mca{H})\to\mf{n}\otimes C^\infty(\mca{H}). 
\end{equation*}
Then this is commutative and all maps are $\mf{an}$-equivariant. Hence we get the isomorphism of the lemma. 
\end{proof}

Using this we get isomorphisms 
\begin{align*}
H^1\left(\mf{n};C^\infty(\mca{H})\right)&\simeq H_0\left(\mf{n};C^\infty(\mca{H})\right)\otimes\mf{n}^*\\
&\simeq H_0\left(\mf{n};C^\infty(\mca{H})\right)\otimes\bb{R}_{-\frac{1}{\sqrt{2}}}\\
&\simeq H_0\left(\mf{n};C^\infty(\mca{H})\right)\otimes\bb{C}_{-\frac{1}{\sqrt{2}}}
\end{align*}
of $\mf{an}$-modules. Therefore 
\begin{equation*}
H^1\left(\mf{n};C^\infty(\mca{H})\right)\simeq
\begin{cases}
\bb{C}_{-\frac{1}{\sqrt{2}}}&\mca{H}=\bb{C}\\
\bb{C}_{-\frac{1-n}{\sqrt{2}}}&\mca{H}=D_n^\pm\\
\bb{C}_\frac{-1+\sqrt{1+8\mu}}{2\sqrt{2}}\oplus\bb{C}_\frac{-1-\sqrt{1+8\mu}}{2\sqrt{2}}&\mca{H}=\mca{H}_\mu\quad\left(\mu\neq-\frac{1}{8}\right)\\
\bb{C}^2_\natural&\mca{H}=\mca{H}_{-\frac{1}{8}}, 
\end{cases}
\end{equation*}
where $\bb{C}^2_\natural=\bb{C}^2$ is a two dimensional representation of $\mf{an}$ defined by 
\begin{equation*}
H\mapsto
\begin{pmatrix}
-\frac{1}{2\sqrt{2}}&\frac{1}{2\sqrt{2}}\\
&-\frac{1}{2\sqrt{2}}
\end{pmatrix}
, \quad E\mapsto0. 
\end{equation*}
By \eqref{111111} and \eqref{111222}, we have 
\begin{align*}
H^0\left(\mf{an}/\mf{n};H^q(\mf{n};C^\infty(\mca{H})\otimes\bb{C}_\lambda)\right)&\simeq\ker\left(H^q(\mf{n};C^\infty(\mca{H}))\stackrel{L_H+\lambda}{\longrightarrow}H^q(\mf{n};C^\infty(\mca{H}))\right)\otimes\bb{C}_\lambda, \\
H^1\left(\mf{an}/\mf{n};H^q(\mf{n};C^\infty(\mca{H})\otimes\bb{C}_\lambda)\right)&\simeq\coker\left(H^q(\mf{n};C^\infty(\mca{H}))\stackrel{L_H+\lambda}{\longrightarrow}H^q(\mf{n};C^\infty(\mca{H}))\right)\otimes\bb{C}_\lambda. 
\end{align*}
Hence the $E_2$ of the spectral sequence is computed as follows. 

\begin{lem}\label{compres}
\begin{align*}
H^0\left(\mf{an}/\mf{n};H^0\left(\mf{n};C^\infty(\bb{C})\otimes\bb{C}_\lambda\right)\right)&\simeq H^1\left(\mf{an}/\mf{n};H^0\left(\mf{n};C^\infty(\bb{C})\otimes\bb{C}_\lambda\right)\right)\\
&\simeq
\begin{cases}
0&\lambda\neq0\\
\bb{C}&\lambda=0
\end{cases}
\end{align*}
\begin{equation*}
H^0\left(\mf{an}/\mf{n};H^0\left(\mf{n};C^\infty(\mca{H})\otimes\bb{C}_\lambda\right)\right)\simeq H^1\left(\mf{an}/\mf{n};H^0\left(\mf{n};C^\infty(\mca{H})\otimes\bb{C}_\lambda\right)\right)=0
\end{equation*}
for a nontrivial irreducible unitary representation $\mca{H}$
\begin{align*}
H^0\left(\mf{an}/\mf{n};H^1\left(\mf{n};C^\infty(\bb{C})\otimes\bb{C}_\lambda\right)\right)&\simeq H^1\left(\mf{an}/\mf{n};H^1\left(\mf{n};C^\infty(\bb{C})\otimes\bb{C}_\lambda\right)\right)\\
&\simeq
\begin{cases}
0&\lambda\neq\frac{1}{\sqrt{2}}\\
\bb{C}&\lambda=\frac{1}{\sqrt{2}}
\end{cases}
\end{align*}
\begin{align*}
H^0\left(\mf{an}/\mf{n};H^1\left(\mf{n};C^\infty(D_n^\pm)\otimes\bb{C}_\lambda\right)\right)&\simeq H^1\left(\mf{an}/\mf{n};H^1\left(\mf{n};C^\infty(D_n^\pm)\otimes\bb{C}_\lambda\right)\right)\\
&\simeq
\begin{cases}
0&\lambda\neq\frac{1-n}{\sqrt{2}}\\
\bb{C}&\lambda=\frac{1-n}{\sqrt{2}}
\end{cases}
\end{align*}
\begin{align*}
H^0\left(\mf{an}/\mf{n};H^1\left(\mf{n};C^\infty(\mca{H}_\mu)\otimes\bb{C}_\lambda\right)\right)&\simeq H^1\left(\mf{an}/\mf{n};H^1\left(\mf{n};C^\infty(\mca{H}_\mu)\otimes\bb{C}_\lambda\right)\right)\\
&\simeq
\begin{cases}
0&\lambda\neq\frac{1\pm\sqrt{1+8\mu}}{2\sqrt{2}}\\
\bb{C}&\lambda=\frac{1\pm\sqrt{1+8\mu}}{2\sqrt{2}}
\end{cases}
\end{align*}
for $\mu\neq-\frac{1}{8}$
\begin{align*}
H^0\left(\mf{an}/\mf{n};H^1\left(\mf{n};C^\infty\left(\mca{H}_{-\frac{1}{8}}\right)\otimes\bb{C}_\lambda\right)\right)&\simeq H^1\left(\mf{an}/\mf{n};H^1\left(\mf{n};C^\infty\left(\mca{H}_{-\frac{1}{8}}\right)\otimes\bb{C}_\lambda\right)\right)\\
&\simeq
\begin{cases}
0&\lambda\neq\frac{1}{2\sqrt{2}}\\
\bb{C}&\lambda=\frac{1}{2\sqrt{2}}
\end{cases}
\end{align*}
\end{lem}

\subsection{Computation results of $H^*\left(\mf{an};C^\infty(\mca{H})\otimes\bb{C}_\lambda\right)$ for an irreducible unitary representation $\mca{H}$}\label{cohcomp}
By Lemma \ref{compres} and the isomorphisms \eqref{666}, \eqref{777}, \eqref{888}, we get the following computation results of $H^*\left(\mf{an};C^\infty(\mca{H})\otimes\bb{C}_\lambda\right)$ for an irreducible unitary representation $\mca{H}$ and $\lambda\in\bb{C}$.

\subsubsection*{The trivial representation $\bb{C}$}
\begin{equation*}
H^*\left(\mf{an};\bb{C}_\lambda\right)=0\quad\text{for}\ \ \lambda\neq0,\frac{1}{\sqrt{2}}
\end{equation*}
\begin{equation*}
H^i\left(\mf{an};\bb{C}_0\right)\simeq
\begin{cases}
\bb{C}&i=0,1\\
0&\text{otherwise}
\end{cases}
\end{equation*}
\begin{equation*}
H^i\left(\mf{an};\bb{C}_\frac{1}{\sqrt{2}}\right)\simeq
\begin{cases}
\bb{C}&i=1,2\\
0&\text{otherwise}
\end{cases}
\end{equation*}

\subsubsection*{The discrete series representations $D_n^\pm$\ \ $(n\in\bb{Z}_{\geq1})$}
\begin{equation*}
H^*\left(\mf{an};C^\infty(D_n^\pm)\otimes\bb{C}_\lambda\right)=0\quad\text{for}\ \ \lambda\neq\frac{1-n}{\sqrt{2}}
\end{equation*}
\begin{equation*}
H^i\left(\mf{an};C^\infty(D_n^\pm)\otimes\bb{C}_\frac{1-n}{\sqrt{2}}\right)\simeq
\begin{cases}
\bb{C}&i=1,2\\
0&\text{otherwise}
\end{cases}
\end{equation*}

\subsubsection*{The principal and complementary series representations $\mca{H}_\mu$\ \ $(\mu<0)$}
\begin{equation*}
H^*\left(\mf{an};C^\infty(\mca{H}_\mu)\otimes\bb{C}_\lambda\right)=0\quad\text{for}\ \ \lambda\neq\frac{1\pm\sqrt{1+8\mu}}{2\sqrt{2}}
\end{equation*}
\begin{equation*}
H^i\left(\mf{an};C^\infty(\mca{H}_\mu)\otimes\bb{C}_\frac{1\pm\sqrt{1+8\mu}}{2\sqrt{2}}\right)\simeq
\begin{cases}
\bb{C}&i=1,2\\
0&\text{otherwise}
\end{cases}
\end{equation*}

\section{Hodge decomposition of a generic part}\label{firsthodge99}
In this section we will construct an operator $\delta$ of degree $-1$ on $C^*(\mf{an};\mca{H}\otimes V)$ such that 
\begin{equation*}
d\delta+\delta d=\id\otimes\pi(\Omega)\otimes\id+\id\otimes\id\otimes\xi\left(-H^2+\frac{1}{\sqrt{2}}H\right), 
\end{equation*}
where $\sl(2,\bb{R})\stackrel{\pi}{\curvearrowright}\mca{H}$ and $\mf{an}\stackrel{\xi}{\curvearrowright}V$ are representations such that $\xi(\mf{n})=0$, and $\Omega$ is the Casimir element of $\sl(2,\bb{R})$. 

Applying to $\mca{H}=C^\infty(P,\bb{C})$ and $V=\bb{C}_\lambda$, we obtain a ``Hodge decomposition'' of $C^*\left(\mf{an};C^\infty(P,\bb{C})\otimes\bb{C}_\lambda\right)$, which shows most part of (ie a large direct summand of) $C^*\left(\mf{an};C^\infty(P,\bb{C})\otimes\bb{C}_\lambda\right)$ does not contribute to the cohomology. The rest of $C^*\left(\mf{an};C^\infty(P,\bb{C})\otimes\bb{C}_\lambda\right)$ might have a nontrivial cohomology and we will give other ``Hodge decompositions'' to it in Sections \ref{easy} and \ref{discrete series} (except for some cases). 

However, to compute $H^*(\mca{F};\bb{C}_\lambda)$, together with the results in Section \ref{cohcomp}, we need only the Hodge decomposition given in this section, which is used to prove an infinite direct sum (a generic part) has a trivial cohomology. So Section \ref{beginning} completes the computation of $H^*(\mca{F};\bb{C}_\lambda)$, though the computation results are listed in Section \ref{comprrrr}.

\subsection{A construction of an operator $\delta$ on $C^*(\mf{an};\mca{H}\otimes V)$}\label{construc}
In this section we consider the Chevalley--Eilenberg complex $\left(C^*(\mf{an};\mca{H}\otimes V),d\right)$ for representations $\sl(2,\bb{R})\curvearrowright\mca{H}$ and $\mf{an}\curvearrowright V$ such that $\mf{n}$ acts trivially on $V$. We will construct an operator $\delta$ on $C^*(\mf{an};\mca{H}\otimes V)$ which decreases the degree by $1$ such that $d\delta+\delta d$ has a simple form. To make the construction universal, ie independent of the choices of $\mca{H}$ and $V$, we first introduce the algebra 
\begin{equation*}
\mca{A}=C(\mf{an}\oplus\mf{an}^*)\otimes U(\sl(2,\bb{R}))\otimes U(\mf{an}), 
\end{equation*}
where $C(\mf{an}\oplus\mf{an}^*)$ is the Clifford algebra of $\mf{an}\oplus\mf{an}^*$ with respect to the symmetric bilinear form coming from the duality of $\mf{an}$ and $\mf{an}^*$. (A precise definition is given below.) Then the algebra $\mca{A}$ has a natural representation on $C^*(\mf{an};\mca{H}\otimes V)$ and the differential $d$ is expressed as an action of an element $\hat{d}_f$ of $\mca{A}$, which is independent of $\mca{H}$ and $V$. We will define an element $\hat{\delta}_f$ of $\mca{A}$ and $\delta$ as the action of the element $\hat{\delta}_f$ on $C^*(\mf{an};\mca{H}\otimes V)$ through the representation. We compute the element $\hat{d}_f\hat{\delta}_f+\hat{\delta}_f\hat{d}_f$ in $\mca{A}$ and then get a formula for $d\delta+\delta d$ on $C^*(\mf{an};\mca{H}\otimes V)$. 

Let $\mf{an}$ be the Lie algebra of $AN$ and let $\mf{an}^*=\Hom(\mf{an},\bb{R})$. Consider the symmetric bilinear form $(\cdot|\cdot)\colon(\mf{an}\oplus\mf{an}^*)\times(\mf{an}\oplus\mf{an}^*)\to\bb{R}$ defined by 
\begin{equation*}
(X+\varphi|Y+\psi)=\frac{1}{2}\left(\psi(X)+\varphi(Y)\right)
\end{equation*}
for $X$, $Y\in\mf{an}$ and $\varphi$, $\psi\in\mf{an}^*$. Let $C(\mf{an}\oplus\mf{an}^*)$ denote the Clifford algebra of $\left(\mf{an}\oplus\mf{an}^*,(\cdot|\cdot)\right)$. This is defined as the quotient of the tensor algebra over $\bb{R}$ by an ideal: 
\begin{align*}
C(\mf{an}\oplus\mf{an}^*)&=T(\mf{an}\oplus\mf{an}^*)/\left(v\otimes v-(v|v), v\in\mf{an}\oplus\mf{an}^*\right)\\
&=T(\mf{an}\oplus\mf{an}^*)/\left(v\otimes w+w\otimes v-2(v|w), v,w\in\mf{an}\oplus\mf{an}^*\right). 
\end{align*}

Define the maps 
\begin{align*}
\varepsilon(\varphi)\colon\bigwedge^*\mf{an}^*&\to\bigwedge^*\mf{an}^*\\
\psi&\mapsto\varphi\wedge\psi
\end{align*}
for $\varphi\in\mf{an}^*$ and 
\begin{align*}
\iota(X)\colon\bigwedge^*\mf{an}^*&\to\bigwedge^*\mf{an}^*\\
\psi&\mapsto\psi(X,\cdots)
\end{align*}
for $X\in\mf{an}$. 

\begin{lem}\label{cliffrule}
We have $\varepsilon(\varphi)^2=0$, $\iota(X)^2=0$ and $\iota(X)\varepsilon(\varphi)+\varepsilon(\varphi)\iota(X)=\varphi(X)$ for all $\varphi\in\mf{an}^*$ and $X\in\mf{an}$. 
\end{lem}

\begin{proof}
The first two equations are obvious and for the third one, take any $\psi\in\bigwedge^*\mf{an}^*$ and compute as follows: 
\begin{align*}
(\iota(X)\varepsilon(\varphi)+\varepsilon(\varphi)\iota(X))\psi&=\iota(X)(\varphi\wedge\psi)+\varphi\wedge\iota(X)\psi\\
&=\varphi(X)\psi-\varphi\wedge\iota(X)\psi+\varphi\wedge\iota(X)\psi\\
&=\varphi(X)\psi. 
\end{align*}
\end{proof}

Since the linear map 
\begin{align*}
\mf{an}\oplus\mf{an}^*&\to\End\left(\bigwedge^*\mf{an}^*\right)\\
X+\varphi&\mapsto\iota(X)+\varepsilon(\varphi)
\end{align*}
satisfies 
\begin{align*}
(\iota(X)+\varepsilon(\varphi))(\iota(X)+\varepsilon(\varphi))-(X+\varphi|X+\varphi)&=\iota(X)\varepsilon(\varphi)+\varepsilon(\varphi)\iota(X)-\varphi(X)\\
&=0, 
\end{align*}
it extends to an algebra homomorphism 
\begin{equation*}
\sigma\colon C\left(\mf{an}\oplus\mf{an}^*\right)\to\End\left(\bigwedge^*\mf{an}^*\right), 
\end{equation*}
which is called the spin representation of $C\left(\mf{an}\oplus\mf{an}^*\right)$. It is well-known that $\sigma$ is an isomorphism (at least over $\bb{C}$ but the same proof works over $\bb{R}$). 

Consider the $\bb{R}$-algebra 
\begin{equation*}
\mca{A}=C(\mf{an}\oplus\mf{an}^*)\otimes U(\mf{sl}(2,\bb{R}))\otimes U(\mf{an}), 
\end{equation*}
where $U(\sl(2,\bb{R}))$ (resp. $U(\mf{an})$) is the universal enveloping algebra of $\sl(2,\bb{R})$ (resp. $\mf{an}$) over $\bb{R}$. Let $\sl(2,\bb{R})\stackrel{\pi}{\curvearrowright}\mca{H}$ and $\mf{an}\stackrel{\xi}{\curvearrowright}V$ be representations such that $\xi(\mf{n})=0$. Then we get a representation 
\begin{equation*}
\mca{A}\stackrel{\sigma\otimes\pi\otimes\xi}{\curvearrowright}\bigwedge^*\mf{an}^*\otimes\mca{H}\otimes V=C^*\left(\mf{an};\mca{H}\otimes V\right). 
\end{equation*}
Recall the elements $H$ and $E$ of $\mf{an}$ in Section \ref{notation}. Let $\theta_H$, $\theta_E$ be the dual basis of $H$, $E$ in $\mf{an}^*$. Let $d$ (resp. $d_\star$) be the differential of $C^*\left(\mf{an};\mca{H}\otimes V\right)$ (resp. $C^*(\mf{an};\bb{R})$, where $\bb{R}$ is the trivial representation of $\mf{an}$). Note that we have $C^*\left(\mf{an};\mca{H}\otimes V\right)=C^*(\mf{an};\bb{R})\otimes\mca{H}\otimes V$. Then the differential $d$ is expressed as follows: 
\begin{align}\label{dformula}
d&=d_\star\otimes\id\otimes\id+\varepsilon(\theta_H)\otimes(\pi(H)\otimes\id+\id\otimes\xi(H))\nonumber\\
&\qquad\qquad\qquad\ +\varepsilon(\theta_E)\otimes(\pi(E)\otimes\id+\id\otimes\xi(E))\nonumber\\
&=d_\star\otimes\id\otimes\id+\varepsilon(\theta_H)\otimes(\pi(H)\otimes\id+\id\otimes\xi(H))+\varepsilon(\theta_E)\otimes\pi(E)\otimes\id. 
\end{align}
See for example Maruhashi--Tsutaya \cite{MTIII}. Let $\hat{d}_\star$ be the unique element of $C(\mf{an}\oplus\mf{an}^*)$ such that $\sigma(\hat{d}_\star)=d_\star$. Consider the elements 
\begin{equation*}
\hat{d}=\hat{d}_\star\otimes1\otimes1+\theta_H\otimes(H\otimes1+1\otimes H)+\theta_E\otimes(E\otimes1+1\otimes E)
\end{equation*}
and 
\begin{equation*}
\hat{d}_f=\hat{d}_\star\otimes1\otimes1+\theta_H\otimes(H\otimes1+1\otimes H)+\theta_E\otimes E\otimes1
\end{equation*}
of $\mca{A}$. These elements represent the differential $d$: 
\begin{gather*}
(\sigma\otimes\pi\otimes\xi)(\hat{d})=d,\quad(\sigma\otimes\pi\otimes\xi)(\hat{d}_f)=d. 
\end{gather*}
They are independent of the choice of $\mca{H}$ and $V$. We call $\hat{d}_f$ the fake differential. 

Recall the inner product $B_\theta$ of $\sl(2,\bb{R})$ defined in Section \ref{notation}. Its restriction to $\mf{an}$ gives a linear isomorphism $\kappa\colon\mf{an}\to\mf{an}^*$. Consider a map 
\begin{align*}
\mf{an}\oplus\mf{an}^*&\to\mf{an}\oplus\mf{an}^*\\
X+\varphi&\mapsto\kappa^{-1}(\varphi)+\kappa(X). 
\end{align*}
It extends to the unique antihomomorphism 
\begin{equation*}
T(\mf{an}\oplus\mf{an}^*)\to C(\mf{an}\oplus\mf{an}^*)
\end{equation*}
of algebras. Since we have 
\begin{align*}
&(\kappa^{-1}(\varphi)+\kappa(X))(\kappa^{-1}(\varphi)+\kappa(X))\\
&=\kappa^{-1}(\varphi)\kappa(X)+\kappa(X)\kappa^{-1}(\varphi)=\kappa(X)(\kappa^{-1}(\varphi))\\
&=B_\theta\left(\kappa^{-1}(\varphi),X\right)=\varphi(X)=(X+\varphi|X+\varphi), 
\end{align*}
the antihomomorphism factors through $C(\mf{an}\oplus\mf{an}^*)$ and we get an involutive antiautomorphism 
\begin{equation*}
\cdot^\top\colon C(\mf{an}\oplus\mf{an}^*)\to C(\mf{an}\oplus\mf{an}^*). 
\end{equation*}
Since $H$, $E$ is an orthonormal basis of $\mf{an}$, we have $\kappa(H)=\theta_H$, $\kappa(E)=\theta_E$, hence $H^\top=\theta_H$, $E^\top=\theta_E$ in $C(\mf{an}\oplus\mf{an}^*)$. 

The inner product of $\mf{an}$ induces a natural inner product of $\bigwedge^*\mf{an}^*$. For a linear map $T\colon\bigwedge^*\mf{an}^*\to\bigwedge^*\mf{an}^*$, let $T^\top\colon\bigwedge^*\mf{an}^*\to\bigwedge^*\mf{an}^*$ denote the adjoint of $T$ with respect to the inner product. 

\begin{lem}
We have $\sigma(x)^\top=\sigma\left(x^\top\right)$ for all $x\in C(\mf{an}\oplus\mf{an}^*)$. 
\end{lem}

\begin{proof}
See Maruhashi--Tsutaya \cite{MTIII}. 
\end{proof}

We define an element $\hat{\delta}_f$ of $\mca{A}$ by 
\begin{equation*}
\hat{\delta}_f=\hat{d}_\star^\top\otimes1\otimes1+H\otimes\left(\left(H+\frac{1}{\sqrt{2}}\right)\otimes1+1\otimes(-H)\right)+E\otimes2F\otimes1
\end{equation*}
and an operator $\delta$ on $C^*(\mf{an};\mca{H}\otimes V)$ by 
\begin{align*}
\delta&=(\sigma\otimes\pi\otimes\xi)(\hat{\delta}_f)\\
&=d_\star^\top\otimes\id\otimes\id+\iota(H)\otimes\left(\left(\pi(H)+\frac{1}{\sqrt{2}}\right)\otimes\id-\id\otimes\xi(H)\right)+\iota(E)\otimes2\pi(F)\otimes\id. 
\end{align*}
The operator $\delta$ decreases the degree by $1$. 

\begin{rem}
The $C(\mf{an}\oplus\mf{an}^*)$-components of $\hat{\delta}_f$ are obtained from those of $\hat{d}_f$ by applying $\cdot^\top$. The $U(\sl(2,\bb{R}))$-components of $\hat{\delta}_f$ appear in the expression 
\begin{equation*}
\Omega=2FE+\left(H+\frac{1}{\sqrt{2}}\right)H
\end{equation*}
of the Casimir element of $\sl(2,\bb{R})$. 
\end{rem}

If $\mca{H}^\prime$ is an $\sl(2,\bb{R})$-invariant subspace of $\mca{H}$, then $C^*(\mf{an};\mca{H}^\prime\otimes V)$ is a subcomplex of $C^*(\mf{an};\mca{H}\otimes V)$ and invariant under $\delta$. 

For $X\in\mf{an}$, let $L_X\colon\bigwedge^*\mf{an}^*\to\bigwedge^*\mf{an}^*$ be the Lie derivative by $X$ defined in Section \ref{hochschildserre}. Let $\hat{L}_X$ be the unique element of $C(\mf{an}\oplus\mf{an}^*)$ such that $\sigma(\hat{L}_X)=L_X$. 

\begin{thm}\label{mtrm}
We have 
\begin{align*}
\hat{d}_f\hat{\delta}_f+\hat{\delta}_f\hat{d}_f&=\left(\hat{d}_\star\hat{d}_\star^\top+\hat{d}_\star^\top\hat{d}_\star+\frac{1}{\sqrt{2}}\hat{L}_H\right)\otimes1\otimes1\\
&\quad+1\otimes\Omega\otimes1+1\otimes1\otimes\left(-H^2+\frac{1}{\sqrt{2}}H\right)
\end{align*}
in $\mca{A}$, where $\Omega\in U(\sl(2,\bb{R}))$ is the Casimir element of $\sl(2,\bb{R})$. 
\end{thm}

Before starting the proof we list some properties we use in the proof: 
\begin{itemize}
\item We have 
\begin{equation}\label{liederivativei}
\hat{d}_\star X+X\hat{d}_\star=\hat{L}_X\quad\text{and}\quad\hat{d}_\star^\top\kappa(X)+\kappa(X)\hat{d}_\star^\top=\hat{L}_X^\top
\end{equation}
in $C(\mf{an}\oplus\mf{an}^*)$ for all $X\in\mf{an}$. 
\item We have 
\begin{equation}\label{liederivativefo}
\hat{L}_X=-\theta_H[X,H]-\theta_E[X,E]
\end{equation}
in $C(\mf{an}\oplus\mf{an}^*)$ for all $X\in\mf{an}$. 
\item Since $\hat{L}_H=-\frac{1}{\sqrt{2}}\theta_EE$, we have 
\begin{equation}\label{symmmm}
\hat{L}_H^\top=\hat{L}_H. 
\end{equation}
\end{itemize}
See Maruhashi--Tsutaya \cite{MTIII} for the proofs of the first and second. 

\begin{proof}[Proof of Theorem \ref{mtrm}]
By the definitions 
\begin{gather*}
\hat{d}_f=\hat{d}_\star\otimes1\otimes1+\theta_H\otimes H\otimes1+\theta_H\otimes1\otimes H+\theta_E\otimes E\otimes1,\\
\hat{\delta}_f=\hat{d}_\star^\top\otimes1\otimes1+H\otimes\left(H+\frac{1}{\sqrt{2}}\right)\otimes1+H\otimes1\otimes(-H)+E\otimes2F\otimes1
\end{gather*}
and \eqref{liederivativei}, \eqref{symmmm}, we have 
\begin{align*}
&\hat{d}_f\hat{\delta}_f+\hat{\delta}_f\hat{d}_f\\
&=\left(\hat{d}_\star\hat{d}_\star^\top+\hat{d}_\star^\top\hat{d}_\star\right)\otimes1\otimes1+\textcolor{green}{\hat{L}_H\otimes\left(H+\frac{1}{\sqrt{2}}\right)\otimes1}+\textcolor{blue}{\hat{L}_H\otimes1\otimes(-H)}\\
&\quad+\hat{L}_E\otimes2F\otimes1+\textcolor{green}{\hat{L}_H^\top\otimes H\otimes1}+\textcolor{blue}{\hat{L}_H^\top\otimes1\otimes H}+\hat{L}_E^\top\otimes E\otimes1\\
&\quad+1\otimes H\left(H+\frac{1}{\sqrt{2}}\right)\otimes1+\textcolor{red}{1\otimes H\otimes(-H)}+\theta_HE\otimes2[H,F]\otimes1\\
&\quad+1\otimes\left(\textcolor{red}{H}+\frac{1}{\sqrt{2}}\right)\otimes H+1\otimes1\otimes(-H^2)+\theta_EH\otimes[E,H]\otimes1\\
&\quad+\theta_EE\otimes2[E,F]\otimes1+1\otimes2FE\otimes1\\
&=\left(\hat{d}_\star\hat{d}_\star^\top+\hat{d}_\star^\top\hat{d}_\star\right)\otimes1\otimes1+\hat{L}_H\otimes\left(2H+\frac{1}{\sqrt{2}}\right)\otimes1\\
&\quad+\hat{L}_E\otimes2F\otimes1+\hat{L}_E^\top\otimes E\otimes1\\
&\quad+1\otimes H\left(H+\frac{1}{\sqrt{2}}\right)\otimes1+1\otimes\frac{1}{\sqrt{2}}\otimes H+1\otimes1\otimes(-H^2)+1\otimes2FE\otimes1\\
&\quad+\theta_HE\otimes2[H,F]\otimes1+\theta_EH\otimes[E,H]\otimes1+\theta_EE\otimes2[E,F]\otimes1. 
\end{align*}
By the formulas 
\begin{gather*}
\hat{L}_H=-\frac{1}{\sqrt{2}}\theta_EE,\quad\hat{L}_E=\frac{1}{\sqrt{2}}\theta_HE,\\
\Omega=2FE+H^2+\frac{1}{\sqrt{2}}H,\\
[H,E]=\frac{1}{\sqrt{2}}E,\quad[H,F]=-\frac{1}{\sqrt{2}}F,\quad[E,F]=\frac{1}{\sqrt{2}}H, 
\end{gather*}
we get 
\begin{align*}
&\hat{d}_f\hat{\delta}_f+\hat{\delta}_f\hat{d}_f\\
&=\left(\hat{d}_\star\hat{d}_\star^\top+\hat{d}_\star^\top\hat{d}_\star\right)\otimes1\otimes1-\frac{1}{\sqrt{2}}\theta_EE\otimes\left(2H+\frac{1}{\sqrt{2}}\right)\otimes1\\
&\quad+\frac{1}{\sqrt{2}}\theta_HE\otimes2F\otimes1+\frac{1}{\sqrt{2}}\theta_EH\otimes E\otimes1\\
&\quad+1\otimes\Omega\otimes1+1\otimes1\otimes\left(-H^2+\frac{1}{\sqrt{2}}H\right)\\
&\quad+\theta_HE\otimes\left(-\sqrt{2}F\right)\otimes1+\theta_EH\otimes\left(-\frac{1}{\sqrt{2}}E\right)\otimes1+\theta_EE\otimes\sqrt{2}H\otimes1\\
&=\left(\hat{d}_\star\hat{d}_\star^\top+\hat{d}_\star^\top\hat{d}_\star\right)\otimes1\otimes1-\frac{1}{\sqrt{2}}\theta_EE\otimes\frac{1}{\sqrt{2}}\otimes1\\
&\quad+1\otimes\Omega\otimes1+1\otimes1\otimes\left(-H^2+\frac{1}{\sqrt{2}}H\right)\\
&=\left(\hat{d}_\star\hat{d}_\star^\top+\hat{d}_\star^\top\hat{d}_\star+\frac{1}{\sqrt{2}}\hat{L}_H\right)\otimes1\otimes1+1\otimes\Omega\otimes1+1\otimes1\otimes\left(-H^2+\frac{1}{\sqrt{2}}H\right). 
\end{align*}
\end{proof}

\begin{prop}\label{yuuui}
We have 
\begin{equation*}
\hat{d}_\star\hat{d}_\star^\top+\hat{d}_\star^\top\hat{d}_\star+\frac{1}{\sqrt{2}}\hat{L}_H=0
\end{equation*}
in $C(\mf{an}\oplus\mf{an}^*)$. 
\end{prop}

\begin{proof}
The element $\hat{d}_\star$ can be expressed as follows: 
\begin{equation}\label{dfor}
\hat{d}_\star=-\theta_H\theta_E[H,E]=-\frac{1}{\sqrt{2}}\theta_H\theta_EE. 
\end{equation}
See Maruhashi--Tsutaya \cite{MTIII}. Hence we have 
\begin{align*}
2\left(\hat{d}_\star\hat{d}_\star^\top+\hat{d}_\star^\top\hat{d}_\star\right)&=\theta_H\theta_EE\theta_EEH+\theta_EEH\theta_H\theta_EE\\
&=\theta_H\theta_EEH-\theta_EE\theta_HH\theta_EE+\theta_EE\theta_EE\\
&=\theta_H\theta_EEH+\theta_E\theta_HEH\theta_EE+\theta_EE\\
&=\theta_H\theta_EEH-\theta_E\theta_HE\theta_EHE+\theta_EE\\
&=\theta_H\theta_EEH-\theta_E\theta_HHE+\theta_EE\\
&=\theta_EE\\
&=-\sqrt{2}\hat{L}_H. 
\end{align*}
\end{proof}

By Theorem \ref{mtrm} and Proposition \ref{yuuui} we get the following theorem. 

\begin{thm}\label{mmtthh}
We have 
\begin{equation*}
\hat{d}_f\hat{\delta}_f+\hat{\delta}_f\hat{d}_f=1\otimes\Omega\otimes1+1\otimes1\otimes\left(-H^2+\frac{1}{\sqrt{2}}H\right)
\end{equation*}
in $\mca{A}$, hence 
\begin{equation*}
d\delta+\delta d=\id\otimes\pi(\Omega)\otimes\id+\id\otimes\id\otimes\left(-\xi(H)^2+\frac{1}{\sqrt{2}}\xi(H)\right)
\end{equation*}
on $C^*(\mf{an};\mca{H}\otimes V)$. 
\end{thm}

\begin{prop}\label{squarezero}
We have $\hat{\delta}_f^2=0$ as an element of $\mca{A}$, hence $\delta^2=0$ on $C^*(\mf{an};\mca{H}\otimes V)$. 
\end{prop}

In the proof we use the following properties: 
\begin{itemize}
\item We have $\hat{d}_\star^2=0$ in $C(\mf{an}\oplus\mf{an}^*)$. 
\item There is an injective algebra homomorphism 
\begin{equation*}
j\colon\bigwedge^*\mf{an}^*\hookrightarrow C(\mf{an}\oplus\mf{an}^*)
\end{equation*}
which is the identity on $\mf{an}^*$. 
\item We have 
\begin{equation}\label{dphi}
\hat{d}_\star\varphi+\varphi\hat{d}_\star=j(d_\star\varphi)
\end{equation}
in $C(\mf{an}\oplus\mf{an}^*)$ for all $\varphi\in\mf{an}^*$. See Maruhashi--Tsutaya \cite{MTIII}. 
\item We have 
\begin{equation}\label{fffooo}
d_\star\theta_H=0,\quad d_\star\theta_E=-\frac{1}{\sqrt{2}}\theta_H\wedge\theta_E. 
\end{equation}
\end{itemize}

\begin{proof}[Proof of Proposition \ref{squarezero}]
By \eqref{dphi} and \eqref{fffooo} we get 
\begin{equation*}
\hat{d}_\star^\top H+H\hat{d}_\star^\top=0,\quad\hat{d}_\star^\top E+E\hat{d}_\star^\top=-\frac{1}{\sqrt{2}}EH. 
\end{equation*}
Using the definition 
\begin{equation*}
\hat{\delta}_f=\hat{d}_\star^\top\otimes1\otimes1+H\otimes\left(H+\frac{1}{\sqrt{2}}\right)\otimes1+H\otimes1\otimes(-H)+E\otimes2F\otimes1, 
\end{equation*}
we have 
\begin{equation*}
\hat{\delta}_f^2=-\frac{1}{\sqrt{2}}EH\otimes2F\otimes1+HE\otimes2[H,F]\otimes1=0. 
\end{equation*}
\end{proof}

\begin{rem}
We can also show $\hat{d}^2=0$ and $\hat{d}_f^2=0$ in $\mca{A}$. 
\end{rem}

\subsection{Hodge decomposition of $C^*(\mf{an};\mca{H}\otimes\bb{C}_\lambda)$ for generic $\mca{H}$}\label{Hooddg}
In this section we give a ``Hodge decomposition'' of $C^*(\mf{an};\mca{H}\otimes\bb{C}_\lambda)$ for ``generic'' $\mca{H}$ by using the operator $\delta$ constructed in Section \ref{construc}. 

Recall that for $\lambda\in\bb{C}$ the representation $\mf{an}\curvearrowright\bb{C}_\lambda$ is defined by 
\begin{gather*}
\bb{C}_\lambda=\bb{C},\quad H1=\lambda1,\quad E1=0. 
\end{gather*}
Let $\sl(2,\bb{R})\stackrel{\pi}{\curvearrowright}\mca{H}$ be a representation and consider the Chevalley--Eilenberg complex $\left(C^*(\mf{an};\mca{H}\otimes\bb{C}_\lambda),d\right)$. Then by Theorem \ref{mmtthh} we have an operator $\delta$ on $C^*(\mf{an};\mca{H}\otimes\bb{C}_\lambda)$ of degree $-1$ such that 
\begin{align*}
d\delta+\delta d&=\id\otimes\pi(\Omega)\otimes\id+\id\otimes\id\otimes\left(-\lambda^2+\frac{1}{\sqrt{2}}\lambda\right)\\
&=\id\otimes\left(\pi(\Omega)-\lambda^2+\frac{1}{\sqrt{2}}\lambda\right)\otimes\id. 
\end{align*}
Assume that the operator 
\begin{equation*}
\pi(\Omega)-\lambda^2+\frac{1}{\sqrt{2}}\lambda\colon\mca{H}\to\mca{H}
\end{equation*}
is bijective. Put 
$Q=\id\otimes\left(\pi(\Omega)-\lambda^2+\frac{1}{\sqrt{2}}\lambda\right)\otimes\id$, so that $d\delta+\delta d=Q$. Then we have $dQ=Qd$ and $\delta Q=Q\delta$ since $\Omega$ lies in the center of $U(\sl(2,\bb{R}))$. Hence we get $dQ^{-1}=Q^{-1}d$ and $\delta Q^{-1}=Q^{-1}\delta$. Put 
\begin{equation*}
\delta^\prime=Q^{-1}\delta\colon C^*(\mf{an};\mca{H}\otimes\bb{C}_\lambda)\to C^*(\mf{an};\mca{H}\otimes\bb{C}_\lambda). 
\end{equation*}
So we have 
\begin{gather*}
d\delta^\prime+\delta^\prime d=\id,\quad(\delta^\prime)^2=0,\\
\im\delta^\prime=\im\delta,\quad\ker\delta^\prime=\ker\delta, 
\end{gather*}
which imply 
\begin{gather*}
C^*(\mf{an};\mca{H}\otimes\bb{C}_\lambda)=\im d\oplus\im\delta^\prime,\\
\ker d=\im d,\quad\ker\delta^\prime=\im\delta^\prime,\\
H^*(\mf{an};\mca{H}\otimes\bb{C}_\lambda)=0. 
\end{gather*}

When the operator $\pi(\Omega)-\lambda^2+\frac{1}{\sqrt{2}}\lambda$ is not necessarily bijective on $\mca{H}$, let 
\begin{equation*}
\mca{H}^0=\ker\left(\pi(\Omega)-\lambda^2+\frac{1}{\sqrt{2}}\lambda\right). 
\end{equation*}
Then $\mca{H}^0$ is an $\sl(2,\bb{R})$-invariant subspace of $\mca{H}$. Assume that there exists an $\sl(2,\bb{R})$-invariant subspace $\mca{H}^\prime$ of $\mca{H}$ such that $\mca{H}=\mca{H}^0\oplus\mca{H}^\prime$ and the operator 
\begin{equation*}
\pi(\Omega)-\lambda^2+\frac{1}{\sqrt{2}}\lambda\colon\mca{H}^\prime\to\mca{H}^\prime
\end{equation*}
is bijective. Then we have 
\begin{equation*}
C^*(\mf{an};\mca{H}\otimes\bb{C}_\lambda)=C^*(\mf{an};\mca{H}^0\otimes\bb{C}_\lambda)\oplus C^*(\mf{an};\mca{H}^\prime\otimes\bb{C}_\lambda)
\end{equation*}
and $H^*(\mf{an};\mca{H}^\prime\otimes\bb{C}_\lambda)=0$. Hence 
\begin{align*}
H^*(\mf{an};\mca{H}\otimes\bb{C}_\lambda)&=H^*(\mf{an};\mca{H}^0\otimes\bb{C}_\lambda)\oplus H^*(\mf{an};\mca{H}^\prime\otimes\bb{C}_\lambda)\\
&=H^*(\mf{an};\mca{H}^0\otimes\bb{C}_\lambda). 
\end{align*}
Note that we have $d\delta+\delta d=0$ on $C^*(\mf{an};\mca{H}^0\otimes\bb{C}_\lambda)$.

\subsection{Computation of $H^*(\mca{F};\bb{C}_\lambda)$}\label{beginning}
To compute the cohomology $H^*\left(\mca{F};\bb{C}_\lambda\right)$ for $\lambda\in\bb{C}$, we use the isomorphism 
\begin{equation*}
H^*\left(\mca{F};\bb{C}_\lambda\right)\simeq H^*\left(\mf{an};C^\infty(P,\bb{C})\otimes_\bb{C}\bb{C}_\lambda\right)
\end{equation*}
in Proposition \ref{drce}. Note that $C^\infty(P,\bb{C})$ is not just a representation of $\mf{an}$ but also that of $\sl(2,\bb{R})$. Let $\sl(2,\bb{R})\stackrel{\pi}{\curvearrowright}C^\infty(P,\bb{C})$ denote the representation. We apply things in Section \ref{Hooddg}. First we put 
\begin{equation*}
C^\infty(P,\bb{C})^0=\left\{f\in C^\infty(P,\bb{C})\ \middle|\ \pi(\Omega)f=\left(\lambda^2-\frac{1}{\sqrt{2}}\lambda\right)f\right\}
\end{equation*}
following a notation in Section \ref{Hooddg}. 

Recall that we have a unitary representation $\PSL(2,\bb{R})\curvearrowright L^2(P)$. By \eqref{decompooo} in Section \ref{decompl2} we have a decomposition 
\begin{equation*}
L^2(P)=\bigoplus_{\mu\in\Upsilon}W_\mu=W_{\lambda^2-\frac{1}{\sqrt{2}}\lambda}\oplus\bigoplus_{\mu\in\Upsilon_\lambda}W_\mu, 
\end{equation*}
where 
\begin{equation*}
\Upsilon=\left\{-\frac{\nu}{2}\ \middle|\ \nu\in\sigma(\ol{\Delta_\Sigma})\right\}\cup\left\{\frac{n(n-1)}{2}\ \middle|\ n\in\bb{Z}_{\geq1}\right\}, 
\end{equation*}
\begin{equation*}
\Upsilon_\lambda=\Upsilon\setminus\left\{\lambda^2-\frac{1}{\sqrt{2}}\lambda\right\}
\end{equation*}
and $W_\mu\subset L^2(P)$ is the sum of the irreducible unitary subrepresentations of $L^2(P)$ with Casimir parameter $\mu$. By taking the space of $C^\infty$ vectors of both sides, we get a decomposition 
\begin{equation*}
C^\infty(P,\bb{C})=C^\infty\left(W_{\lambda^2-\frac{1}{\sqrt{2}}\lambda}\right)\oplus C^\infty\left(\bigoplus_{\mu\in\Upsilon_\lambda}W_\mu\right)
\end{equation*}
into $\sl(2,\bb{R})$-invariant subspaces. We use this as the decomposition $\mca{H}=\mca{H}^0\oplus\mca{H}^\prime$ in Section \ref{Hooddg}. First note that we have 
\begin{equation*}
C^\infty(P,\bb{C})^0=C^\infty\left(W_{\lambda^2-\frac{1}{\sqrt{2}}\lambda}\right). 
\end{equation*}
So we need to prove the following to apply results in Section \ref{Hooddg}. 

\begin{lem}\label{biject}
The operator 
\begin{equation*}
\pi(\Omega)-\lambda^2+\frac{1}{\sqrt{2}}\lambda\colon C^\infty\left(\bigoplus_{\mu\in\Upsilon_\lambda}W_\mu\right)\to C^\infty\left(\bigoplus_{\mu\in\Upsilon_\lambda}W_\mu\right)
\end{equation*}
is bijective. 
\end{lem}

To prove this lemma we use the following description of the space of $C^\infty$ vectors. 

\begin{lem}\label{smooooth}
For $i\in\bb{Z}_{\geq1}$ let $\PSL(2,\bb{R})\curvearrowright\mca{H}_i$ be a unitary representation and $\PSL(2,\bb{R})\stackrel{\pi}{\curvearrowright}\mca{H}=\bigoplus_{i=1}^\infty\mca{H}_i$ be the direct sum. Then we have 
\begin{equation*}
C^\infty(\mca{H})=\left\{\sum_{i=1}^\infty f_i\ \middle|\ f_i\in C^\infty(\mca{H}_i),\ \sum_{i=1}^\infty\lVert Xf_i\rVert^2<\infty\ \ \text{for all $X\in U(\sl(2,\bb{R}))$}\right\}. 
\end{equation*}
Moreover, if $f\in C^\infty(\mca{H})$ and $f=\sum_{i=1}^\infty f_i$ in $\mca{H}$ with $f_i\in\mca{H}_i$, we have $Xf=\sum_{i=1}^\infty Xf_i$ in $\mca{H}$ for all $X\in U(\sl(2,\bb{R}))$. 
\end{lem}

\begin{proof}
Let $p_i\colon\mca{H}\to\mca{H}_i$ be the projection. Let $f\in C^\infty(\mca{H})$ and $f=\sum_{i=1}^\infty f_i$ in $\mca{H}$ with $f_i\in\mca{H}_i$. Since $p_i$ is bounded and $\PSL(2,\bb{R})$-equivariant, the component $f_i=p_i(f)$ is an element of $C^\infty(\mca{H}_i)$. Take $X\in\sl(2,\bb{R})$. We have $Xf=\sum_{i=1}^\infty g_i$ for some $g_i\in\mca{H}_i$. Since $\pi$ is continuous, we have $\pi(e^{tX})f=\sum_{i=1}^\infty\pi(e^{tX})f_i$ for $t\in\bb{R}$. Hence 
\begin{equation*}
\sum_{i=1}^\infty g_i=Xf=\lim_{t\to0}\frac{\pi(e^{tX})f-f}{t}=\lim_{t\to0}\sum_{i=1}^\infty\frac{\pi(e^{tX})f_i-f_i}{t}. 
\end{equation*}
Applying $p_i$ to this equation, we get 
\begin{equation*}
g_i=\lim_{t\to0}\frac{\pi(e^{tX})f_i-f_i}{t}=Xf_i. 
\end{equation*}
Thus $Xf=\sum_{i=1}^\infty Xf_i$. We see that this holds for all $X\in U(\sl(2,\bb{R}))$, by repeating the argument. In particular, we have $\sum_{i=1}^\infty\lVert Xf_i\rVert^2<\infty$. 

Conversely let $f=\sum_{i=1}^\infty f_i\in\mca{H}$ be such that $f_i\in C^\infty(\mca{H}_i)$ and $\sum_{i=1}^\infty\lVert Xf_i\rVert^2<\infty$ for all $X\in U(\sl(2,\bb{R}))$. Take $X\in\sl(2,\bb{R})$. By assumption we have an element $\sum_{i=1}^\infty Xf_i\in\mca{H}$. We prove that $Xf$ exists and $Xf=\sum_{i=1}^\infty Xf_i$. First we have 
\begin{equation*}
\frac{\pi(e^{tX})f-f}{t}-\sum_{i=1}^\infty Xf_i=\sum_{i=1}^\infty\left(\frac{\pi(e^{tX})f_i-f_i}{t}-Xf_i\right). 
\end{equation*}
Put $c(t)=\pi(e^{tX})f_i$. Then $c\colon\bb{R}\to\mca{H}_i$ is a $C^\infty$ map. The fundamental theorem of calculus holds for functions with values in a Fr\'{e}chet space, hence a Hilbert space in particular. See for example 2.2.2 of Hamilton \cite{Ham1}. So we get the following form of Taylor's theorem by integration by parts: 
\begin{align*}
c(t)-c(0)&=\int_0^tc^\prime(s)ds=tc^\prime(t)-\int_0^tsc^{\prime\prime}(s)ds\\
&=tc^\prime(0)+t\left(c^\prime(t)-c^\prime(0)\right)-\int_0^tsc^{\prime\prime}(s)ds\\
&=tc^\prime(0)+t\int_0^tc^{\prime\prime}(s)ds-\int_0^tsc^{\prime\prime}(s)ds\\
&=tc^\prime(0)+\int_0^t(t-s)c^{\prime\prime}(s)ds. 
\end{align*}
Thus we have 
\begin{align*}
\left\lVert\frac{\pi(e^{tX})f_i-f_i}{t}-Xf_i\right\rVert&=\left\lVert\int_0^t\left(1-\frac{s}{t}\right)\pi(e^{sX})X^2f_ids\right\rVert\\
&\leq\left\lvert\int_0^t\left(1-\frac{s}{t}\right)\left\lVert X^2f_i\right\rVert ds\right\rvert\\
&=\left\lVert X^2f_i\right\rVert\left\lvert\left[s-\frac{s^2}{2t}\right]_0^t\right\rvert\\
&=\left\lVert X^2f_i\right\rVert\frac{\lvert t\rvert}{2}. 
\end{align*}
Therefore, 
\begin{align*}
\left\lVert\frac{\pi(e^{tX})f-f}{t}-\sum_{i=1}^\infty Xf_i\right\rVert^2&=\sum_{i=1}^\infty\left\lVert\frac{\pi(e^{tX})f_i-f_i}{t}-Xf_i\right\rVert^2\\
&\leq\frac{t^2}{4}\sum_{i=1}^\infty\left\lVert X^2f_i\right\rVert^2\to0\quad(t\to0). 
\end{align*}
Hence we proved $Xf=\sum_{i=1}^\infty Xf_i$. By repeating the argument, we see that this holds for all $X\in U(\sl(2,\bb{R}))$. From this we can show that the map 
\begin{align*}
\PSL(2,\bb{R})&\to\mca{H}\\
g&\mapsto\pi(g)f
\end{align*}
is a $C^\infty$ map, hence $f\in C^\infty(\mca{H})$. 
\end{proof}

\begin{proof}[Proof of Lemma \ref{biject}]
We have 
\begin{equation*}
C^\infty\left(\bigoplus_{\mu\in\Upsilon_\lambda}W_\mu\right)=\left\{\sum_\mu f_\mu\ \middle|\ 
\begin{gathered}
f_\mu\in C^\infty(W_\mu),\\
\sum_\mu\lVert Xf_\mu\rVert^2<\infty\ \text{for all $X\in U(\sl(2,\bb{R}))$}
\end{gathered}
\right\}. 
\end{equation*}
by Lemma \ref{smooooth} and 
\begin{align*}
\pi(\Omega)-\lambda^2+\frac{1}{\sqrt{2}}\lambda\colon C^\infty\left(\bigoplus_{\mu\in\Upsilon_\lambda}W_\mu\right)&\to C^\infty\left(\bigoplus_{\mu\in\Upsilon_\lambda}W_\mu\right)\\
\sum_\mu f_\mu&\mapsto\sum_\mu\left(\mu-\lambda^2+\frac{1}{\sqrt{2}}\lambda\right)f_\mu. 
\end{align*}
Since the set $\Upsilon$ has no accumulation points in $\bb{R}$, the quantity $\left(\mu-\lambda^2+\frac{1}{\sqrt{2}}\lambda\right)^{-1}$ is bounded for $\mu\in\Upsilon_\lambda$. By this and the description of $C^\infty\left(\bigoplus_{\mu\in\Upsilon_\lambda}W_\mu\right)$, we can define an operator 
\begin{align*}
C^\infty\left(\bigoplus_{\mu\in\Upsilon_\lambda}W_\mu\right)&\to C^\infty\left(\bigoplus_{\mu\in\Upsilon_\lambda}W_\mu\right)\\
\sum_\mu f_\mu&\mapsto\sum_\mu\left(\mu-\lambda^2+\frac{1}{\sqrt{2}}\lambda\right)^{-1}f_\mu, 
\end{align*}
which is the inverse of $\pi(\Omega)-\lambda^2+\frac{1}{\sqrt{2}}\lambda$. 
\end{proof}

Therefore, by Section \ref{Hooddg}, we obtain the following. 

\begin{thm}\label{36kwlxisee}
For any $\lambda\in\bb{C}$, we have a decomposition 
\begin{equation*}
C^*(\mf{an};C^\infty(P,\bb{C})\otimes\bb{C}_\lambda)=C_\lambda\oplus C^*\left(\mf{an};C^\infty\left(\bigoplus_{\mu\in\Upsilon_\lambda}W_\mu\right)\otimes\bb{C}_\lambda\right)
\end{equation*}
of the cochain complex, where we put 
\begin{equation*}
C_\lambda=C^*\left(\mf{an};C^\infty\left(W_{\lambda^2-\frac{1}{\sqrt{2}}\lambda}\right)\otimes\bb{C}_\lambda\right)
\end{equation*}
and a ``Hodge decomposition'' 
\begin{equation*}
C^*\left(\mf{an};C^\infty\left(\bigoplus_{\mu\in\Upsilon_\lambda}W_\mu\right)\otimes\bb{C}_\lambda\right)=\im d\oplus\im\delta. 
\end{equation*}
Hence 
\begin{equation*}
H^*\left(\mf{an};C^\infty\left(\bigoplus_{\mu\in\Upsilon_\lambda}W_\mu\right)\otimes\bb{C}_\lambda\right)=0
\end{equation*}
and 
\begin{equation*}
H^*(\mf{an};C^\infty(P,\bb{C})\otimes\bb{C}_\lambda)=H^*(C_\lambda)=H^*\left(\mf{an};C^\infty\left(W_{\lambda^2-\frac{1}{\sqrt{2}}\lambda}\right)\otimes\bb{C}_\lambda\right). 
\end{equation*}
\end{thm}

Let us see when $C_\lambda=0$. Put $\Lambda=\left\{\lambda\in\bb{C}\ \middle|\ \lambda^2-\frac{1}{\sqrt{2}}\lambda\in\Upsilon\right\}$. Then 
\begin{equation*}
C_\lambda=0\quad\Leftrightarrow\quad W_{\lambda^2-\frac{1}{\sqrt{2}}\lambda}=0\quad\Leftrightarrow\quad\lambda^2-\frac{1}{\sqrt{2}}\lambda\not\in\Upsilon\quad\Leftrightarrow\quad\lambda\not\in\Lambda. 
\end{equation*}
Recall 
\begin{equation*}
\Upsilon=\left\{-\frac{\nu}{2}\ \middle|\ \nu\in\sigma(\ol{\Delta_\Sigma})\right\}\cup\left\{\frac{n(n-1)}{2}\ \middle|\ n\in\bb{Z}_{\geq1}\right\}. 
\end{equation*}
Since 
\begin{equation*}
\lambda^2-\frac{1}{\sqrt{2}}\lambda=-\frac{\nu}{2}\quad\Leftrightarrow\quad\lambda=\frac{1\pm\sqrt{1-4\nu}}{2\sqrt{2}}, 
\end{equation*}
\begin{equation*}
\lambda^2-\frac{1}{\sqrt{2}}\lambda=\frac{n(n-1)}{2}\quad\Leftrightarrow\quad\lambda=\frac{n}{\sqrt{2}},\frac{1-n}{\sqrt{2}}, 
\end{equation*}
we have 
\begin{equation*}
\Lambda=\left\{\frac{1\pm\sqrt{1-4\nu}}{2\sqrt{2}}\ \middle|\ \nu\in\sigma(\ol{\Delta_\Sigma})\right\}\cup\frac{1}{\sqrt{2}}\bb{Z}. 
\end{equation*}

\begin{cor}
We have 
\begin{equation*}
H^*(\mca{F};\bb{C}_\lambda)\simeq H^*(\mf{an};C^\infty(P,\bb{C})\otimes\bb{C}_\lambda)=0
\end{equation*}
for any 
\begin{equation*}
\lambda\not\in\Lambda=\left\{\frac{1\pm\sqrt{1-4\nu}}{2\sqrt{2}}\ \middle|\ \nu\in\sigma(\ol{\Delta_\Sigma})\right\}\cup\frac{1}{\sqrt{2}}\bb{Z}. 
\end{equation*}
\end{cor}

Let us give a decomposition of $C_\lambda=C^*\left(\mf{an};C^\infty\left(W_{\lambda^2-\frac{1}{\sqrt{2}}\lambda}\right)\otimes\bb{C}_\lambda\right)$ for $\lambda\in\Lambda$. By Section \ref{conclu}, we have 
\begin{align*}
C_0&=C^*\left(\mf{an};C^\infty(W_0)\right)\\
&\simeq C^*(\mf{an};\bb{C})\oplus g_\Sigma C^*(\mf{an};C^\infty(D_1^+))\oplus g_\Sigma C^*(\mf{an};C^\infty(D_1^-))
\end{align*}
and 
\begin{align*}
C_\frac{1}{\sqrt{2}}&=C^*\left(\mf{an};C^\infty(W_0)\otimes\bb{C}_\frac{1}{\sqrt{2}}\right)\\
&\simeq C^*\left(\mf{an};\bb{C}_\frac{1}{\sqrt{2}}\right)\oplus g_\Sigma C^*\left(\mf{an};C^\infty(D_1^+)\otimes\bb{C}_\frac{1}{\sqrt{2}}\right)\oplus g_\Sigma C^*\left(\mf{an};C^\infty(D_1^-)\otimes\bb{C}_\frac{1}{\sqrt{2}}\right), 
\end{align*}
where $g_\Sigma$ is the genus of the surface $\Sigma=\Gamma\backslash\PSL(2,\bb{R})/\PSO(2)$. For $n\in\bb{Z}_{\geq2}$, we have 
\begin{align*}
C_\frac{n}{\sqrt{2}}&=C^*\left(\mf{an};C^\infty\left(W_\frac{n(n-1)}{2}\right)\otimes\bb{C}_\frac{n}{\sqrt{2}}\right)\\
&\simeq(2n-1)(g_\Sigma-1)\left(C^*\left(\mf{an};C^\infty(D_n^+)\otimes\bb{C}_\frac{n}{\sqrt{2}}\right)\oplus C^*\left(\mf{an};C^\infty(D_n^-)\otimes\bb{C}_\frac{n}{\sqrt{2}}\right)\right)
\end{align*}
and 
\begin{align}
C_\frac{1-n}{\sqrt{2}}&=C^*\left(\mf{an};C^\infty\left(W_\frac{n(n-1)}{2}\right)\otimes\bb{C}_\frac{1-n}{\sqrt{2}}\right)\nonumber\\
&\simeq(2n-1)(g_\Sigma-1)\left(C^*\left(\mf{an};C^\infty(D_n^+)\otimes\bb{C}_\frac{1-n}{\sqrt{2}}\right)\oplus C^*\left(\mf{an};C^\infty(D_n^-)\otimes\bb{C}_\frac{1-n}{\sqrt{2}}\right)\right). \label{ywolxi}
\end{align}
For $\nu\in\sigma(\ol{\Delta_\Sigma})\setminus\{0\}$, we have 
\begin{align*}
C_\frac{1\pm\sqrt{1-4\nu}}{2\sqrt{2}}&=C^*\left(\mf{an};C^\infty\left(W_{-\frac{\nu}{2}}\right)\otimes\bb{C}_\frac{1\pm\sqrt{1-4\nu}}{2\sqrt{2}}\right)\\
&\simeq m_\nu C^*\left(\mf{an};C^\infty\left(\mca{H}_{-\frac{\nu}{2}}\right)\otimes\bb{C}_\frac{1\pm\sqrt{1-4\nu}}{2\sqrt{2}}\right), 
\end{align*}
where $m_\nu$ is the multiplicity of the eigenvalue $\nu$ of $\ol{\Delta_\Sigma}$. 

By these decompositions of $C_\lambda$ for $\lambda\in\Lambda$ and the cohomology computation results in Section \ref{cohcomp}, we can now compute $H^*(\mca{F};\bb{C}_\lambda)$ for all $\lambda\in\Lambda$. The results are listed in Section \ref{comprrrr}. However, we can not just compute the cohomology but also give ``Hodge decompositions'' of the cochain complexes for most $\lambda\in\bb{C}$. So before going to Section \ref{comprrrr}, we will give ``Hodge decompositions'' that are not covered by Section \ref{Hooddg}, in the next two sections. 

Among the complexes $C^*(\mf{an};\bb{C}_\lambda)$, those that are not covered by the Hodge decomposition in Section \ref{Hooddg} are 
\begin{equation*}
C^*(\mf{an};\bb{C})\quad\text{and}\quad C^*\left(\mf{an};\bb{C}_\frac{1}{\sqrt{2}}\right). 
\end{equation*}
Constructions of Hodge decompositions for these are easy and given in Section \ref{easy}. 

Among $C^*\left(\mf{an};C^\infty(D_n^\pm)\otimes\bb{C}_\lambda\right)$, those that are not covered by the Hodge decompositions in Section \ref{Hooddg} are 
\begin{equation*}
C^*\left(\mf{an};C^\infty(D_n^\pm)\otimes\bb{C}_\frac{n}{\sqrt{2}}\right)\quad\text{and}\quad C^*\left(\mf{an};C^\infty(D_n^\pm)\otimes\bb{C}_\frac{1-n}{\sqrt{2}}\right). 
\end{equation*}
Hodge decompositions for these complexes will be constructed in Section \ref{discrete series}. 

Among $C^*\left(\mf{an};C^\infty\left(\mca{H}_{-\frac{\nu}{2}}\right)\otimes\bb{C}_\lambda\right)$ for $\nu<0$, those that are not covered by the Hodge decompositions in Section \ref{Hooddg} are 
\begin{equation*}
C^*\left(\mf{an};C^\infty\left(\mca{H}_{-\frac{\nu}{2}}\right)\otimes\bb{C}_\frac{1\pm\sqrt{1-4\nu}}{2\sqrt{2}}\right). 
\end{equation*}
But we did not find a ``Hodge decomposition'' for these complexes.

\section{Hodge decompositions of $C^*(\mf{an};\bb{C})$ and $C^*\left(\mf{an};\bb{C}_\frac{1}{\sqrt{2}}\right)$}\label{easy}
In this section we construct ``Hodge decompositions'' of $C^*(\mf{an};\bb{C})$ and $C^*\left(\mf{an};\bb{C}_\frac{1}{\sqrt{2}}\right)$. Since they are finite dimensional, the constructions are easy. We omit the detailed calculations in this section. 

Note that a ``Hodge decomposition'' of $C^*(\mf{an};\bb{C}_\lambda)$ for $\lambda\neq0,\frac{1}{\sqrt{2}}$ is already given in Section \ref{Hooddg}.

\subsection{$C^*(\mf{an};\bb{C})$}
Consider the following basis of $C^*(\mf{an};\bb{C})$: 
\begin{gather*}
1,\ \ \ \theta_H,\ \ \ \theta_E,\ \ \ \theta_H\wedge\theta_E. 
\end{gather*}
We have 
\begin{gather*}
d1=0,\quad d\theta_H=0,\quad d\theta_E=-\frac{1}{\sqrt{2}}\theta_H\wedge\theta_E,\quad d(\theta_H\wedge\theta_E)=0
\end{gather*}
for the differential $d$ of $C^*(\mf{an};\bb{C})$. Hence 
\begin{equation*}
d_0=
\begin{pmatrix}
0\\
0
\end{pmatrix}
,\quad d_1=
\begin{pmatrix}
0&-\frac{1}{\sqrt{2}}
\end{pmatrix}
, 
\end{equation*}
where $d_i$ is the differential on $C^i(\mf{an};\bb{C})$. Let $\delta_i\colon C^i(\mf{an};\bb{C})\to C^{i-1}(\mf{an};\bb{C})$ and $p_i\colon C^i(\mf{an};\bb{C})\to C^i(\mf{an};\bb{C})$ be linear maps. Then $\delta_i$ and $p_i$ satisfy 
\begin{gather*}
d\delta+\delta d=\id-p,\quad p^2=p,\quad dp=0,\quad\delta p=0,\quad pd=0,\quad p\delta=0,\quad\delta^2=0
\end{gather*}
if and only if 
\begin{gather*}
\delta_1=
\begin{pmatrix}
0&0
\end{pmatrix}
,\quad\delta_2=
\begin{pmatrix}
a\\
-\sqrt{2}
\end{pmatrix}
,\\
p_0=
\begin{pmatrix}
1
\end{pmatrix}
,\quad p_1=
\begin{pmatrix}
1&\frac{1}{\sqrt{2}}a\\
0&0
\end{pmatrix}
,\quad p_2=
\begin{pmatrix}
0
\end{pmatrix}
\end{gather*}
for some $a\in\bb{C}$. Using these $\delta$ and $p$, we get a ``Hodge decomposition'' of $C^*(\mf{an};\bb{C})$: 
\begin{align*}
C^*(\mf{an};\bb{C})&=\im p\oplus\im d\oplus\im\delta\\
\ker d&=\im p\oplus\im d\\
\ker\delta&=\im p\oplus\im\delta\\
\ker d\cap\ker\delta&=\im p\\
\im d\oplus\im\delta&=\ker p. 
\end{align*}
Hence 
\begin{equation*}
H^i(\mf{an};\bb{C})\simeq\im p_i=
\begin{cases}
\bb{C}1&i=0\\
\bb{C}\theta_H&i=1\\
0&\text{otherwise}. 
\end{cases}
\end{equation*}

As in Section \ref{construc} it is possible to give a degree free expression of $\delta$ by using the Clifford algebra. Let $C(\mf{an}\oplus\mf{an}^*)$ be the Clifford algebra of $\mf{an}\oplus\mf{an}^*$ and $\sigma$ be the representation $C(\mf{an}\oplus\mf{an}^*)\otimes\bb{C}\curvearrowright C^*(\mf{an};\bb{C})$ obtained from the spin representation $C(\mf{an}\oplus\mf{an}^*)\curvearrowright C^*(\mf{an};\bb{R})$. Define elements $\hat{\delta}$ and $\hat{p}$ of $C(\mf{an}\oplus\mf{an}^*)\otimes\bb{C}$ by 
\begin{align*}
\hat{\delta}&=\left(\sqrt{2}\theta_E-a\theta_H\right)HE\\
\hat{p}&=1-\frac{1}{\sqrt{2}}\left(\sqrt{2}\theta_E-a\theta_H\right)E. 
\end{align*}
Then we have $\sigma(\hat{\delta})=\delta$ and $\sigma(\hat{p})=p$.

\subsection{$C^*\left(\mf{an};\bb{C}_\frac{1}{\sqrt{2}}\right)$}
Consider the following basis of $C^*\left(\mf{an};\bb{C}_\frac{1}{\sqrt{2}}\right)$: 
\begin{gather*}
1,\ \ \ \theta_H,\ \ \ \theta_E,\ \ \ \theta_H\wedge\theta_E. 
\end{gather*}
We have 
\begin{gather*}
d1=\frac{1}{\sqrt{2}}\theta_H,\quad d\theta_H=0,\quad d\theta_E=0,\quad d(\theta_H\wedge\theta_E)=0
\end{gather*}
for the differential $d$ of $C^*\left(\mf{an};\bb{C}_\frac{1}{\sqrt{2}}\right)$. Hence 
\begin{equation*}
d_0=
\begin{pmatrix}
\frac{1}{\sqrt{2}}\\
0
\end{pmatrix}
,\quad d_1=
\begin{pmatrix}
0&0
\end{pmatrix}
, 
\end{equation*}
where $d_i$ is the differential on $C^i\left(\mf{an};\bb{C}_\frac{1}{\sqrt{2}}\right)$. Let $\delta_i\colon C^i\left(\mf{an};\bb{C}_\frac{1}{\sqrt{2}}\right)\to C^{i-1}\left(\mf{an};\bb{C}_\frac{1}{\sqrt{2}}\right)$ and $p_i\colon C^i\left(\mf{an};\bb{C}_\frac{1}{\sqrt{2}}\right)\to C^i\left(\mf{an};\bb{C}_\frac{1}{\sqrt{2}}\right)$ be linear maps. Then $\delta_i$ and $p_i$ satisfy 
\begin{gather*}
d\delta+\delta d=\id-p,\quad p^2=p,\quad dp=0,\quad\delta p=0,\quad pd=0,\quad p\delta=0,\quad\delta^2=0
\end{gather*}
if and only if 
\begin{gather*}
\delta_1=
\begin{pmatrix}
\sqrt{2}&a
\end{pmatrix}
,\quad\delta_2=
\begin{pmatrix}
0\\
0
\end{pmatrix}
,\\
p_0=
\begin{pmatrix}
0
\end{pmatrix}
,\quad p_1=
\begin{pmatrix}
0&-\frac{1}{\sqrt{2}}a\\
0&1
\end{pmatrix}
,\quad p_2=
\begin{pmatrix}
1
\end{pmatrix}
\end{gather*}
for some $a\in\bb{C}$. Using these $\delta$ and $p$, we get a ``Hodge decomposition'' of $C^*\left(\mf{an};\bb{C}_\frac{1}{\sqrt{2}}\right)$: 
\begin{align*}
C^*\left(\mf{an};\bb{C}_\frac{1}{\sqrt{2}}\right)&=\im p\oplus\im d\oplus\im\delta\\
\ker d&=\im p\oplus\im d\\
\ker\delta&=\im p\oplus\im\delta\\
\ker d\cap\ker\delta&=\im p\\
\im d\oplus\im\delta&=\ker p. 
\end{align*}
Hence 
\begin{equation*}
H^i\left(\mf{an};\bb{C}_\frac{1}{\sqrt{2}}\right)\simeq\im p_i=
\begin{cases}
\bb{C}\left(-\frac{1}{\sqrt{2}}a\theta_H+\theta_E\right)&i=1\\
\bb{C}(\theta_H\wedge\theta_E)&i=2\\
0&\text{otherwise}. 
\end{cases}
\end{equation*}

Let $\sigma$ be the representation $C(\mf{an}\oplus\mf{an}^*)\otimes\bb{C}\curvearrowright C^*\left(\mf{an};\bb{C}_\frac{1}{\sqrt{2}}\right)$ obtained from the spin representation $C(\mf{an}\oplus\mf{an}^*)\curvearrowright C^*(\mf{an};\bb{R})$. Define elements $\hat{\delta}$ and $\hat{p}$ of $C(\mf{an}\oplus\mf{an}^*)\otimes\bb{C}$ by 
\begin{align*}
\hat{\delta}&=HE\left(\sqrt{2}\theta_E-a\theta_H\right)\\
\hat{p}&=\frac{1}{\sqrt{2}}\left(\sqrt{2}\theta_E-a\theta_H\right)E. 
\end{align*}
Then $\sigma(\hat{\delta})=\delta$ and $\sigma(\hat{p})=p$.

\section{Hodge decomposition of $C^*(\mf{an};C^\infty(D_n^\pm)\otimes\bb{C}_\lambda)$}\label{discrete series}
In this section we consider the cochain complex $C^*(\mf{an};C^\infty(D_n^\pm)\otimes\bb{C}_\lambda)$ and give a ``Hodge decomposition'' of it. There are three sections here but Section \ref{1-n2n} is just a restatement of the Hodge decomposition given in Section \ref{Hooddg} and not necessary because Sections \ref{1-n2} and \ref{1-n22} cover all the cases. We included Section \ref{1-n2n} just to compare it with Section \ref{1-n2} and Section \ref{1-n22}.

\subsection{$\lambda\neq\frac{1-n}{\sqrt{2}},\frac{n}{\sqrt{2}}$}\label{1-n2n}
First we consider the case when $\lambda\neq\frac{1-n}{\sqrt{2}},\frac{n}{\sqrt{2}}$. We just apply things in Section \ref{construc} and Section \ref{Hooddg}. Let $\sl(2,\bb{R})\stackrel{\pi}{\curvearrowright}C^\infty(D_n^\pm)$ denote the derivative of the discrete series representation. Consider the operator
\begin{align*}
\delta&=d_\star^\top\otimes\id\otimes\id+\iota(H)\otimes\left(\pi(H)+\frac{1}{\sqrt{2}}-\lambda\right)\otimes\id+\iota(E)\otimes2\pi(F)\otimes\id
\end{align*}
of degree $-1$ on $C^*(\mf{an};C^\infty(D_n^\pm)\otimes\bb{C}_\lambda)$, constructed in Section \ref{construc}. Since 
\begin{align*}
\pi(\Omega)-\lambda^2+\frac{1}{\sqrt{2}}\lambda&=\frac{n(n-1)}{2}-\lambda^2+\frac{1}{\sqrt{2}}\lambda\\
&=-\left(\lambda-\frac{1-n}{\sqrt{2}}\right)\left(\lambda-\frac{n}{\sqrt{2}}\right)\neq0, 
\end{align*}
the operator 
\begin{equation*}
\pi(\Omega)-\lambda^2+\frac{1}{\sqrt{2}}\lambda\colon C^\infty(D_n^\pm)\to C^\infty(D_n^\pm)
\end{equation*}
is bijective. So by Section \ref{Hooddg} we have 
\begin{equation*}
C^*(\mf{an};C^\infty(D_n^\pm)\otimes\bb{C}_\lambda)=\im d\oplus\im\delta
\end{equation*}
and 
\begin{equation*}
H^*(\mf{an};C^\infty(D_n^\pm)\otimes\bb{C}_\lambda)=0. 
\end{equation*}

\subsection{$\lambda\neq\frac{1-n}{\sqrt{2}}$}\label{1-n2}
Here we will construct a ``Hodge decomposition'' of $C^*\left(\mf{an};C^\infty(D_n^\pm)\otimes\bb{C}_\lambda\right)$ when $\lambda\neq\frac{1-n}{\sqrt{2}}$. So a newly covered case is $\lambda=\frac{n}{\sqrt{2}}$. Even for already covered cases, the construction here is different from the one in Sections \ref{construc} and \ref{Hooddg}. 

First we consider a bit more general cochain complex $C^*(\mf{an};\mca{H}\otimes V)$ just to make the argument clearer and apply the results to $C^*\left(\mf{an};C^\infty(D_n^\pm)\otimes\bb{C}_\lambda\right)$ later.

\subsubsection{A general consideration}\label{geneconsss}
Let $\sl(2,\bb{R})\stackrel{\pi}{\curvearrowright}\mca{H}$ be a representation on a complex vector space $\mca{H}$. Assume that $\pi(E)\colon\mca{H}\to\mca{H}$ is injective and there exists a linear map $\varphi\colon\mca{H}\to\bb{C}$ such that $\ker\varphi=\pi(E)\mca{H}$. Let $\mf{an}\stackrel{\xi}{\curvearrowright}V$ be a real (resp. complex) representation such that $\xi(\mf{n})=0$. Consider the tensor product $\mf{an}\curvearrowright\mca{H}\otimes_\bb{R}V$ (resp. $\mca{H}\otimes_\bb{C}V$) of the representations. We will construct a ``Hodge decomposition'' of the Chevalley--Eilenberg complex $\left(C^*(\mf{an};\mca{H}\otimes V),d\right)$ assuming one more condition which amounts to $\lambda\neq\frac{1-n}{\sqrt{2}}$. 

Let $\mca{H}^*=\Hom_\bb{C}(\mca{H},\bb{C})$ be the complex dual and $\sl(2,\bb{R})\stackrel{\pi}{\curvearrowright}\mca{H}^*$ be the representation defined by $(\pi(X)\psi)f=-\psi(\pi(X)f)$ for $X\in\sl(2,\bb{R})$, $\psi\in\mca{H}^*$ and $f\in\mca{H}$. Consider the zeroth cohomology 
\begin{align*}
H^0(\mf{n};\mca{H}^*)&=\left\{\psi\in\mca{H}^*\ \middle|\ \pi(E)\psi=0\right\}\\
&=\left\{\psi\in\mca{H}^*\ \middle|\ \psi(\pi(E)\mca{H})=0\right\}. 
\end{align*}
Since the subspace $\pi(E)\mca{H}$ is of codimension $1$ or $0$ in $\mca{H}$ by the assumption $\ker\varphi=\pi(E)\mca{H}$, we have $H^0(\mf{n};\mca{H}^*)=\bb{C}\varphi$. Since $\mf{n}=\bb{R}E$ is an ideal of $\mf{an}$, the representation $\sl(2,\bb{R})\stackrel{\pi}{\curvearrowright}\mca{H}^*$ restricts to a representation $\mf{an}\curvearrowright H^0(\mf{n};\mca{H}^*)$. It follows that $\pi(H)\varphi=\nu\varphi$ for some $\nu\in\bb{C}$. 

For any $f\in\mca{H}$, we have 
\begin{equation*}
\varphi\left((\pi(H)+\nu)f\right)=-\left(\pi(H)\varphi\right)(f)+\nu\varphi(f)=0. 
\end{equation*}
So we can consider an element $\pi(E)^{-1}(\pi(H)+\nu)f\in\mca{H}$ by the assumptions. Define an operator $S\colon\mca{H}\to\mca{H}$ by 
\begin{equation*}
S=\pi(E)^{-1}(\pi(H)+\nu). 
\end{equation*}

Consider the operator 
\begin{equation}\label{deltaneq}
\delta=\iota(H)\otimes\id\otimes\id-\iota(E)\otimes S\otimes\id
\end{equation}
on $C^*(\mf{an};\mca{H}\otimes V)=\bigwedge^*\mf{an}^*\otimes\mca{H}\otimes V$ of degree $-1$. We have $\delta^2=0$. To compute $d\delta+\delta d$ we need some formulas. First we obviously have $\pi(E)S=\pi(H)+\nu$. 

\begin{lem}\label{spie}
We have $S\pi(E)=\pi(H)+\nu+\frac{1}{\sqrt{2}}$. 
\end{lem}

\begin{proof}
We have 
\begin{align*}
S\pi(E)&=\pi(E)^{-1}(\pi(H)+\nu)\pi(E)=\pi(E)^{-1}\left(\pi(H)\pi(E)+\nu\pi(E)\right)\\
&=\pi(E)^{-1}\left(\pi\left([H,E]\right)+\pi(E)\pi(H)+\nu\pi(E)\right)\\
&=\pi(E)^{-1}\left(\pi\left(\frac{1}{\sqrt{2}}E\right)+\pi(E)\pi(H)+\nu\pi(E)\right)\\
&=\frac{1}{\sqrt{2}}+\pi(H)+\nu. 
\end{align*}
\end{proof}

Thus 
\begin{equation}\label{piesb}
[\pi(E),S]=-\frac{1}{\sqrt{2}}. 
\end{equation}

Since $\mf{n}=\bb{R}E$ is an ideal of $\mf{an}$, the representation $\sl(2,\bb{R})\stackrel{\pi}{\curvearrowright}\mca{H}$ restricts to $\mf{an}\curvearrowright\pi(E)\mca{H}$. Hence the operator $\pi(E)^{-1}\pi(H)$ is defined on $\pi(E)\mca{H}$. 

\begin{lem}\label{commmmp}
We have $\pi(E)^{-1}\pi(H)-\pi(H)\pi(E)^{-1}=\frac{1}{\sqrt{2}}\pi(E)^{-1}$ on $\pi(E)\mca{H}$. 
\end{lem}

\begin{proof}
For $f\in\mca{H}$, put $f^\prime=\pi(E)f$. We have 
\begin{equation*}
\pi(H)\pi(E)f=\pi(E)\pi(H)f+\frac{1}{\sqrt{2}}\pi(E)f, 
\end{equation*}
hence 
\begin{equation*}
\pi(E)^{-1}\pi(H)\pi(E)f=\pi(H)f+\frac{1}{\sqrt{2}}f, 
\end{equation*}
that is, 
\begin{equation*}
\pi(E)^{-1}\pi(H)f^\prime=\pi(H)\pi(E)^{-1}f^\prime+\frac{1}{\sqrt{2}}\pi(E)^{-1}f^\prime. 
\end{equation*}
\end{proof}

\begin{lem}\label{pihsb}
We have $[\pi(H),S]=-\frac{1}{\sqrt{2}}S$. 
\end{lem}

\begin{proof}
By Lemma \ref{commmmp}, we have 
\begin{align*}
\pi(H)S&=\pi(H)\pi(E)^{-1}(\pi(H)+\nu)\\
&=\left(\pi(E)^{-1}\pi(H)-\frac{1}{\sqrt{2}}\pi(E)^{-1}\right)(\pi(H)+\nu)\\
&=\pi(E)^{-1}\left(\pi(H)^2+\left(\nu-\frac{1}{\sqrt{2}}\right)\pi(H)-\frac{1}{\sqrt{2}}\nu\right). 
\end{align*}
Since $S\pi(H)=\pi(E)^{-1}(\pi(H)+\nu)\pi(H)$, we get 
\begin{equation*}
[\pi(H),S]=\pi(E)^{-1}\left(-\frac{1}{\sqrt{2}}\pi(H)-\frac{1}{\sqrt{2}}\nu\right)=-\frac{1}{\sqrt{2}}S. 
\end{equation*}
\end{proof}

Recall that the differential $d$ of $C^*(\mf{an};\mca{H}\otimes V)$ can be expressed as 
\begin{equation*}
d=d_\star\otimes\id\otimes\id+\varepsilon(\theta_H)\otimes\pi(H)\otimes\id+\varepsilon(\theta_H)\otimes\id\otimes\xi(H)+\varepsilon(\theta_E)\otimes\pi(E)\otimes\id
\end{equation*}
by the assumption $\xi(E)=0$, where $d_\star$ is the differential of $C^*(\mf{an};\bb{R})$. See \eqref{dformula} in Section \ref{construc}. 

\begin{thm}\label{ddeltadeltadiii}
We have 
\begin{equation*}
d\delta+\delta d=\id\otimes\id\otimes\xi(H)-\left(\nu+\frac{1}{\sqrt{2}}\right)\id. 
\end{equation*}
\end{thm}

\begin{proof}
Recalling the definition 
\begin{equation*}
\delta=\iota(H)\otimes\id\otimes\id-\iota(E)\otimes S\otimes\id
\end{equation*}
and using Lemmas \ref{cliffrule}, \ref{spie} and \ref{pihsb}, Formulas \eqref{liederivativei} and \eqref{liederivativefo} in Section \ref{construc} and Equation \eqref{piesb}, we get 
\begin{align*}
d\delta+\delta d&=L_H\otimes\id\otimes\id+\id\otimes\pi(H)\otimes\id+\id\otimes\id\otimes\xi(H)-L_E\otimes S\otimes\id\\
&\quad-\varepsilon(\theta_H)\iota(E)\otimes[\pi(H),S]\otimes\id-\varepsilon(\theta_E)\iota(E)\otimes\pi(E)S\otimes\id\\
&\quad+\left(\varepsilon(\theta_E)\iota(E)-\id\right)\otimes S\pi(E)\otimes\id\\
&=\left(-\frac{1}{\sqrt{2}}\varepsilon(\theta_E)\iota(E)\right)\otimes\id\otimes\id+\id\otimes\pi(H)\otimes\id+\id\otimes\id\otimes\xi(H)\\
&\quad-\frac{1}{\sqrt{2}}\varepsilon(\theta_H)\iota(E)\otimes S\otimes\id-\varepsilon(\theta_H)\iota(E)\otimes\left(-\frac{1}{\sqrt{2}}S\right)\otimes\id\\
&\quad-\varepsilon(\theta_E)\iota(E)\otimes[\pi(E),S]\otimes\id-\id\otimes\left(\pi(H)+\nu+\frac{1}{\sqrt{2}}\right)\otimes\id\\
&=\id\otimes\id\otimes\xi(H)-\left(\nu+\frac{1}{\sqrt{2}}\right)\id. 
\end{align*}
\end{proof}

If $\xi(H)$ is a scalar multiplication, we have 
\begin{equation*}
d\delta+\delta d=\left(\xi(H)-\nu-\frac{1}{\sqrt{2}}\right)\id. 
\end{equation*}
Consequently we get the following theorem. 

\begin{thm}
Let $\sl(2,\bb{R})\stackrel{\pi}{\curvearrowright}\mca{H}$ be a representation on a complex vector space $\mca{H}$ such that $\pi(E)\colon\mca{H}\to\mca{H}$ is injective and there exists a linear map $\varphi\colon\mca{H}\to\bb{C}$ with $\ker\varphi=\pi(E)\mca{H}$. Take $\nu\in\bb{C}$ satisfying $\pi(H)\varphi=\nu\varphi$. Let $\mf{an}\stackrel{\xi}{\curvearrowright}V$ be a real (resp. complex) representation such that $\xi(E)=0$ and $\xi(H)$ acts as a scalar multiplication satisfying $\xi(H)\neq\nu+\frac{1}{\sqrt{2}}$. Define an operator $\delta$ on $C^*(\mf{an};\mca{H}\otimes V)$ by \eqref{deltaneq}. Then we have 
\begin{equation*}
C^*(\mf{an};\mca{H}\otimes V)=\im d\oplus\im\delta
\end{equation*}
by 
\begin{equation*}
x=\left(\xi(H)-\nu-\frac{1}{\sqrt{2}}\right)^{-1}d\delta x+\left(\xi(H)-\nu-\frac{1}{\sqrt{2}}\right)^{-1}\delta dx
\end{equation*}
and 
\begin{equation*}
\ker d=\im d,\quad\ker\delta=\im\delta. 
\end{equation*}
Hence $H^*(\mf{an};\mca{H}\otimes V)=0$. 
\end{thm}

\subsubsection{Application to $C^*\left(\mf{an};C^\infty(D_n^\pm)\otimes\bb{C}_\lambda\right)$}
Let $\sl(2,\bb{R})\stackrel{\pi}{\curvearrowright}C^\infty(D_n^\pm)$ be the derivative of the discrete series representation. The operator $\pi(E)$ is injective by \eqref{vvvaaa} in Section \ref{tnidoiu}. By Proposition \ref{ffoorr}, there exists a linear map $\varphi\colon C^\infty(D_n^\pm)\to\bb{C}$ such that $\ker\varphi=\pi(E)C^\infty(D_n^\pm)$ and $\pi(H)\varphi=-\frac{n}{\sqrt{2}}\varphi$. So $\nu=-\frac{n}{\sqrt{2}}$ and we define an operator 
\begin{equation*}
S=\pi(E)^{-1}\left(\pi(H)-\frac{n}{\sqrt{2}}\right)
\end{equation*}
on $C^\infty(D_n^\pm)$ and an operator 
\begin{equation*}
\delta=\iota(H)\otimes\id\otimes\id-\iota(E)\otimes S\otimes\id
\end{equation*}
on $C^*\left(\mf{an};C^\infty(D_n^\pm)\otimes\bb{C}_\lambda\right)$. By Theorem \ref{ddeltadeltadiii} we have 
\begin{equation*}
d\delta+\delta d=\left(\lambda-\frac{1-n}{\sqrt{2}}\right)\id. 
\end{equation*}

\begin{thm}\label{hodge2}
If $\lambda\neq\frac{1-n}{\sqrt{2}}$, we have 
\begin{equation*}
C^*\left(\mf{an};C^\infty(D_n^\pm)\otimes\bb{C}_\lambda\right)=\im d\oplus\im\delta
\end{equation*}
by 
\begin{equation*}
x=\left(\lambda-\frac{1-n}{\sqrt{2}}\right)^{-1}d\delta x+\left(\lambda-\frac{1-n}{\sqrt{2}}\right)^{-1}\delta dx
\end{equation*}
and 
\begin{equation*}
\ker d=\im d,\quad\ker\delta=\im\delta. 
\end{equation*}
In particular, $H^*\left(\mf{an};C^\infty(D_n^\pm)\otimes\bb{C}_\lambda\right)=0$. 
\end{thm}

\subsection{$\lambda=\frac{1-n}{\sqrt{2}}$}\label{1-n22}
Finally we will construct a ``Hodge decomposition'' of $C^*\left(\mf{an};C^\infty(D_n^\pm)\otimes\bb{C}_{\frac{1-n}{\sqrt{2}}}\right)$. In this case the cohomology is nontrivial as we saw in Section \ref{cohcomp}.

\subsubsection{A general consideration}\label{2generc}
Let $\sl(2,\bb{R})\stackrel{\pi}{\curvearrowright}\mca{H}$ be a representation on a complex vector space $\mca{H}$. Assume that $\pi(E)\colon\mca{H}\to\mca{H}$ is injective and there exist a linear map $\varphi\colon\mca{H}\to\bb{C}$ and an element $h\in\mca{H}$ such that $\ker\varphi=\pi(E)\mca{H}$ and $\varphi(h)=1$. Let $\mf{an}\stackrel{\xi}{\curvearrowright}V$ be a real or complex representation such that $\xi(\mf{n})=0$. In this section we will construct a ``Hodge decomposition'' of $C^*(\mf{an};\mca{H}\otimes V)$ assuming one more condition which amounts to $\lambda=\frac{1-n}{\sqrt{2}}$. 

Since $f-\varphi(f)h\in\ker\varphi=\pi(E)\mca{H}$ for all $f\in\mca{H}$, we can define an operator $T\colon\mca{H}\to\mca{H}$ by 
\begin{equation*}
Tf=\pi(E)^{-1}(f-\varphi(f)h). 
\end{equation*}
Consider an operator 
\begin{equation*}
\delta=\iota(E)\otimes T\otimes\id
\end{equation*}
on $C^*(\mf{an};\mca{H}\otimes V)=\bigwedge^*\mf{an}^*\otimes\mca{H}\otimes V$ of degree $-1$. Obviously $\delta^2=0$. We will prove this $\delta$ gives a ``Hodge decomposition'' of $C^*(\mf{an};\mca{H}\otimes V)$ under one more condition. First we have $\pi(E)T=\id-\varphi h$ and $T\pi(E)=\id$, hence $[\pi(E),T]=-\varphi h$. As we have shown in Section \ref{geneconsss}, the assumption $\ker\varphi=\pi(E)\mca{H}$ implies that there exists $\nu\in\bb{C}$ such that $\pi(H)\varphi=\nu\varphi$. Since 
\begin{equation*}
\varphi\left(\pi(H)h+\nu h\right)=-(\pi(H)\varphi)(h)+\nu=0, 
\end{equation*}
the image of the operator $\id-\varphi h$ is contained in $\pi(E)\mca{H}$. So we can define an element $h^\prime=\pi(E)^{-1}\left(\pi(H)h+\nu h\right)\in\mca{H}$. 

\begin{lem}\label{pihtb}
We have 
\begin{equation*}
[\pi(H),T]=-\varphi h^\prime-\frac{1}{\sqrt{2}}T. 
\end{equation*}
\end{lem}

\begin{proof}
By Lemma \ref{commmmp} we have 
\begin{align*}
\pi(H)T&=\pi(H)\pi(E)^{-1}\left(\id-\varphi h\right)\\
&=\pi(E)^{-1}\pi(H)\left(\id-\varphi h\right)-\frac{1}{\sqrt{2}}\pi(E)^{-1}\left(\id-\varphi h\right)\\
&=\pi(E)^{-1}\left(\pi(H)-\varphi\pi(H)h\right)-\frac{1}{\sqrt{2}}T. 
\end{align*}
On the other hand 
\begin{align*}
T\pi(H)&=\pi(E)^{-1}\left(\id-\varphi h\right)\pi(H)=\pi(E)^{-1}\left(\pi(H)+\left(\pi(H)\varphi\right)h\right)\\
&=\pi(E)^{-1}\left(\pi(H)+\varphi\nu h\right), 
\end{align*}
hence 
\begin{equation*}
[\pi(H),T]=\pi(E)^{-1}\left(-\varphi\pi(H)h-\varphi\nu h\right)-\frac{1}{\sqrt{2}}T=-\varphi h^\prime-\frac{1}{\sqrt{2}}T. 
\end{equation*}
\end{proof}

\begin{lem}\label{ddelp}
We have 
\begin{equation*}
d\delta+\delta d=\id-p, 
\end{equation*}
where $p\colon C^*(\mf{an};\mca{H}\otimes V)\to C^*(\mf{an};\mca{H}\otimes V)$ is defined by 
\begin{equation*}
p=\varepsilon(\theta_H)\iota(E)\otimes\varphi h^\prime\otimes\id+\varepsilon(\theta_E)\iota(E)\otimes\varphi h\otimes\id. 
\end{equation*}
\end{lem}

\begin{proof}
Recalling 
\begin{equation*}
d=d_\star\otimes\id\otimes\id+\varepsilon(\theta_H)\otimes\pi(H)\otimes\id+\varepsilon(\theta_H)\otimes\id\otimes\xi(H)+\varepsilon(\theta_E)\otimes\pi(E)\otimes\id
\end{equation*}
and using Lemmas \ref{cliffrule} and \ref{pihtb} and Formulas \eqref{liederivativei} and \eqref{liederivativefo} in Section \ref{construc}, we have 
\begin{align*}
d\delta+\delta d&=L_E\otimes T\otimes\id+\varepsilon(\theta_H)\iota(E)\otimes[\pi(H),T]\otimes\id+\varepsilon(\theta_E)\iota(E)\otimes\pi(E)T\otimes\id\\
&\quad+(\id-\varepsilon(\theta_E)\iota(E))\otimes T\pi(E)\otimes\id\\
&=\frac{1}{\sqrt{2}}\varepsilon(\theta_H)\iota(E)\otimes T\otimes\id+\varepsilon(\theta_H)\iota(E)\otimes\left(-\varphi h^\prime-\frac{1}{\sqrt{2}}T\right)\otimes\id\\
&\quad+\varepsilon(\theta_E)\iota(E)\otimes(-\varphi h)\otimes\id+\id\\
&=\id-\varepsilon(\theta_H)\iota(E)\otimes\varphi h^\prime\otimes\id-\varepsilon(\theta_E)\iota(E)\otimes\varphi h\otimes\id\\
&=\id-p. 
\end{align*}
\end{proof}

\begin{lem}
We have $p^2=p$. Hence 
\begin{align*}
C^*(\mf{an};\mca{H}\otimes V)&=\im p\oplus\ker p\\
x&=px+(x-px). 
\end{align*}
\end{lem}

\begin{proof}
Using Lemma \ref{cliffrule} and $\varphi(h)=1$, we have 
\begin{align*}
p^2&=\left(\varepsilon(\theta_H)\iota(E)\otimes\varphi h^\prime\otimes\id+\varepsilon(\theta_E)\iota(E)\otimes\varphi h\otimes\id\right)^2\\
&=\varepsilon(\theta_H)\iota(E)\varepsilon(\theta_E)\iota(E)\otimes\varphi(\varphi h)h^\prime\otimes\id+\varepsilon(\theta_E)\iota(E)\varepsilon(\theta_E)\iota(E)\otimes\varphi(\varphi h)h\otimes\id\\
&=\varepsilon(\theta_H)\iota(E)\otimes\varphi h^\prime\otimes\id+\varepsilon(\theta_E)\iota(E)\otimes\varphi h\otimes\id\\
&=p. 
\end{align*}
\end{proof}

Obviously $\ker d\cap\ker\delta\subset\ker(d\delta+\delta d)$. We consider whether the opposite inclusion holds. 

\begin{lem}\label{delp}
$\delta p=0$. 
\end{lem}

\begin{proof}
Using Lemma \ref{cliffrule} and $\varphi(h)=1$, we have 
\begin{align*}
\delta p&=\left(\iota(E)\otimes T\otimes\id\right)\left(\varepsilon(\theta_H)\iota(E)\otimes\varphi h^\prime\otimes\id+\varepsilon(\theta_E)\iota(E)\otimes\varphi h\otimes\id\right)\\
&=\iota(E)\otimes\left(\pi(E)^{-1}\left(\id-\varphi h\right)\circ\varphi h\right)\otimes\id\\
&=\iota(E)\otimes\pi(E)^{-1}\left(\varphi h-\varphi(\varphi h)h\right)\otimes\id\\
&=0. 
\end{align*}
\end{proof}

\begin{lem}\label{dp}
If $\xi(H)=\nu+\frac{1}{\sqrt{2}}$, then $dp=0$. 
\end{lem}

\begin{proof}
By \eqref{dfor} in Section \ref{construc} we have 
\begin{equation*}
d_\star=-\frac{1}{\sqrt{2}}\varepsilon(\theta_H)\varepsilon(\theta_E)\iota(E). 
\end{equation*}
Recalling $h^\prime=\pi(E)^{-1}\left(\pi(H)h+\nu h\right)$, we have 
\begin{align*}
&dp\\
&=\left(d_\star\otimes\id\otimes\id+\varepsilon(\theta_H)\otimes\pi(H)\otimes\id+\varepsilon(\theta_H)\otimes\id\otimes\xi(H)+\varepsilon(\theta_E)\otimes\pi(E)\otimes\id\right)\\
&\quad\ \left(\varepsilon(\theta_H)\iota(E)\otimes\varphi h^\prime\otimes\id+\varepsilon(\theta_E)\iota(E)\otimes\varphi h\otimes\id\right)\\
&=\varepsilon(\theta_E)\varepsilon(\theta_H)\iota(E)\otimes\varphi\pi(E)h^\prime\otimes\id-\frac{1}{\sqrt{2}}\varepsilon(\theta_H)\varepsilon(\theta_E)\iota(E)\varepsilon(\theta_E)\iota(E)\otimes\varphi h\otimes\id\\
&\quad+\varepsilon(\theta_H)\varepsilon(\theta_E)\iota(E)\otimes\varphi\pi(H)h\otimes\id+\varepsilon(\theta_H)\varepsilon(\theta_E)\iota(E)\otimes\varphi h\otimes\xi(H)\\
&=-\varepsilon(\theta_H)\varepsilon(\theta_E)\iota(E)\otimes\left(\varphi\pi(H)h+\nu\varphi h\right)\otimes\id-\frac{1}{\sqrt{2}}\varepsilon(\theta_H)\varepsilon(\theta_E)\iota(E)\otimes\varphi h\otimes\id\\
&\quad+\varepsilon(\theta_H)\varepsilon(\theta_E)\iota(E)\otimes\varphi\pi(H)h\otimes\id+\varepsilon(\theta_H)\varepsilon(\theta_E)\iota(E)\otimes\varphi h\otimes\xi(H)\\
&=\varepsilon(\theta_H)\varepsilon(\theta_E)\iota(E)\otimes\left(-\left(\nu+\frac{1}{\sqrt{2}}\right)\varphi h\otimes\id+\varphi h\otimes\xi(H)\right)\\
&=0. 
\end{align*}
\end{proof}

Therefore, by Lemmas \ref{ddelp}, \ref{delp} and \ref{dp}, if $\xi(H)=\nu+\frac{1}{\sqrt{2}}$, we have 
\begin{equation*}
\ker d\cap\ker\delta=\ker(d\delta+\delta d). 
\end{equation*}

Without assuming $\xi(H)=\nu+\frac{1}{\sqrt{2}}$ we have: 
\begin{itemize}
\item $\ker p\subset\im d+\im\delta$ by Lemma \ref{ddelp}
\item $p\delta=0$ immediately from the definitions
\item $\im d+\im\delta=\im d\oplus\im\delta$. 
\end{itemize}
For the third one, if $x\in\im d\cap\im\delta$, then $x\in\im p$ by Lemma \ref{ddelp} and $x\in\ker p$ by $p\delta=0$. Since $\ker p\cap\im p=0$, we have $x=0$. 

The following lemma gives a condition under which we have $\ker p=\im d\oplus\im\delta$.

\begin{lem}
If $\xi(H)=\nu+\frac{1}{\sqrt{2}}$, then $pd=0$. 
\end{lem}

\begin{proof}
By the formula 
\begin{equation*}
d_\star=-\frac{1}{\sqrt{2}}\varepsilon(\theta_H)\varepsilon(\theta_E)\iota(E)
\end{equation*}
and $\varphi\left(\pi(E)\mca{H}\right)=0$, we have 
\begin{align*}
&pd\\
&=\left(\varepsilon(\theta_H)\iota(E)\otimes\varphi h^\prime\otimes\id+\varepsilon(\theta_E)\iota(E)\otimes\varphi h\otimes\id\right)\\
&\quad\ \left(d_\star\otimes\id\otimes\id+\varepsilon(\theta_H)\otimes\pi(H)\otimes\id+\varepsilon(\theta_H)\otimes\id\otimes\xi(H)+\varepsilon(\theta_E)\otimes\pi(E)\otimes\id\right)\\
&=-\frac{1}{\sqrt{2}}\varepsilon(\theta_E)\iota(E)\varepsilon(\theta_H)\varepsilon(\theta_E)\iota(E)\otimes\varphi h\otimes\id\\
&\quad+\varepsilon(\theta_E)\iota(E)\varepsilon(\theta_H)\otimes\varphi(\pi(H))h\otimes\id+\varepsilon(\theta_E)\iota(E)\varepsilon(\theta_H)\otimes\varphi h\otimes\xi(H)\\
&=\frac{1}{\sqrt{2}}\varepsilon(\theta_E)\varepsilon(\theta_H)\iota(E)\varepsilon(\theta_E)\iota(E)\otimes\varphi h\otimes\id\\
&\quad-\varepsilon(\theta_E)\varepsilon(\theta_H)\iota(E)\otimes\left(-\left(\pi(H)\varphi\right)h\right)\otimes\id-\varepsilon(\theta_E)\varepsilon(\theta_H)\iota(E)\otimes\varphi h\otimes\xi(H)\\
&=\frac{1}{\sqrt{2}}\varepsilon(\theta_E)\varepsilon(\theta_H)\iota(E)\otimes\varphi h\otimes\id\\
&\quad+\varepsilon(\theta_E)\varepsilon(\theta_H)\iota(E)\otimes\nu\varphi h\otimes\id-\varepsilon(\theta_E)\varepsilon(\theta_H)\iota(E)\otimes\varphi h\otimes\xi(H)\\
&=\varepsilon(\theta_E)\varepsilon(\theta_H)\iota(E)\otimes\left(\left(\nu+\frac{1}{\sqrt{2}}\right)\varphi h\otimes\id-\varphi h\otimes\xi(H)\right)\\
&=0. 
\end{align*}
\end{proof}

Thus if $\xi(H)=\nu+\frac{1}{\sqrt{2}}$, we have $\ker p=\im d\oplus\im\delta$. 

\begin{thm}
Let $\sl(2,\bb{R})\stackrel{\pi}{\curvearrowright}\mca{H}$ be a representation on a complex vector space $\mca{H}$ such that $\pi(E)\colon\mca{H}\to\mca{H}$ is injective and there exist a linear map $\varphi\colon\mca{H}\to\bb{C}$ and an element $h\in\mca{H}$ satisfying $\ker\varphi=\pi(E)\mca{H}$ and $\varphi(h)=1$. Define $\nu\in\bb{C}$ by $\pi(H)\varphi=\nu\varphi$. Let $\mf{an}\stackrel{\xi}{\curvearrowright}V$ be a real or complex representation such that $\xi(E)=0$ and $\xi(H)=\nu+\frac{1}{\sqrt{2}}$. Then we have the following ``Hodge decomposition'': 
\begin{align*}
C^*(\mf{an};\mca{H}\otimes V)&=\im p\oplus\im d\oplus\im\delta\\
x&=px+d\delta x+\delta dx,\\
\im p&=\ker(d\delta+\delta d)=\ker d\cap\ker\delta,\\
\ker p&=\im d\oplus\im\delta,\\
\ker d&=\im p\oplus\im d,\\
\ker\delta&=\im p\oplus\im\delta. 
\end{align*}
Hence 
\begin{equation*}
H^*(\mf{an};\mca{H}\otimes V)\simeq\im p. 
\end{equation*}
\end{thm}

%If $\Omega$ acts on $\mca{H}$ as a scalar $\mu$, we have $\mu=\nu^2+\frac{1}{\sqrt{2}}\nu$. 

\subsubsection{Application to $C^*\left(\mf{an};C^\infty(D_n^\pm)\otimes\bb{C}_\frac{1-n}{\sqrt{2}}\right)$}
Let $\sl(2,\bb{R})\stackrel{\pi}{\curvearrowright}C^\infty(D_n^\pm)$ be the derivative of the discrete series representation. The operator $\pi(E)$ is injective by \eqref{vvvaaa} in Section \ref{tnidoiu}. By Proposition \ref{ffoorr}, there exists a linear map $\varphi\colon C^\infty(D_n^\pm)\to\bb{C}$ such that $\ker\varphi=\pi(E)C^\infty(D_n^\pm)$ and $\pi(H)\varphi=-\frac{n}{\sqrt{2}}\varphi$. Hence we have $\nu=-\frac{n}{\sqrt{2}}$. 

Let $D_n^\pm=\bigoplus_{k=\pm n}^{\pm\infty}(D_n^\pm)_{(k)}$ be the decomposition \eqref{macom} in Section \ref{gener}. See also \eqref{dimmm} in Section \ref{iurops}. The subspace $(D_n^\pm)_{(\pm k)}$ is $1$-dimensional for all $k\geq n$. Since $(D_n^\pm)_{(\pm k)}$ is finite dimensional, $(D_n^\pm)_{(\pm k)}\subset C^\infty(D_n^\pm)$. See Corollary 4.4.3.3 of Warner \cite{Warner1}. 

By the explicit formula of $\varphi$ in Section 3.1 of Flaminio--Forni \cite{FF}, we know $\varphi\left((D_n^\pm)_{(\pm n)}\right)\neq0$. So we can take $h\in(D_n^\pm)_{(\pm n)}$ such that $\varphi(h)=1$. 

Following the previous section, we define operators
\begin{equation*}
T=\pi(E)^{-1}\left(\id-\varphi h\right)
\end{equation*}
on $C^\infty(D_n^\pm)$, 
\begin{equation*}
\delta=\iota(E)\otimes T\otimes\id
\end{equation*}
on $C^*\left(\mf{an};C^\infty(D_n^\pm)\otimes\bb{C}_\frac{1-n}{\sqrt{2}}\right)$ and an element 
\begin{equation*}
h^\prime=\pi(E)^{-1}\left(\pi(H)h-\frac{n}{\sqrt{2}}h\right)\in C^\infty(D_n^\pm). 
\end{equation*}

\begin{lem}\label{hprimeih}
$h^\prime=\pm\sqrt{2}ih$. 
\end{lem}

\begin{proof}
Recall the element $H_0=i
\begin{pmatrix}
&-\frac{1}{2}\\
\frac{1}{2}&
\end{pmatrix}$, which acts on $(D_n^\pm)_{(k)}$ by multiplication by $k$. Let $H_\pm=i
\begin{pmatrix}
\frac{1}{2}&\\
&-\frac{1}{2}
\end{pmatrix}
\mp
\begin{pmatrix}
&\frac{1}{2}\\
\frac{1}{2}&
\end{pmatrix}
\in\sl(2,\bb{R})\otimes\bb{C}$. We have $\pi(H_\mp)h=0$. See for example Section 2 of Chapter 8 of Taylor \cite{Taylor}. Writing 
\begin{equation*}
H_\pm=\sqrt{2}iH\mp2E\pm iH_0, 
\end{equation*}
we have 
\begin{align}
0&=\pi(H_\mp)h=\sqrt{2}i\pi(H)h\pm2\pi(E)h-inh\label{2ihheh}\\
&=\sqrt{2}i\left(\pi(H)h-\frac{n}{\sqrt{2}}h\right)\pm2\pi(E)h. \nonumber
\end{align}
Hence 
\begin{equation*}
h^\prime=\pi(E)^{-1}\left(\pm\sqrt{2}i\pi(E)h\right)=\pm\sqrt{2}ih. 
\end{equation*}
\end{proof}

By Lemma \ref{ddelp}, we have 
\begin{equation*}
d\delta+\delta d=\id-p, 
\end{equation*}
where 
\begin{align*}
p&=\pm\sqrt{2}i\varepsilon(\theta_H)\iota(E)\otimes\varphi h\otimes\id+\varepsilon(\theta_E)\iota(E)\otimes\varphi h\otimes\id\\
&=\varepsilon\left(\theta_E\pm\sqrt{2}i\theta_H\right)\iota(E)\otimes\varphi h\otimes\id. 
\end{align*}
Here $C^*\left(\mf{an};C^\infty(D_n^\pm)\otimes\bb{C}_\frac{1-n}{\sqrt{2}}\right)$ is regarded as $\left(\bigwedge^*\left(\mf{an}^*\otimes\bb{C}\right)\right)\otimes C^\infty(D_n^\pm)\otimes\bb{C}_\frac{1-n}{\sqrt{2}}$. 

The subspace $\im p$ is a graded subspace of $C^*\left(\mf{an};C^\infty(D_n^\pm)\otimes\bb{C}_\frac{1-n}{\sqrt{2}}\right)$. Let $\im p=\bigoplus_{i=0}^2(\im p)^i$ be the decomposition. 

\begin{lem}\label{impds}
We have 
\begin{align*}
(\im p)^0&=0,\\
(\im p)^1&=\bb{C}\left(\theta_E\pm\sqrt{2}i\theta_H\right)\otimes h\otimes1,\\
(\im p)^2&=\bb{C}\left(\theta_E\wedge\theta_H\otimes h\otimes1\right). 
\end{align*}
\end{lem}

\begin{proof}
For $f,g\in C^\infty(D_n^\pm)$, we have 
\begin{equation*}
p\left(1\otimes f\otimes1\right)=0, 
\end{equation*}
\begin{equation*}
p\left(\theta_E\otimes f\otimes1+\theta_H\otimes g\otimes1\right)=\varphi(f)\left(\theta_E\pm\sqrt{2}i\theta_H\right)\otimes h\otimes1
\end{equation*}
and 
\begin{equation*}
p\left(\theta_E\wedge\theta_H\otimes f\otimes1\right)=\varphi(f)\left(\theta_E\wedge\theta_H\otimes h\otimes1\right), 
\end{equation*}
from which the assertions follow. 
\end{proof}

\begin{thm}\label{hodge3}
We have 
\begin{align*}
C^*\left(\mf{an};C^\infty(D_n^\pm)\otimes\bb{C}_\frac{1-n}{\sqrt{2}}\right)&=\im p\oplus\im d\oplus\im\delta,\\
\ker d&=\im p\oplus\im d. 
\end{align*}
In particular, $H^*\left(\mf{an};C^\infty(D_n^\pm)\otimes\bb{C}_\frac{1-n}{\sqrt{2}}\right)\simeq\im p$. 
\end{thm}

\section{Computation results of $H^*(\mca{F};\bb{C}_\lambda)$}\label{comprrrr}
Here we state the computation results. Let $\Gamma$ be a torsion free cocompact lattice of $\PSL(2,\bb{R})$, $\mca{F}$ be the orbit foliation of the action $\Gamma\backslash\PSL(2,\bb{R})\curvearrowleft AN$ by right multiplication and $g_\Sigma$ be the genus of the surface $\Sigma=\Gamma\backslash\PSL(2,\bb{R})/\PSO(2)$. Then we have 
\begin{equation*}
H^i(\mca{F};\bb{C})\simeq
\begin{cases}
\bb{C}&i=0\\
\bb{C}^{2g_\Sigma+1}&i=1\\
\bb{C}^{2g_\Sigma}&i=2\\
0&\text{otherwise}
\end{cases}
\end{equation*}
and 
\begin{equation*}
H^i\left(\mca{F};\bb{C}_\frac{1}{\sqrt{2}}\right)\simeq
\begin{cases}
\bb{C}&i=1,2\\
0&\text{otherwise}. 
\end{cases}
\end{equation*}
For $n\in\bb{Z}_{\geq2}$ we have 
\begin{equation*}
H^i\left(\mca{F};\bb{C}_\frac{1-n}{\sqrt{2}}\right)\simeq
\begin{cases}
\bb{C}^{2(2n-1)(g_\Sigma-1)}&i=1,2\\
0&\text{otherwise}. 
\end{cases}
\end{equation*}
Or equivalently, for $n\in\bb{Z}_{\leq-1}$, 
\begin{equation*}
H^i\left(\mca{F};\bb{C}_\frac{n}{\sqrt{2}}\right)\simeq
\begin{cases}
\bb{C}^{2(1-2n)(g_\Sigma-1)}&i=1,2\\
0&\text{otherwise}. 
\end{cases}
\end{equation*}

Let $\Delta_\Sigma$ be the Laplacian of $\Sigma$, which is an operator on $L^2(\Sigma)$ with domain $C^\infty(\Sigma,\bb{C})$. Let $\sigma(\ol{\Delta_\Sigma})$ be the spectrum of the closure $\ol{\Delta_\Sigma}$ of $\Delta_\Sigma$. We have $\sigma(\ol{\Delta_\Sigma})\subset[0,\infty)$. For $\nu\in\sigma(\ol{\Delta_\Sigma})$ let $m_\nu$ denote the multiplicity of the eigenvalue $\nu$. Then for $\nu\in\sigma(\ol{\Delta_\Sigma})\setminus\{0\}$, we have 
\begin{equation}\label{nohodge}
H^i\left(\mca{F};\bb{C}_\frac{1\pm\sqrt{1-4\nu}}{2\sqrt{2}}\right)\simeq
\begin{cases}
\bb{C}^{m_\nu}&i=1,2\\
0&\text{otherwise}. 
\end{cases}
\end{equation}

Set 
\begin{equation*}
\Lambda^\prime=\left\{\frac{n}{\sqrt{2}}\ \middle|\ n\in\bb{Z}_{\leq1}\right\}\cup\left\{\frac{1\pm\sqrt{1-4\nu}}{2\sqrt{2}}\ \middle|\ \nu\in\sigma(\ol{\Delta_\Sigma})\setminus\{0\}\right\}. 
\end{equation*}
See Figure \eqref{figure2}. 
\begin{figure}[hbtp]
\begin{center}
\includegraphics[keepaspectratio, scale=0.55]{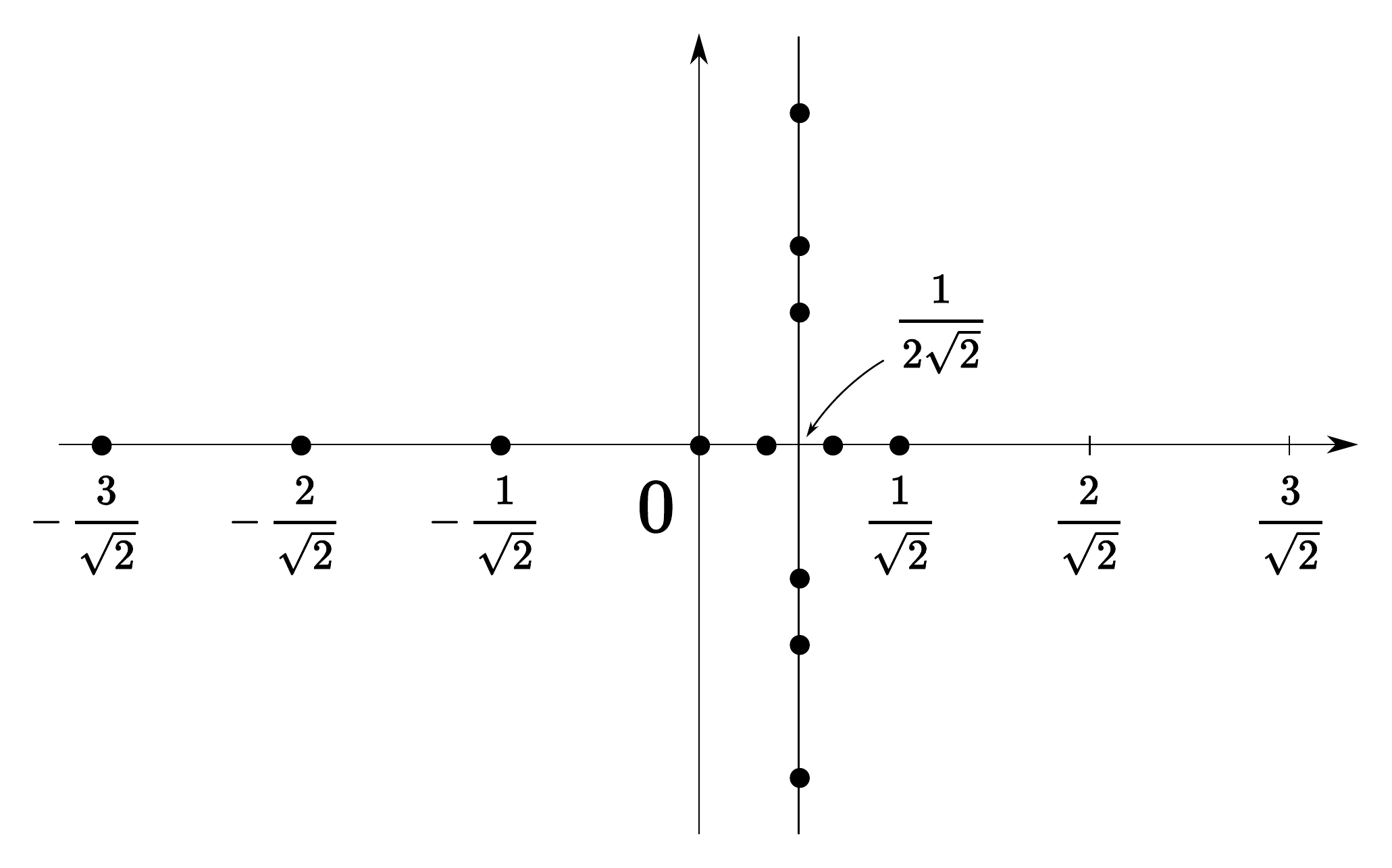}
\caption{$\Lambda^\prime$}
\label{figure2}
\end{center}
\end{figure}
Then 
\begin{equation*}
H^*\left(\mca{F};\bb{C}_\lambda\right)=0
\end{equation*}
for all $\lambda\not\in\Lambda^\prime$. Therefore, $H^*(\mca{F};\bb{C}_\lambda)$ is nontrivial if and only if $\lambda\in\Lambda^\prime$. 

We have ``Hodge decompositions'' except for the case \eqref{nohodge}. The ``Hodge decompositions'' are given by Theorem \ref{36kwlxisee}, Section \ref{easy}, Theorem \ref{hodge2}, Theorem \ref{hodge3} and Lemma \ref{impds}. For the case \eqref{nohodge}, the complex $C^*\left(\mf{an};C^\infty\left(\mca{H}_{-\frac{\nu}{2}}\right)\otimes\bb{C}_\frac{1\pm\sqrt{1-4\nu}}{2\sqrt{2}}\right)$ appears, for which we do not find a ``Hodge decomposition''. 

For $\lambda\in\bb{R}$, computation results of $H^*(\mca{F};\bb{R}_\lambda)$ are obtained from those of $H^*(\mca{F};\bb{C}_\lambda)$ and the isomorphism \eqref{realcomplex} in Section \ref{acofpcf}.

\section{Computation of the restriction map $H^*\left(\Gamma\backslash\PSL(2,\bb{R});\bb{C}\right)\to H^*(\mca{F};\bb{C})$}\label{restmapH}
The goal of this section is computation of the restriction map $H^*(P;\bb{C})\to H^*(\mca{F};\bb{C})$, which has an application as explained in Section \ref{haefp}. The authors know two ways to compute this map. The first one is as follows. We compute $H^*(P;\bb{C})$ by using the Serre spectral sequence of the circle bundle $P\to\Sigma$. Then we construct bases of $H^*(P;\bb{C})$ and $H^*(\mca{F};\bb{C})$ and compute how the restriction map maps the basis. This computation is relatively straightforward. We omit it in this paper. Instead, we will explain the second way of computation, which is less straightforward and is compatible with the way how we computed the cohomology $H^*(\mca{F};\bb{C})$. 

The scheme of the computation is as follows. We use the isomorphism 
\begin{equation*}
H^*(P;\bb{C})\simeq H^*\left(\sl(2,\bb{R});C^\infty(P,\bb{C})\right). 
\end{equation*}
We compute $H^*\left(\sl(2,\bb{R});C^\infty(\mca{H})\right)$ for each irreducible unitary representation $\mca{H}$ of $\PSL(2,\bb{R})$ by using the Hochschild--Serre spectral sequence for the pair $\mf{an}\subset\sl(2,\bb{R})$ and the computation results of $H^*\left(\mf{an};C^\infty(\mca{H})\otimes\bb{C}_\lambda\right)$ for $\lambda=0$, $\frac{1}{\sqrt{2}}$ in Section \ref{cohcomp}. We construct an operator $\delta$ on $C^*\left(\sl(2,\bb{R});C^\infty(P,\bb{C})\right)$ of degree $-1$ such that $d\delta+\delta d=\id\otimes\Omega$, where $\Omega$ is the Casimir element of $\sl(2,\bb{R})$. This shows that a generic part of the cohomology $H^*\left(\sl(2,\bb{R});C^\infty(P,\bb{C})\right)$ vanishes and only the part corresponding to $\bb{C}$ and $C^\infty(D_1^\pm)$ in $C^\infty(P,\bb{C})$ can be nontrivial. So the computation results of $H^*(\sl(2,\bb{R});\bb{C})$ and $H^*\left(\sl(2,\bb{R});C^\infty(D_1^\pm)\right)$ conclude the computation of $H^*(P;\bb{C})$. The computation of the restriction map $H^*(P;\bb{C})\to H^*(\mca{F};\bb{C})$ is based on the computation of the restriction maps 
\begin{equation*}
H^*(\sl(2,\bb{R});\bb{C})\to H^*(\mf{an};\bb{C})
\end{equation*}
and 
\begin{equation*}
H^*\left(\sl(2,\bb{R});C^\infty(D_1^\pm)\right)\to H^*\left(\mf{an};C^\infty(D_1^\pm)\right). 
\end{equation*}
Since we will show $H^*\left(\sl(2,\bb{R});C^\infty(D_1^\pm)\right)\to H^*\left(\mf{an};C^\infty(D_1^\pm)\right)$ is an isomorphism, the noninjectivity and nonsurjectivity come only from $H^*(\sl(2,\bb{R});\bb{C})\to H^*(\mf{an};\bb{C})$.

\subsection{Hochschild--Serre spectral sequence relative to a subalgebra}
Since we will use the Hochschild--Serre spectral sequence for the pair $\mf{an}\subset\sl(2,\bb{R})$, we explain here the Hochschild--Serre spectral sequence relative to a subalgebra. 

Let $K\subset L$ be a field extension, $\mf{g}$ be a Lie algebra over $K$, $\mf{h}$ be a subalgebra of $\mf{g}$ and $\pi\colon\mf{g}\to\End_L(V)$ be a representation on a vector space $V$ over $L$. Then we have a spectral sequence $E^{p,q}_r$ defined from the filtered complex $F^*C^*(\mf{g};V)$, which converges to $H^*(\mf{g};V)$. See Section \ref{hochschildserre}. Recall that the differentials are $d\colon E^{p,q}_r\to E^{p+r,q-r+1}_r$. 

Note that we have a representation $\mf{h}\curvearrowright\bigwedge^p(\mf{g}/\mf{h})^*\otimes V$ defined from the adjoint representation and the restriction of $\pi$. 

\begin{prop}[Hochschild--Serre \cite{HS}]\label{hochserr}
\begin{equation*}
E^{p,q}_1\simeq H^q\left(\mf{h};\bigwedge^p(\mf{g}/\mf{h})^*\otimes V\right). 
\end{equation*}
\end{prop}

\begin{proof}
This is Corollary to Theorem 2 in \cite{HS}. Recall that 
\begin{equation*}
F^pC^{p+q}(\mf{g};V)=\left\{\varphi\in C^{p+q}(\mf{g};V)\ \middle|\ \iota_{X_1}\cdots\iota_{X_{q+1}}\varphi=0\ \text{if $X_1,\ldots,X_{q+1}\in\mf{h}$}\right\}. 
\end{equation*}
Consider the map 
\begin{align*}
F^pC^{p+q}(\mf{g};V)&\to C^q\left(\mf{h};\bigwedge^p(\mf{g}/\mf{h})^*\otimes V\right)\\
\varphi&\mapsto((X_1,\ldots,X_q)\mapsto\iota_{X_q}\cdots\iota_{X_1}\varphi), 
\end{align*}
which induces 
\begin{equation*}
E^{p,q}_0=F^pC^{p+q}(\mf{g};V)/F^{p+1}C^{p+q}(\mf{g};V)\to C^q\left(\mf{h};\bigwedge^p(\mf{g}/\mf{h})^*\otimes V\right). 
\end{equation*}
This map gives an isomorphism 
\begin{equation*}
E^{p,q}_0\simeq C^q\left(\mf{h};\bigwedge^p(\mf{g}/\mf{h})^*\otimes V\right)
\end{equation*}
of cochain complexes, hence the proposition follows by taking the cohomology. See \cite{HS} for the proof of the last sentence. 
\end{proof}

\subsection{Computation of $H^*\left(\sl(2,\bb{R});C^\infty(\mca{H})\right)$ for an irreducible unitary representation $\mca{H}$}\label{compirruni}
Let $\PSL(2,\bb{R})\curvearrowright\mca{H}$ be an irreducible unitary representation and $\sl(2,\bb{R})\curvearrowright C^\infty(\mca{H})$ be the derivative representation. We compute $H^*\left(\sl(2,\bb{R});C^\infty(\mca{H})\right)$ by using the Hochschild--Serre spectral sequence. 

Let $E_r^{p,q}$ be the Hochschild--Serre spectral sequence of $C^\infty(\mca{H})$ relative to the pair $\mf{an}\subset\sl(2,\bb{R})$, which converges to $H^*\left(\sl(2,\bb{R});C^\infty(\mca{H})\right)$. We have 
\begin{equation*}
\sl(2,\bb{R})/\mf{an}\simeq\bb{R}_{-\frac{1}{\sqrt{2}}},\quad\left(\sl(2,\bb{R})/\mf{an}\right)^*\simeq\bb{R}_\frac{1}{\sqrt{2}}
\end{equation*}
as representations of $\mf{an}$. By Proposition \ref{hochserr}, 
\begin{align*}
E_1^{0,q}&\simeq H^q\left(\mf{an};C^\infty(\mca{H})\right), \\
E_1^{1,q}&\simeq H^q\left(\mf{an};C^\infty(\mca{H})\otimes\left(\sl(2,\bb{R})/\mf{an}\right)^*\right)\\
&\simeq H^q\left(\mf{an};C^\infty(\mca{H})\otimes\bb{C}_\frac{1}{\sqrt{2}}\right). 
\end{align*}
Therefore, the $E_1$ sheet is 
\begin{equation*}
\begin{tikzpicture}
\matrix(m)[matrix of math nodes,
nodes in empty cells,
nodes={minimum width=5ex,minimum height=5ex,outer sep=-5pt},
column sep=1ex,row sep=1ex]{
H^2\left(\mf{an};C^\infty(\mca{H})\right)&H^2\left(\mf{an};C^\infty(\mca{H})\otimes\bb{C}_\frac{1}{\sqrt{2}}\right)\\
H^1\left(\mf{an};C^\infty(\mca{H})\right)&H^1\left(\mf{an};C^\infty(\mca{H})\otimes\bb{C}_\frac{1}{\sqrt{2}}\right)\\
H^0\left(\mf{an};C^\infty(\mca{H})\right)&\ H^0\left(\mf{an};C^\infty(\mca{H})\otimes\bb{C}_\frac{1}{\sqrt{2}}\right), \\
};
\end{tikzpicture}
\end{equation*}
where we have $d\colon E^{p,q}_1\to E^{p+1,q}_1$. We consider the spectral sequence separately depending on $\mca{H}$.

\subsubsection*{The trivial representation $\bb{C}$}
First we consider the case when $\mca{H}=\bb{C}$. In this case the $E_1$ sheet is 
\begin{equation*}
\begin{tikzpicture}
\matrix(m)[matrix of math nodes,
nodes in empty cells,
nodes={minimum width=5ex,minimum height=5ex,outer sep=-5pt},
column sep=1ex,row sep=1ex]{
0&\bb{C}\\
\bb{C}&\bb{C}\\
\bb{C}&0\\
};
\end{tikzpicture}
\end{equation*}
by Section \ref{cohcomp}. Since $\sl(2,\bb{R})=[\sl(2,\bb{R}),\sl(2,\bb{R})]$, we have $H^1\left(\sl(2,\bb{R});\bb{C}\right)=0$ directly from the definition. It follows that the only possibly nontrivial differential is actually nontrivial. Hence 
\begin{equation*}
H^i\left(\sl(2,\bb{R});\bb{C}\right)\simeq
\begin{cases}
\bb{C}&i=0,3\\
0&\text{otherwise}. 
\end{cases}
\end{equation*}

\subsubsection*{The discrete series representations $D_1^\pm$}
Next we consider the case when $\mca{H}=D_1^\pm$. The $E_1$ sheet is 
\begin{equation*}
\begin{tikzpicture}
\matrix(m)[matrix of math nodes,
nodes in empty cells,
nodes={minimum width=5ex,minimum height=5ex,outer sep=-5pt},
column sep=1ex,row sep=1ex]{
\bb{C}&0\\
\bb{C}&0\\
0&0\\
};
\end{tikzpicture}
\end{equation*}
by Section \ref{cohcomp}. So all the differentials of $E_1$ are zero. Therefore, 
\begin{equation*}
H^i\left(\sl(2,\bb{R});C^\infty(D_1^\pm)\right)\simeq
\begin{cases}
\bb{C}&i=1,2\\
0&\text{otherwise}. 
\end{cases}
\end{equation*}

\begin{prop}\label{restrichom}
The homomorphism 
\begin{equation*}
H^*\left(\sl(2,\bb{R});C^\infty(D_1^\pm)\right)\to H^*\left(\mf{an};C^\infty(D_1^\pm)\right)
\end{equation*}
induced by the restriction map is an isomorphism. 
\end{prop}

\begin{proof}
Let $F_r^{p,q}$ be the Hochschild--Serre spectral sequence of $C^*\left(\mf{an};C^\infty(D_1^\pm)\right)$ relative to the inclusion $\mf{an}\subset\mf{an}$. The restriction map 
\begin{equation*}
s\colon C^*\left(\sl(2,\bb{R});C^\infty(D_1^\pm)\right)\to C^*\left(\mf{an};C^\infty(D_1^\pm)\right)
\end{equation*}
is a cochain map and preserves the filtrations. Hence it induces a morphism $s\colon E_r^{p,q}\to F_r^{p,q}$ of spectral sequences. We show $s\colon E_1\to F_1$ is an isomorphism. The spectral sequence $F_r^{p,q}$ converges to $H^*\left(\mf{an};C^\infty(D_1^\pm)\right)$ and the $F_1$ sheet is 
\begin{equation*}
\begin{tikzpicture}
\matrix(m)[matrix of math nodes,
nodes in empty cells,
nodes={minimum width=5ex,minimum height=5ex,outer sep=-5pt},
column sep=1ex,row sep=1ex]{
H^2\left(\mf{an};C^\infty(D_1^\pm)\right)&0\\
H^1\left(\mf{an};C^\infty(D_1^\pm)\right)&0\\
H^0\left(\mf{an};C^\infty(D_1^\pm)\right)&0\\
};
\end{tikzpicture}
\end{equation*}
by Proposition \ref{hochserr}. Consider the following diagram 
\begin{equation*}
\begin{tikzcd}
E_1^{0,q}\ar[r,"\sim"]\ar[d,"s"']&H^q\left(\mf{an};C^\infty(D_1^\pm)\right)\ar[d,"\id"]\\
F_1^{0,q}\ar[r,"\sim"]&H^q\left(\mf{an};C^\infty(D_1^\pm)\right), 
\end{tikzcd}
\end{equation*}
where the horizontal arrows are those given by Proposition \ref{hochserr}. This is commutative by the description of the horizontal maps in the proof of Proposition \ref{hochserr}. Therefore, $s\colon E_1\to F_1$ is an isomorphism. This implies that $s\colon E_\infty\to F_\infty$ is also an isomorphism. Using the commutative diagram 
\begin{equation*}
\begin{tikzcd}
E_\infty^{0,q}\ar[r,"\sim"]\ar[d,"s"']&H^q\left(\sl(2,\bb{R});C^\infty(D_1^\pm)\right)\ar[d,"s"]\\
F_\infty^{0,q}\ar[r,"\sim"]&H^q\left(\mf{an};C^\infty(D_1^\pm)\right), 
\end{tikzcd}
\end{equation*}
we see $s\colon H^q\left(\sl(2,\bb{R});C^\infty(D_1^\pm)\right)\to H^q\left(\mf{an};C^\infty(D_1^\pm)\right)$ is an isomorphism. 
\end{proof}

\subsubsection*{The discrete series representations $D_n^\pm$ for $n\geq2$}
When $\mca{H}=D_n^\pm$ for some $n\geq2$, we have $E_1=0$ by Section \ref{cohcomp}. Hence 
\begin{equation*}
H^*\left(\sl(2,\bb{R});C^\infty(D_n^\pm)\right)=0. 
\end{equation*}

\subsubsection*{The principal and complementary series representations $\mca{H}_\mu$}
When $\mca{H}=\mca{H}_\mu$ for some $\mu<0$, we have $E_1=0$ by Section \ref{cohcomp}. Hence 
\begin{equation*}
H^*\left(\sl(2,\bb{R});C^\infty(\mca{H}_\mu)\right)=0. 
\end{equation*}

\subsection{A construction of an operator $\delta$ on $C^*\left(\sl(2,\bb{R});\mca{H}\otimes V\right)$}
In this section we will construct an operator $\delta$ of degree $-1$ on $C^*\left(\sl(2,\bb{R});\mca{H}\otimes V\right)$ such that 
\begin{equation*}
d\delta+\delta d=\id\otimes\Omega\otimes\id-\id\otimes\id\otimes\Omega, 
\end{equation*}
where $\mca{H}$ and $V$ are representations of $\sl(2,\bb{R})$ and $\Omega$ is the Casimir element of $\sl(2,\bb{R})$. The existence of this operator virtually says that a generic part of the cohomology $H^*\left(\sl(2,\bb{R});\mca{H}\otimes V\right)$ vanishes. This is an analogue of the operator $\delta$ on $C^*\left(\mf{an};\mca{H}\otimes V\right)$ constructed in Section \ref{firsthodge99} and turns out to be an extension of an operator defined on a smaller subspace of $C^*\left(\sl(2,\bb{R});\mca{H}\otimes V\right)$ which appears in 2.3 of Borel--Wallach \cite{BW}. Though we consider $C^*\left(\sl(2,\bb{R});\mca{H}\otimes V\right)$, we will need only the case when $V=\bb{R}$ for our application.

\subsubsection{The definition of $\hat{\delta}$ and computation of $\hat{d}\hat{\delta}+\hat{\delta}\hat{d}$}
Let $\sl(2,\bb{R})\stackrel{\pi}{\curvearrowright}\mca{H}$ be a representation on a complex vector space and $\sl(2,\bb{R})\stackrel{\xi}{\curvearrowright}V$ be a representation on a real vector space. Consider the tensor product $\sl(2,\bb{R})\curvearrowright\mca{H}\otimes V$ of the representations and the Chevalley--Eilenberg complex $C^*\left(\sl(2,\bb{R});\mca{H}\otimes V\right)=\bigwedge^*\sl(2,\bb{R})^*\otimes\mca{H}\otimes V$ of it. We equip $\sl(2,\bb{R})\oplus\sl(2,\bb{R})^*$ with the symmetric bilinear form defined by the duality between $\sl(2,\bb{R})$ and $\sl(2,\bb{R})^*$. (See Section \ref{construc}, where we defined a similar form for $\mf{an}\oplus\mf{an}^*$.) Let $C\left(\sl(2,\bb{R})\oplus\sl(2,\bb{R})^*\right)$ be the Clifford algebra of $\sl(2,\bb{R})\oplus\sl(2,\bb{R})^*$ and $U(\sl(2,\bb{R}))$ be the (real) universal enveloping algebra of $\sl(2,\bb{R})$. Consider the tensor product 
\begin{equation*}
\mca{A}=C\left(\sl(2,\bb{R})\oplus\sl(2,\bb{R})^*\right)\otimes U(\sl(2,\bb{R}))\otimes U(\sl(2,\bb{R}))
\end{equation*}
of the algebras. Let $C\left(\sl(2,\bb{R})\oplus\sl(2,\bb{R})^*\right)\stackrel{\sigma}{\curvearrowright}\bigwedge^*\sl(2,\bb{R})^*$ be the spin representation. Then the homomorphism 
\begin{equation*}
\sigma\colon C\left(\sl(2,\bb{R})\oplus\sl(2,\bb{R})^*\right)\to\End\left(\bigwedge^*\sl(2,\bb{R})^*\right)
\end{equation*}
is an isomorphism of algebras. We have a representation 
\begin{equation*}
\mca{A}\stackrel{\sigma\otimes\pi\otimes\xi}{\curvearrowright} C^*\left(\sl(2,\bb{R});\mca{H}\otimes V\right). 
\end{equation*}
Let $\hat{d}_\star\in C\left(\sl(2,\bb{R})\oplus\sl(2,\bb{R})^*\right)$ be the unique element such that $\sigma(\hat{d}_\star)$ is the differential of the cochain complex $C^*(\sl(2,\bb{R});\bb{R})=\bigwedge^*\sl(2,\bb{R})^*$. Recall the basis $E$, $H$, $F$ of $\sl(2,\bb{R})$ and let $\theta_E$, $\theta_H$, $\theta_F\in\sl(2,\bb{R})^*$ be the dual basis of it. Define an element $\hat{d}\in\mca{A}$ by 
\begin{align*}
\hat{d}&=\hat{d}_\star\otimes1\otimes1+\theta_E\otimes E\otimes1+\theta_E\otimes1\otimes E\\
&\quad+\theta_H\otimes H\otimes1+\theta_H\otimes1\otimes H\\
&\quad+\theta_F\otimes F\otimes1+\theta_F\otimes1\otimes F. 
\end{align*}
Then $\left(\sigma\otimes\pi\otimes\xi\right)(\hat{d})=d$, where $d$ is the differential of $C^*\left(\sl(2,\bb{R});\mca{H}\otimes V\right)$. 

Now define $\hat{\delta}\in\mca{A}$ by 
\begin{align*}
\hat{\delta}&=E\otimes F\otimes1+H\otimes H\otimes1+F\otimes E\otimes1\\
&\quad-E\otimes1\otimes F-H\otimes1\otimes H-F\otimes1\otimes E
\end{align*}
and put $\delta=\left(\sigma\otimes\pi\otimes\xi\right)(\hat{\delta})$. 

\begin{thm}\label{5555}
We have 
\begin{equation*}
\hat{d}\hat{\delta}+\hat{\delta}\hat{d}=1\otimes\Omega\otimes1-1\otimes1\otimes\Omega
\end{equation*}
in $\mca{A}$, where $\Omega\in U(\sl(2,\bb{R}))$ is the Casimir element of $\sl(2,\bb{R})$. 
\end{thm}

\begin{proof}
In $C\left(\sl(2,\bb{R})\oplus\sl(2,\bb{R})^*\right)$ we have 
\begin{gather*}
\hat{d}_\star X+X\hat{d}_\star=\hat{L}_X,\\
\varphi X+X\varphi=\varphi(X)
\end{gather*}
for $X\in\sl(2,\bb{R})$ and $\varphi\in\sl(2,\bb{R})^*$. Using these formulas and the definitions 
\begin{align*}
\hat{d}&=\hat{d}_\star\otimes1\otimes1+\theta_E\otimes E\otimes1+\theta_E\otimes1\otimes E\\
&\quad+\theta_H\otimes H\otimes1+\theta_H\otimes1\otimes H\\
&\quad+\theta_F\otimes F\otimes1+\theta_F\otimes1\otimes F
\end{align*}
and 
\begin{align*}
\hat{\delta}&=E\otimes F\otimes1+H\otimes H\otimes1+F\otimes E\otimes1\\
&\quad-E\otimes1\otimes F-H\otimes1\otimes H-F\otimes1\otimes E, 
\end{align*}
we get 
\begin{align*}
&\hat{d}\hat{\delta}+\hat{\delta}\hat{d}\\
&=\hat{L}_E\otimes F\otimes1+\hat{L}_H\otimes H\otimes1+\hat{L}_F\otimes E\otimes1\\
&\quad-\hat{L}_E\otimes1\otimes F-\hat{L}_H\otimes1\otimes H-\hat{L}_F\otimes1\otimes E\\
&\quad+\theta_EE\otimes EF\otimes1+(1-\theta_EE)\otimes FE\otimes1+\theta_EH\otimes EH\otimes1\\
&\quad-\theta_EH\otimes HE\otimes1-1\otimes E\otimes F\\
&\quad+1\otimes F\otimes E-\theta_EE\otimes1\otimes EF-(1-\theta_EE)\otimes1\otimes FE\\
&\quad-\theta_EH\otimes1\otimes EH+\theta_EH\otimes1\otimes HE\\
&\quad+\theta_HE\otimes HF\otimes1-\theta_HE\otimes FH\otimes1+1\otimes H^2\otimes1+\theta_HF\otimes HE\otimes1\\
&\quad-\theta_HF\otimes EH\otimes1-1\otimes H\otimes H\\
&\quad+1\otimes H\otimes H-\theta_HE\otimes1\otimes HF+\theta_HE\otimes1\otimes FH-1\otimes1\otimes H^2\\
&\quad-\theta_HF\otimes1\otimes HE+\theta_HF\otimes1\otimes EH\\
&\quad+\theta_FH\otimes FH\otimes1-\theta_FH\otimes HF\otimes1+\theta_FF\otimes FE\otimes1\\
&\quad+(1-\theta_FF)\otimes EF\otimes1-1\otimes F\otimes E\\
&\quad+1\otimes E\otimes F-\theta_FH\otimes1\otimes FH+\theta_FH\otimes1\otimes HF-\theta_FF\otimes1\otimes FE\\
&\quad-(1-\theta_FF)\otimes1\otimes EF. 
\end{align*}
By the formula 
\begin{equation*}
\hat{L}_X=-\theta_E[X,E]-\theta_H[X,H]-\theta_F[X,F]
\end{equation*}
for $X\in\sl(2,\bb{R})$, we have 
\begin{gather*}
\hat{L}_E=\frac{1}{\sqrt{2}}\theta_HE-\frac{1}{\sqrt{2}}\theta_FH,\quad\hat{L}_H=-\frac{1}{\sqrt{2}}\theta_EE+\frac{1}{\sqrt{2}}\theta_FF,\\
\hat{L}_F=\frac{1}{\sqrt{2}}\theta_EH-\frac{1}{\sqrt{2}}\theta_HF. 
\end{gather*}
Recalling 
\begin{equation*}
\Omega=FE+H^2+EF, 
\end{equation*}
it follows that 
\begin{align*}
&\hat{d}\hat{\delta}+\hat{\delta}\hat{d}\\
&=\left(\frac{1}{\sqrt{2}}\theta_HE-\frac{1}{\sqrt{2}}\theta_FH\right)\otimes F\otimes1+\left(-\frac{1}{\sqrt{2}}\theta_EE+\frac{1}{\sqrt{2}}\theta_FF\right)\otimes H\otimes1\\
&\quad+\left(\frac{1}{\sqrt{2}}\theta_EH-\frac{1}{\sqrt{2}}\theta_HF\right)\otimes E\otimes1-\left(\frac{1}{\sqrt{2}}\theta_HE-\frac{1}{\sqrt{2}}\theta_FH\right)\otimes1\otimes F\\
&\quad-\left(-\frac{1}{\sqrt{2}}\theta_EE+\frac{1}{\sqrt{2}}\theta_FF\right)\otimes1\otimes H-\left(\frac{1}{\sqrt{2}}\theta_EH-\frac{1}{\sqrt{2}}\theta_HF\right)\otimes1\otimes E\\
&\quad+1\otimes FE\otimes1-1\otimes E\otimes F+1\otimes F\otimes E-1\otimes1\otimes FE+1\otimes H^2\otimes1\\
&\quad-1\otimes H\otimes H+1\otimes H\otimes H-1\otimes1\otimes H^2+1\otimes EF\otimes1-1\otimes F\otimes E\\
&\quad+1\otimes E\otimes F-1\otimes1\otimes EF\\
&\quad+\theta_EE\otimes\frac{1}{\sqrt{2}}H\otimes1+\theta_EH\otimes\left(-\frac{1}{\sqrt{2}}E\right)\otimes1-\theta_EE\otimes1\otimes\frac{1}{\sqrt{2}}H\\
&\quad+\theta_EH\otimes1\otimes\frac{1}{\sqrt{2}}E+\theta_HE\otimes\left(-\frac{1}{\sqrt{2}}F\right)\otimes1+\theta_HF\otimes\frac{1}{\sqrt{2}}E\otimes1\\
&\quad+\theta_HE\otimes1\otimes\frac{1}{\sqrt{2}}F-\theta_HF\otimes1\otimes\frac{1}{\sqrt{2}}E+\theta_FH\otimes\frac{1}{\sqrt{2}}F\otimes1\\
&\quad+\theta_FF\otimes\left(-\frac{1}{\sqrt{2}}H\right)\otimes1+\theta_FH\otimes1\otimes\left(-\frac{1}{\sqrt{2}}F\right)-\theta_FF\otimes1\otimes\left(-\frac{1}{\sqrt{2}}H\right)\\
&=1\otimes FE\otimes1-1\otimes1\otimes FE+1\otimes H^2\otimes1-1\otimes1\otimes H^2+1\otimes EF\otimes1\\
&\quad-1\otimes1\otimes EF\\
&=1\otimes\Omega\otimes1-1\otimes1\otimes\Omega. 
\end{align*}
\end{proof}

\begin{cor}\label{ddddxi}
We have 
\begin{equation*}
d\delta+\delta d=\id\otimes\pi(\Omega)\otimes\id-\id\otimes\id\otimes\xi(\Omega)
\end{equation*}
on $C^*\left(\sl(2,\bb{R});\mca{H}\otimes V\right)$. 
\end{cor}

\begin{rem}
We can get 
\begin{align*}
\hat{\delta}^2&=\frac{1}{\sqrt{2}}\left(EH\otimes(F\otimes1+1\otimes F)+HF\otimes(E\otimes1+1\otimes E)\right.\\
&\left.\qquad\quad+FE\otimes(H\otimes1+1\otimes H)\right), 
\end{align*}
which shows $\delta^2$ is not necessarily $0$ on $C^*\left(\sl(2,\bb{R});\mca{H}\otimes V\right)$. 
\end{rem}

\subsubsection{The operator $\delta$ preserves $C^*\left(\sl(2,\bb{R}),\so(2);\mca{H}\otimes V\right)$}
Now we will prove that the operator $\delta$ preserves the subcomplex 
\begin{equation*}
C^*\left(\sl(2,\bb{R}),\so(2);\mca{H}\otimes V\right)
\end{equation*}
of $C^*\left(\sl(2,\bb{R});\mca{H}\otimes V\right)$, which will be defined later. Since we will not use this property in this paper, this section and Section \ref{relationco} can be skipped. 

First we will define a representation $\sl(2,\bb{R})\stackrel{\ad}{\curvearrowright}\mca{A}$. Recall that we have the adjoint representation $\sl(2,\bb{R})\stackrel{\ad}{\curvearrowright}U(\sl(2,\bb{R}))$ defined by 
\begin{equation*}
(\ad X)u=Xu-uX
\end{equation*}
for $X\in\sl(2,\bb{R})$ and $u\in U(\sl(2,\bb{R}))$. We also have a representation $\sl(2,\bb{R})\stackrel{\ad}{\curvearrowright}C\left(\sl(2,\bb{R})\oplus\sl(2,\bb{R})^*\right)$ defined by 
\begin{equation*}
(\ad X)x=\hat{L}_Xx-x\hat{L}_X
\end{equation*}
for $X\in\sl(2,\bb{R})$ and $x\in C\left(\sl(2,\bb{R})\oplus\sl(2,\bb{R})^*\right)$. This is actually a representation since 
\begin{itemize}
\item the map 
\begin{align*}
\sl(2,\bb{R})&\to C\left(\sl(2,\bb{R})\oplus\sl(2,\bb{R})^*\right)\\
X&\mapsto\hat{L}_X
\end{align*}
is a Lie algebra homomorphism, ie the formula 
\begin{equation*}
\hat{L}_{[X,Y]}=\hat{L}_X\hat{L}_Y-\hat{L}_Y\hat{L}_X
\end{equation*}
holds for all $X$, $Y\in\sl(2,\bb{R})$, 
\item the map 
\begin{align*}
C\left(\sl(2,\bb{R})\oplus\sl(2,\bb{R})^*\right)&\to\End\left(C\left(\sl(2,\bb{R})\oplus\sl(2,\bb{R})^*\right)\right)\\
x&\mapsto\left(y\mapsto xy-yx\right)
\end{align*}
is the adjoint representation of the Lie algebra $C\left(\sl(2,\bb{R})\oplus\sl(2,\bb{R})^*\right)$
\item and it is the composition of these two Lie algebra homomorphisms. 
\end{itemize}

The notation $\ad$ for the representation $\sl(2,\bb{R})\curvearrowright C\left(\sl(2,\bb{R})\oplus\sl(2,\bb{R})^*\right)$ is justified by the following lemma. 

\begin{lem}
We have 
\begin{gather*}
\hat{L}_XY-Y\hat{L}_X=[X,Y],\quad\hat{L}_X\varphi-\varphi\hat{L}_X=-\varphi\circ\ad X
\end{gather*}
in $C\left(\sl(2,\bb{R})\oplus\sl(2,\bb{R})^*\right)$ for $X$, $Y\in\sl(2,\bb{R})$ and $\varphi\in\sl(2,\bb{R})^*$. Here $[X,Y]$ in the first equation is the Lie bracket of $X$ and $Y$ in the Lie algebra $\sl(2,\bb{R})$, not in $C\left(\sl(2,\bb{R})\oplus\sl(2,\bb{R})^*\right)$. 
\end{lem}

\begin{proof}
Let $X_1$, $X_2$, $X_3$ be a basis of $\sl(2,\bb{R})$ and $\theta_1$, $\theta_2$, $\theta_3\in\sl(2,\bb{R})^*$ be its dual basis. The lemma follows from the computation based on the formula 
\begin{equation*}
\hat{L}_X=-\sum_i\theta_i[X,X_i]. 
\end{equation*}
\end{proof}

Consider the direct product $\sl(2,\bb{R})\curvearrowright\sl(2,\bb{R})\oplus\sl(2,\bb{R})^*$ of the adjoint representation $\sl(2,\bb{R})\curvearrowright\sl(2,\bb{R})$ and its dual $\sl(2,\bb{R})\curvearrowright\sl(2,\bb{R})^*$. Note that we have a natural inclusion $\sl(2,\bb{R})\oplus\sl(2,\bb{R})^*\subset C\left(\sl(2,\bb{R})\oplus\sl(2,\bb{R})^*\right)$. The representation $\sl(2,\bb{R})\stackrel{\ad}{\curvearrowright}C\left(\sl(2,\bb{R})\oplus\sl(2,\bb{R})^*\right)$ is obtained by extending $\sl(2,\bb{R})\curvearrowright\sl(2,\bb{R})\oplus\sl(2,\bb{R})^*$ to $C\left(\sl(2,\bb{R})\oplus\sl(2,\bb{R})^*\right)$ so that all the operators act as derivations of the algebra $C\left(\sl(2,\bb{R})\oplus\sl(2,\bb{R})^*\right)$. 

By taking tensor product of representations, we get a representation 
\begin{equation*}
\sl(2,\bb{R})\stackrel{\ad}{\curvearrowright}\mca{A}. 
\end{equation*}
Consider the Lie subalgebra $\so(2)$ of $\sl(2,\bb{R})$. Let $R=E-F\in\so(2)$. Then $R$ is a basis of $\so(2)$. 

\begin{prop}\label{soinvariant}
We have 
\begin{equation*}
\hat{\delta}\in\mca{A}^{\so(2)}, 
\end{equation*}
ie $(\ad R)\hat{\delta}=0$. 
\end{prop}

\begin{proof}
We have 
\begin{gather*}
(\ad R)E=\frac{1}{\sqrt{2}}H,\quad(\ad R)H=-\frac{1}{\sqrt{2}}(E+F),\quad(\ad R)F=\frac{1}{\sqrt{2}}H. 
\end{gather*}
Using the definition 
\begin{align*}
\hat{\delta}&=E\otimes F\otimes1+H\otimes H\otimes1+F\otimes E\otimes1\\
&\quad-E\otimes1\otimes F-H\otimes1\otimes H-F\otimes1\otimes E, 
\end{align*}
we get 
\begin{align*}
(\ad R)\hat{\delta}&=\frac{1}{\sqrt{2}}H\otimes F\otimes1+E\otimes\left(-\frac{1}{\sqrt{2}}F\right)\otimes1-\frac{1}{\sqrt{2}}(E+F)\otimes H\otimes1\\
&\quad+H\otimes\left(-\frac{1}{\sqrt{2}}(E+F)\right)\otimes1+\frac{1}{\sqrt{2}}H\otimes E\otimes1+F\otimes\frac{1}{\sqrt{2}}H\otimes1\\
&\quad-\frac{1}{\sqrt{2}}H\otimes1\otimes F-E\otimes1\otimes\frac{1}{\sqrt{2}}H+\frac{1}{\sqrt{2}}(E+F)\otimes1\otimes H\\
&\quad-H\otimes1\otimes\left(-\frac{1}{\sqrt{2}}(E+F)\right)-\frac{1}{\sqrt{2}}H\otimes1\otimes E-F\otimes1\otimes\frac{1}{\sqrt{2}}H\\
&=0. 
\end{align*}
\end{proof}

\begin{lem}\label{drrd0}
We have 
\begin{equation*}
\hat{\delta}\left(R\otimes1\otimes1\right)+\left(R\otimes1\otimes1\right)\hat{\delta}=0
\end{equation*}
in $\mca{A}$. 
\end{lem}

\begin{proof}
This follows from the equation $RX+XR=0$ in $C\left(\sl(2,\bb{R})\oplus\sl(2,\bb{R})^*\right)$ for any $X\in\sl(2,\bb{R})$. 
\end{proof}

We define a representation $\sl(2,\bb{R})\stackrel{\mca{L}}{\curvearrowright}C^*\left(\sl(2,\bb{R});\mca{H}\otimes V\right)$ by 
\begin{align*}
(\mca{L}_X\varphi)\left(X_1,\ldots,X_p\right)&=-\sum_{i=1}^p\varphi\left(X_1,\ldots,[X,X_i],\ldots,X_p\right)\\
&\quad+\left(\pi(X)\otimes\id+\id\otimes\xi(X)\right)\varphi(X_1,\ldots,X_p)
\end{align*}
for $\varphi\in C^p\left(\sl(2,\bb{R});\mca{H}\otimes V\right)$ and $X$, $X_1,\ldots,X_p\in\sl(2,\bb{R})$. Then we have 
\begin{align*}
\mca{L}_X&=L_X\otimes\id\otimes\id+\id\otimes\pi(X)\otimes\id+\id\otimes\id\otimes\xi(X)\\
&=\left(\sigma\otimes\pi\otimes\xi\right)\left(\hat{L}_X\otimes1\otimes1+1\otimes X\otimes1+1\otimes1\otimes X\right)
\end{align*}
for $X\in\sl(2,\bb{R})$. 

Define $C^*\left(\sl(2,\bb{R}),\so(2);\mca{H}\otimes V\right)\subset C^*\left(\sl(2,\bb{R});\mca{H}\otimes V\right)$ by 
\begin{equation*}
C^*\left(\sl(2,\bb{R}),\so(2);\mca{H}\otimes V\right)=\left\{\varphi\in C^*\left(\sl(2,\bb{R});\mca{H}\otimes V\right)\ \middle|\ 
\begin{gathered}
\mca{L}_R\varphi=0,\\
\left(\iota(R)\otimes\id\otimes\id\right)\varphi=0
\end{gathered}
\right\}. 
\end{equation*}
Since 
\begin{gather*}
d\mca{L}_R=\mca{L}_Rd,\quad d\left(\iota(R)\otimes\id\otimes\id\right)+\left(\iota(R)\otimes\id\otimes\id\right)d=\mca{L}_R, 
\end{gather*}
the subspace $C^*\left(\sl(2,\bb{R}),\so(2);\mca{H}\otimes V\right)$ is a subcomplex of $C^*\left(\sl(2,\bb{R});\mca{H}\otimes V\right)$. 

By the definition of $\sl(2,\bb{R})\stackrel{\ad}{\curvearrowright}\mca{A}$, we have 
\begin{align*}
(\ad X)x&=\left(\hat{L}_X\otimes1\otimes1+1\otimes X\otimes1+1\otimes1\otimes X\right)x\\
&\quad-x\left(\hat{L}_X\otimes1\otimes1+1\otimes X\otimes1+1\otimes1\otimes X\right)
\end{align*}
for any $X\in\sl(2,\bb{R})$ and $x\in\mca{A}$, hence 
\begin{equation}\label{coad}
\left(\sigma\otimes\pi\otimes\xi\right)\left((\ad X)x\right)=\mca{L}_X\left(\sigma\otimes\pi\otimes\xi\right)(x)-\left(\sigma\otimes\pi\otimes\xi\right)(x)\mca{L}_X, 
\end{equation}
which means 
\begin{equation*}
\sigma\otimes\pi\otimes\xi\colon\mca{A}\to\End\left(C^*\left(\sl(2,\bb{R});\mca{H}\otimes V\right)\right)
\end{equation*}
is $\sl(2,\bb{R})$-equivariant. 

\begin{prop}
The operator $\delta$ preserves $C^*\left(\sl(2,\bb{R}),\so(2);\mca{H}\otimes V\right)$. 
\end{prop}

\begin{proof}
Take $\varphi\in C^*\left(\sl(2,\bb{R}),\so(2);\mca{H}\otimes V\right)$. By Equation \eqref{coad} and Proposition \ref{soinvariant}, we have 
\begin{equation*}
0=\left(\sigma\otimes\pi\otimes\xi\right)\left((\ad R)\hat{\delta}\right)=\mca{L}_R\delta-\delta\mca{L}_R. 
\end{equation*}
Thus $\mca{L}_R\delta\varphi=0$. Applying $\sigma\otimes\pi\otimes\xi$ to the equation in Lemma \ref{drrd0}, we also get 
\begin{equation*}
\delta\left(\iota(R)\otimes\id\otimes\id\right)+\left(\iota(R)\otimes\id\otimes\id\right)\delta=0, 
\end{equation*}
hence $\left(\iota(R)\otimes\id\otimes\id\right)\delta\varphi=0$. Therefore $\delta\varphi\in C^*\left(\sl(2,\bb{R}),\so(2);\mca{H}\otimes V\right)$. 
\end{proof}

\subsubsection{Relation with the usual codifferential by Hodge theory}\label{relationco}
In this section we will show that the operator $\delta$ on $C^*\left(\sl(2,\bb{R});\mca{H}\otimes V\right)$ is an extension of an operator $\delta_{MM}$ considered in 2.3 of Borel--Wallach \cite{BW} defined on certain subspace of $C^*\left(\sl(2,\bb{R});\mca{H}\otimes V\right)$ which contains $C^*\left(\sl(2,\bb{R}),\so(2);\mca{H}\otimes V\right)$. The operator $\delta_{MM}$ is a generalization of (an extension of) an operator considered in Matsushima--Murakami \cite{MaMu}. 

Let $\sl(2,\bb{R})=\so(2)\oplus\mf{p}$ be the Cartan decomposition with respect to the Cartan involution $\theta\colon X\mapsto-X^\top$. Consider a graded vector space 
\begin{equation*}
D^*=\bigwedge^*\mf{p}^*\otimes\mca{H}\otimes V. 
\end{equation*}
The projection $\sl(2,\bb{R})\to\mf{p}$ induces a map $\bigwedge^*\mf{p}^*\to\bigwedge^*\sl(2,\bb{R})^*$ and hence an injective map 
\begin{equation*}
i\colon D^*\to C^*\left(\sl(2,\bb{R});\mca{H}\otimes V\right). 
\end{equation*}
By identifying the image of $i$, we get  
\begin{equation}\label{natident}
D^*\simeq\left\{\varphi\in C^*\left(\sl(2,\bb{R});\mca{H}\otimes V\right)\ \middle|\ \left(\iota(R)\otimes\id\otimes\id\right)\varphi=0\right\}. 
\end{equation}
Note that, by recalling the formula 
\begin{equation*}
d\left(\iota(R)\otimes\id\otimes\id\right)+\left(\iota(R)\otimes\id\otimes\id\right)d=\mca{L}_R, 
\end{equation*}
the right hand side of \eqref{natident} is not a subcomplex of $C^*\left(\sl(2,\bb{R});\mca{H}\otimes V\right)$ in general. 

By using representations $\so(2)\stackrel{\ad}{\curvearrowright}\mf{p}$, $\sl(2,\bb{R})\stackrel{\pi}{\curvearrowright}\mca{H}$ and $\sl(2,\bb{R})\stackrel{\xi}{\curvearrowright}V$, we get a representation $\so(2)\stackrel{\mca{L}}{\curvearrowright}D^*$. Define a graded subspace $C^*\subset D^*$ by 
\begin{equation*}
C^*=\left(D^*\right)^{\so(2)}. 
\end{equation*}
Since $i\colon D^*\to C^*\left(\sl(2,\bb{R});\mca{H}\otimes V\right)$ is $\so(2)$-equivariant with respect to $\mca{L}$ and $\mca{L}|_{\so(2)}$, we obtain a natural isomorphism 
\begin{equation*}
C^*\simeq C^*\left(\sl(2,\bb{R}),\so(2);\mca{H}\otimes V\right). 
\end{equation*}

Let $C\left(\mf{p}\oplus\mf{p}^*\right)$ be the Clifford algebra of $\mf{p}\oplus\mf{p}^*$ equipped with the symmetric bilinear form defined by the duality of $\mf{p}$ and $\mf{p}^*$. Let 
\begin{equation*}
\mca{B}=C\left(\mf{p}\oplus\mf{p}^*\right)\otimes U(\sl(2,\bb{R}))\otimes U(\sl(2,\bb{R})). 
\end{equation*}
We have a representation $\mca{B}\stackrel{\sigma\otimes\pi\otimes\xi}{\curvearrowright}D^*$, where $C\left(\mf{p}\oplus\mf{p}^*\right)\stackrel{\sigma}{\curvearrowright}\bigwedge^*\mf{p}^*$ is the spin representation. Therefore we have the following commutative diagram: 
\begin{equation*}
\begin{tikzcd}
C^*\left(\sl(2,\bb{R}),\so(2);\mca{H}\otimes V\right)&[-28pt]\subset&[-28pt]C^*\left(\sl(2,\bb{R});\mca{H}\otimes V\right)&[-28pt]\curvearrowleft&[-28pt]\mca{A}\\
C^*\ar[u,"\wr"',pos=0.4]&\subset&D^*\ar[u,hook]&\curvearrowleft&\mca{B}. 
\end{tikzcd}
\end{equation*}

Note that $H,\frac{1}{\sqrt{2}}(E+F)\in\mf{p}$ is an orthonormal basis of $\mf{p}$ with respect to $B$. See Section \ref{notation}. Define $\hat{\delta}_{MM}\in\mca{B}$ by 
\begin{align*}
\hat{\delta}_{MM}&=H\otimes H\otimes1+\frac{1}{2}(E+F)\otimes(E+F)\otimes1\\
&\quad-H\otimes1\otimes H-\frac{1}{2}(E+F)\otimes1\otimes(E+F)
\end{align*}
and set $\delta_{MM}=\left(\sigma\otimes\pi\otimes\xi\right)(\hat{\delta}_{MM})$. Then $\delta_{MM}$ is an operator on $D^*$ of degree $-1$. The operator $\delta_{MM}$ is the same as the operator in 2.3 of Borel--Wallach \cite{BW}. 

The following proposition describes the relation between $\delta$ and $\delta_{MM}$, where $\delta$ is the operator constructed in the previous section. 

\begin{prop}
We have 
\begin{equation*}
\delta|_{D^*}=\delta_{MM}
\end{equation*}
by identifying $D^*$ with a subspace of $C^*\left(\sl(2,\bb{R});\mca{H}\otimes V\right)$ through the isomorphism \eqref{natident}. 
\end{prop}

\begin{proof}
Following the decomposition $\sl(2,\bb{R})=\so(2)\oplus\mf{p}$, we get 
\begin{gather*}
E=\frac{1}{2}R+\frac{1}{2}(E+F),\quad F=-\frac{1}{2}R+\frac{1}{2}(E+F). 
\end{gather*}
On $D^*$ we have 
\begin{align*}
\delta&=\iota(E)\otimes\pi(F)\otimes\id+\iota(H)\otimes\pi(H)\otimes\id+\iota(F)\otimes\pi(E)\otimes\id\\
&\quad-\iota(E)\otimes\id\otimes\xi(F)-\iota(H)\otimes\id\otimes\xi(H)-\iota(F)\otimes\id\otimes\xi(E)\\
&=\iota\left(\frac{1}{2}(E+F)\right)\otimes\pi(F)\otimes\id+\iota(H)\otimes\pi(H)\otimes\id+\iota\left(\frac{1}{2}(E+F)\right)\otimes\pi(E)\otimes\id\\
&\quad-\iota\left(\frac{1}{2}(E+F)\right)\otimes\id\otimes\xi(F)-\iota(H)\otimes\id\otimes\xi(H)-\iota\left(\frac{1}{2}(E+F)\right)\otimes\id\otimes\xi(E)\\
&=\iota(H)\otimes\pi(H)\otimes\id+\iota\left(\frac{1}{2}(E+F)\right)\otimes\pi(E+F)\otimes\id\\
&\quad-\iota(H)\otimes\id\otimes\xi(H)-\iota\left(\frac{1}{2}(E+F)\right)\otimes\id\otimes\xi(E+F)\\
&=\delta_{MM}. 
\end{align*}
\end{proof}

\begin{cor}
The operator $\delta_{MM}$ preserves $C^*$. 
\end{cor}

\begin{proof}
This is because $\delta|_{D^*}=\delta_{MM}$ and $\delta$ preserves $C^*\left(\sl(2,\bb{R}),\so(2);\mca{H}\otimes V\right)$. 
\end{proof}

\subsection{Computation of $H^*\left(\Gamma\backslash\PSL(2,\bb{R});\bb{C}\right)$}
In this section we compute the de Rham cohomology $H^*(P;\bb{C})$, where $P=\Gamma\backslash\PSL(2,\bb{R})$ and $\Gamma$ is a torsion free cocompact lattice of $\PSL(2,\bb{R})$. (As we said before this cohomology can be computed by using the Serre spectral sequence of the circle bundle $P\to\Sigma$, but here we will compute it by using the representation theory.)

By Proposition \ref{drce} there is an isomorphism
\begin{equation*}
\Gamma\left(\bigwedge^*T^*P\otimes\bb{C}\right)\simeq C^*\left(\sl(2,\bb{R});C^\infty(P,\bb{C})\right)
\end{equation*}
of cochain complexes. Let 
\begin{equation*}
L^2(P)=W_0\oplus\bigoplus_{\mu\neq0}W_\mu
\end{equation*}
be the decomposition \eqref{decompooo}, where $W_\mu$ is the sum of the irreducible unitary subrepresentations of $L^2(P)$ with Casimir parameter $\mu$. This is a decomposition of a unitary representation of $\PSL(2,\bb{R})$. By Proposition \ref{cinfinityvect} and Lemma \ref{smooooth}, we obtain 
\begin{equation*}
C^\infty(P,\bb{C})=C^\infty(W_0)\oplus C^\infty\left(\bigoplus_{\mu\neq0}W_\mu\right), 
\end{equation*}
which is a decomposition of a representation of $\sl(2,\bb{R})$. Therefore 
\begin{equation*}
C^*\left(\sl(2,\bb{R});C^\infty(P,\bb{C})\right)\simeq C^*\left(\sl(2,\bb{R});C^\infty(W_0)\right)\oplus C^*\left(\sl(2,\bb{R});C^\infty\left(\bigoplus_{\mu\neq0}W_\mu\right)\right)
\end{equation*}
as cochain complexes. 

Put $\mca{H}=C^\infty\left(\bigoplus_{\mu\neq0}W_\mu\right)$. Applying Corollary \ref{ddddxi} to $\mca{H}$ and $V=\bb{R}$, there exists an operator $\delta$ on $C^*\left(\sl(2,\bb{R});\mca{H}\right)$ such that 
\begin{equation*}
d\delta+\delta d=\id\otimes\pi(\Omega), 
\end{equation*}
where $\pi$ is the representation of $\sl(2,\bb{R})$ on $\mca{H}$. We can show that $\id\otimes\pi(\Omega)$ is bijective on $C^*\left(\sl(2,\bb{R});\mca{H}\right)$ by following the proof of Lemma \ref{biject}. Putting $\delta^\prime=(\id\otimes\pi(\Omega))^{-1}\delta$, we have $d\delta^\prime+\delta^\prime d=\id$ on $C^*\left(\sl(2,\bb{R});\mca{H}\right)$. Therefore 
\begin{equation*}
H^*\left(\sl(2,\bb{R});\mca{H}\right)=0. 
\end{equation*}
It follows that 
\begin{equation}\label{wzero}
H^*\left(\sl(2,\bb{R});C^\infty(P,\bb{C})\right)\simeq H^*\left(\sl(2,\bb{R});C^\infty(W_0)\right). 
\end{equation}
By Section \ref{conclu} we have $W_0\simeq\bb{C}\oplus g_\Sigma(D_1^+\oplus D_1^-)$ as unitary representations of $\PSL(2,\bb{R})$, where $g_\Sigma$ is the genus of the surface $\Sigma=P/\PSO(2)$. Hence 
\begin{equation*}
C^\infty(W_0)\simeq\bb{C}\oplus g_\Sigma\left(C^\infty(D_1^+)\oplus C^\infty(D_1^-)\right)
\end{equation*}
by Lemma \ref{smooooth} and 
\begin{align*}
C^*\left(\sl(2,\bb{R});C^\infty(W_0)\right)\simeq C^*(\sl(2,\bb{R});\bb{C})&\oplus g_\Sigma C^*\left(\sl(2,\bb{R});C^\infty(D_1^+)\right)\\
&\oplus g_\Sigma C^*\left(\sl(2,\bb{R});C^\infty(D_1^-)\right). 
\end{align*}
By Section \ref{compirruni} we get the following. 

\begin{prop}
\begin{align*}
H^*(P;\bb{C})&\simeq H^*\left(\sl(2,\bb{R});C^\infty(P,\bb{C})\right)\\
&\simeq
\begin{cases}
\bb{C}&i=0,3\\
\bb{C}^{2g_\Sigma}&i=1,2\\
0&\text{otherwise}. 
\end{cases}
\end{align*}
\end{prop}

\subsection{The restriction map $H^*\left(\Gamma\backslash\PSL(2,\bb{R});\bb{C}\right)\to H^*(\mca{F};\bb{C})$}
In this section we compute the restriction map $H^*(P;\bb{C})\to H^*(\mca{F};\bb{C})$. This is the goal of Section \ref{restmapH} and will be applied in Section \ref{haefp}. 

Consider the commutative diagram 
\begin{equation*}
\begin{tikzcd}
\Gamma\left(\bigwedge^*T^*P\otimes\bb{C}\right)\ar[r,"\sim"]\ar[d]&C^*\left(\sl(2,\bb{R});C^\infty(P,\bb{C})\right)\ar[d]&C^*\left(\sl(2,\bb{R});C^\infty(W_0)\right)\ar[l]\ar[d]\\
\Gamma\left(\bigwedge^*T^*\mca{F}\otimes\bb{C}\right)\ar[r,"\sim"]&C^*\left(\mf{an};C^\infty(P,\bb{C})\right)&C^*\left(\mf{an};C^\infty(W_0)\right)\ar[l]
\end{tikzcd}
\end{equation*}
of cochain complexes and cochain maps, where the vertical maps are given by restriction. By taking cohomology and applying Theorem \ref{36kwlxisee} and \eqref{wzero}, we get 
\begin{equation*}
\begin{tikzcd}
H^*(P;\bb{C})\ar[r,"\sim"]\ar[d]&H^*\left(\sl(2,\bb{R});C^\infty(P,\bb{C})\right)\ar[d]&H^*\left(\sl(2,\bb{R});C^\infty(W_0)\right)\ar[l,"\sim"']\ar[d]\\
H^*(\mca{F};\bb{C})\ar[r,"\sim"]&H^*\left(\mf{an};C^\infty(P,\bb{C})\right)&H^*\left(\mf{an};C^\infty(W_0)\right).\ar[l,"\sim"']
\end{tikzcd}
\end{equation*}
Fix an isomorphism $C^\infty(W_0)\simeq\bb{C}\oplus g_\Sigma C^\infty(D_1^+)\oplus g_\Sigma C^\infty(D_1^-)$ of $\sl(2,\bb{R})$-modules. Then we have 
\begin{equation*}
\begin{tikzcd}
H^*\left(\sl(2,\bb{R});C^\infty(W_0)\right)\ar[r,"\sim"]\ar[d]&H^*(\sl(2,\bb{R});\bb{C})\oplus\bigoplus_{\epsilon=\pm}g_\Sigma H^*\left(\sl(2,\bb{R});C^\infty(D_1^\epsilon)\right)\ar[d]\\
H^*\left(\mf{an};C^\infty(W_0)\right)\ar[r,"\sim"]&H^*(\mf{an};\bb{C})\oplus\bigoplus_{\epsilon=\pm}g_\Sigma H^*\left(\mf{an};C^\infty(D_1^\epsilon)\right).
\end{tikzcd}
\end{equation*}

By Proposition \ref{restrichom}, the restriction map $H^*\left(\sl(2,\bb{R});C^\infty(D_1^\pm)\right)\to H^*\left(\mf{an};C^\infty(D_1^\pm)\right)$ is an isomorphism. 

By Section \ref{compirruni} and \ref{cohcomp}, we have 
\begin{equation*}
H^i\left(\sl(2,\bb{R});\bb{C}\right)\simeq
\begin{cases}
\bb{C}&i=0,3\\
0&\text{otherwise}
\end{cases}
\end{equation*}
and 
\begin{equation*}
H^i\left(\mf{an};\bb{C}\right)\simeq
\begin{cases}
\bb{C}&i=0,1\\
0&\text{otherwise}. 
\end{cases}
\end{equation*}
It follows that $H^i\left(\sl(2,\bb{R});\bb{C}\right)\to H^i\left(\mf{an};\bb{C}\right)$ is the identity map when $i=0,2$ and is $0$ otherwise. 

\begin{thm}\label{resthhi}
The homomorphism $H^i(P;\bb{C})\to H^i(\mca{F};\bb{C})$ induced by the restriction map is an isomorphism when $i=0,2$. 
\end{thm}

\section{Application: $\mca{F}$ is totally minimizable}\label{haefp}
In this section we prove that the foliation $\mca{F}$ is totally minimizable. 

\begin{dfn}
Let $\mca{G}$ be a $C^\infty$ foliation of a $C^\infty$ manifold $M$. Following Matsumoto \cite{Mat}, we say that $\mca{G}$ is totally minimizable if for any $C^\infty$ Riemannian metric $g$ of $T\mca{G}$, there exists a $C^\infty$ Riemannian metric $\tilde{g}$ on $M$ such that $\tilde{g}|_{T\mca{G}\times_MT\mca{G}}=g$ and every leaf of $\mca{G}$ is a minimal submanifold of the Riemannian manifold $(M,\tilde{g})$. 
\end{dfn}

Let $\Gamma$ be a torsion free cocompact lattice of $\PSL(2,\bb{R})$. Put $P=\Gamma\backslash\PSL(2,\bb{R})$ and let $\mca{F}$ be the orbit foliation of the action of $AN$ on $P$ by right multiplication. By identifying $P$ with the unit tangent bundle of the hyperbolic surface $\Sigma=P/\PSO(2)$, the foliation $\mca{F}$ is identified with the weak stable foliation of the geodesic flow of $\Sigma$. 

It seems that in Haefliger \cite{Haef} and Haefliger--Li \cite{HL}, Haefliger and Li were interested in the following problem. 

\begin{prob}\label{tomin}
Is $\mca{F}$ totally minimizable?
\end{prob}

Haefliger obtained the following weaker result. 

\begin{prop}[Haefliger \cite{Haef}]
For any $C^\infty$ Riemannian metric $g$ of $T\mca{F}$ and any open neighborhood $U$ of $g$ in $\Gamma\left(\bigwedge^2T^*\mca{F}\right)$ with the $C^\infty$ topology, there exists a $C^\infty$ Riemannian metric $\tilde{g}$ on $P$ such that $\tilde{g}|_{T\mca{F}\times_PT\mca{F}}\in U$ and every leaf of $\mca{F}$ is a minimal submanifold of $(P,\tilde{g})$.  
\end{prop}

This proposition follows from Corollary 5 of Theorem 4.1 in \cite{Haef} and the classification of holonomy invariant $1$-currents of $\mca{F}$ on page 281 of \cite{Haef}. 

We will prove the following theorem. 

\begin{thm}\label{totalfm}
The foliation $\mca{F}$ of $P$ is totally minimizable. 
\end{thm}

The approach by Haefliger and Li to Problem \ref{tomin} and our solution are both based on the following theorem. 

\begin{thm}[Rummler--Sullivan \cite{Ru,Su}]\label{rumsul}
Let $\mca{G}$ be a $p$-dimensional oriented $C^\infty$ foliation of a $C^\infty$ manifold $M$ and $g$ be a $C^\infty$ Riemannian metric of $T\mca{G}$. Let $\omega\in\Gamma\left(\bigwedge^pT^*\mca{G}\right)$ be the volume form of $g$. Then the following two conditions are equivalent: 
\begin{enumerate}
\item There exists a $C^\infty$ Riemannian metric $\tilde{g}$ on $M$ such that $\tilde{g}|_{T\mca{G}\times_MT\mca{G}}=g$ and every leaf of $\mca{G}$ is a minimal submanifold of $(M,\tilde{g})$
\item There exists a $C^\infty$ $p$-form $\tilde{\omega}$ on $M$ such that $\tilde{\omega}|_{T\mca{G}\times_M\cdots\times_MT\mca{G}}=\omega$ and 
\begin{equation*}
d\tilde{\omega}\left(X_1,\ldots,X_{p+1}\right)=0
\end{equation*}
for $X_1,\ldots,X_p\in\Gamma(T\mca{G})$ and $X_{p+1}\in\Gamma(TM)$. 
\end{enumerate}
\end{thm}

A short proof of this theorem can be found in Haefliger \cite{Haef}. 

Though we use the same theorem as Haefliger and Li, our approach is different from theirs. To consider whether $\omega$ has an extension $\tilde{\omega}$ satisfying the second condition of Theorem \ref{rumsul}, Haefliger and Li translate the problem into a problem of certain $1$-form $d\int_\mca{F}\omega$ on a complete transversal of the foliation $\mca{F}$. Instead of dealing with objects on a transversal, we consider the problem directly by dealing with leafwise objects. 

\begin{cor}[Matsumoto \cite{Mat}]
Let $\mca{G}$ be a $p$-dimensional oriented $C^\infty$ foliation of a $C^\infty$ manifold $M$. If the restriction map 
\begin{equation*}
H^p(M;\bb{R})\to H^p(\mca{G};\bb{R})
\end{equation*}
is surjective, the foliation $\mca{G}$ is totally minimizable. 
\end{cor}

\begin{proof}
The proof is taken from \cite{Mat} for the convenience of the reader. Let $r$ denote the restriction maps. Take any $C^\infty$ Riemannian metric $g$ of $T\mca{G}$ and let $\omega\in\Gamma\left(\bigwedge^pT^*\mca{G}\right)$ be the volume form of $g$. We have $[\omega]\in H^p(\mca{G};\bb{R})$. By assumption, there exists $[\omega^\prime]\in H^p(M;\bb{R})$ such that $r[\omega^\prime]=[\omega]$. So we have $r\omega^\prime=\omega+d_\mca{G}\eta$ for some $\eta\in\Gamma\left(\bigwedge^{p-1}T^*\mca{G}\right)$. We can take $\eta^\prime\in\Gamma\left(\bigwedge^{p-1}T^*M\right)$ such that $r\eta^\prime=\eta$. Define $\tilde{\omega}\in\Gamma(\bigwedge^pT^*M)$ by $\tilde{\omega}=\omega^\prime-d\eta^\prime$. Then 
\begin{equation*}
r\tilde{\omega}=\omega+d_\mca{G}\eta-d_\mca{G}r\eta^\prime=\omega
\end{equation*}
and $d\tilde{\omega}=0$. Hence $\tilde{\omega}$ satisfies the second condition of Theorem \ref{rumsul}. 
\end{proof}

%\begin{rem}
%The converse of the above corollary is true if the foliation $\mca{G}$ is of codimension $1$. 
%\end{rem}

So Theorem \ref{totalfm} follows from the next lemma. 

\begin{lem}
The restriction map 
\begin{equation*}
H^2(P;\bb{R})\to H^2(\mca{F};\bb{R})
\end{equation*}
is an isomorphism. 
\end{lem}

\begin{proof}
We have a commutative diagram 
\begin{equation*}
\begin{tikzcd}
\Gamma\left(\bigwedge^*T^*P\otimes\bb{C}\right)\ar[r]\ar[d,equal]&\Gamma\left(\bigwedge^*T^*\mca{F}\otimes\bb{C}\right)\ar[d,equal]\\[-3pt]
\Gamma\left(\bigwedge^*T^*P\right)\oplus i\Gamma\left(\bigwedge^*T^*P\right)\ar[r]&\Gamma\left(\bigwedge^*T^*\mca{F}\right)\oplus i\Gamma\left(\bigwedge^*T^*\mca{F}\right)
\end{tikzcd}
\end{equation*}
of real cochain complexes, where the horizontal maps are given by restriction. It induces 
\begin{equation*}
\begin{tikzcd}
H^*(P;\bb{C})\ar[r]\ar[d,equal]&H^*(\mca{F};\bb{C})\ar[d,equal]\\[-3pt]
H^*(P;\bb{R})\oplus iH^*(P;\bb{R})\ar[r]\ar[d,dash,"\wr"]&H^*(\mca{F};\bb{R})\oplus iH^*(\mca{F};\bb{R})\ar[d,dash,"\wr"]\\[-3pt]
H^*(P;\bb{R})\otimes\bb{C}\ar[r]&H^*(\mca{F};\bb{R})\otimes\bb{C}.
\end{tikzcd}
\end{equation*}
So the lemma follows from Theorem \ref{resthhi}. 
\end{proof}

\section{Computation of $H^*(\mca{F};\mf{an})$ by a Hodge decomposition}\label{hodgean}
We will compute the de Rham cohomology $H^*(\mca{F};\mf{an})$ for the adjoint representation $\mf{an}\stackrel{\ad}{\curvearrowright}\mf{an}$ by giving a ``Hodge decomposition'' of $C^*(\mf{an};C^\infty(P,\bb{C})\otimes\mf{an})$. Since $\mf{n}$ acts nontrivially on $\mf{an}$, we cannot apply the results in Section \ref{construc}. 

However, if we want only the computation of $H^*(\mca{F};\mf{an})$, we can compute it as follows. In this case we do not obtain a Hodge decomposition. 

There is a short exact sequence 
\begin{equation*}
0\to\mf{n}\to\mf{an}\to\mf{an}/\mf{n}\to0
\end{equation*}
of $\mf{an}$-modules. This is isomorphic to 
\begin{equation*}
0\to\bb{R}_\frac{1}{\sqrt{2}}\to\mf{an}\to\bb{R}_0\to0, 
\end{equation*}
which, together with the computation results of $H^*\left(\mca{F};\bb{R}_\frac{1}{\sqrt{2}}\right)$ and $H^*(\mca{F};\bb{R}_0)$ in Section \ref{comprrrr}, gives a long exact sequence 
\begin{equation*}
\begin{tikzcd}
0&&\\
\bb{R}\ar[r]&H^2(\mca{F};\mf{an})\ar[r]&\bb{R}^{2g_\Sigma}\ar[llu,shorten <= 1ex, shorten >= 1ex]\\
\bb{R}\ar[r]&H^1(\mca{F};\mf{an})\ar[r]&\bb{R}^{2g_\Sigma+1}\ar[llu,shorten <= 0.5ex, shorten >= 1ex]\\
0\ar[r]&H^0(\mca{F};\mf{an})\ar[r]&\bb{R}. \ar[llu,shorten <= 1ex, shorten >= 1ex]
\end{tikzcd}
\end{equation*}
It is easy to show $H^0(\mca{F};\mf{an})=0$. In \cite{A}, Asaoka showed $H^1(\mca{F};\mf{an})\simeq\bb{R}^{2g_\Sigma}$ by using a result in Matsumoto--Mitsumatsu \cite{MM}. Therefore we get $H^2(\mca{F};\mf{an})\simeq\bb{R}^{2g_\Sigma}$ from the long exact sequence.

\subsection{Hodge decomposition of $C^*\left(\mf{an};\mca{H}\otimes\mf{an}\right)$ for generic $\mca{H}$}\label{an1h}
Let $\sl(2,\bb{R})\stackrel{\pi}{\curvearrowright}\mca{H}$ be a representation on a complex vector space $\mca{H}$. Consider the tensor product $\mf{an}\curvearrowright\mca{H}\otimes\mf{an}$ of the restriction $\mf{an}\curvearrowright\mca{H}$ and the adjoint representation $\mf{an}\stackrel{\ad}{\curvearrowright}\mf{an}$. In this section we will construct an operator $\delta$ on $C^*\left(\mf{an};\mca{H}\otimes\mf{an}\right)$ of degree $-1$ such that 
\begin{equation*}
d\delta+\delta d=\id\otimes\pi(\Omega)^2\otimes\id, 
\end{equation*}
where $\Omega$ is the Casimir element of $\sl(2,\bb{R})$. 

Let $\ad\colon U(\mf{an})\to\End(\mf{an})$ be the extension of the adjoint representation $\ad\colon\mf{an}\to\End(\mf{an})$ and $I$ be the kernel of $\ad\colon U(\mf{an})\to\End(\mf{an})$. Consider the algebra 
\begin{equation*}
\mca{A}=C(\mf{an}\oplus\mf{an}^*)\otimes U(\sl(2,\bb{R}))\otimes U(\mf{an})/I, 
\end{equation*}
which has a natural representation 
\begin{equation*}
\mca{A}\stackrel{\sigma\otimes\pi\otimes\ad}{\curvearrowright}C^*\left(\mf{an};\mca{H}\otimes\mf{an}\right)=\bigwedge^*\mf{an}^*\otimes\mca{H}\otimes\mf{an}. 
\end{equation*}
Define an element $\hat{d}\in\mca{A}$ by 
\begin{equation*}
\hat{d}=\hat{d}_\star\otimes1\otimes1+\theta_H\otimes H\otimes1+\theta_H\otimes1\otimes H+\theta_E\otimes E\otimes1+\theta_E\otimes1\otimes E, 
\end{equation*}
which satisfies $\left(\sigma\otimes\pi\otimes\ad\right)(\hat{d})=d$ for the differential $d$. Define $\hat{\delta}\in\mca{A}$ by 
\begin{align*}
\hat{\delta}&=\hat{d}_\star^\top\otimes\Omega\otimes1-6\hat{d}_\star^\top\otimes F\otimes E+H\otimes H\Omega\otimes1-\sqrt{2}H\otimes F\otimes E\\
&\quad-2H\otimes FH\otimes E+\frac{1}{\sqrt{2}}H\otimes\Omega\otimes1-H\otimes\Omega\otimes H\\
&\quad+2E\otimes F\Omega\otimes1-4E\otimes F^2\otimes E
\end{align*}
and put $\delta=\left(\sigma\otimes\pi\otimes\ad\right)(\hat{\delta})$, which is an operator of degree $-1$ on $C^*\left(\mf{an};\mca{H}\otimes\mf{an}\right)$. 

\begin{lem}\label{ellvvisl}
We have 
\begin{equation*}
H^2=\frac{1}{\sqrt{2}}H,\quad E^2=0,\quad HE=\frac{1}{\sqrt{2}}E,\quad EH=0
\end{equation*}
in $U(\mf{an})/I$. 
\end{lem}

\begin{proof}
The formulas follow from easy computation. 
\end{proof}

\begin{thm}
We have 
\begin{equation*}
\hat{d}\hat{\delta}+\hat{\delta}\hat{d}=1\otimes\Omega^2\otimes1
\end{equation*}
in $\mca{A}$. 
\end{thm}

\begin{proof}
A proof can be given by a similar computation as in the proof of Theorems \ref{mtrm} and \ref{5555} using Lemma \ref{ellvvisl}. We omit the proof. 
\end{proof}

\begin{cor}\label{cofff}
We have 
\begin{equation*}
d\delta+\delta d=\id\otimes\pi(\Omega)^2\otimes\id
\end{equation*}
on $C^*\left(\mf{an};\mca{H}\otimes\mf{an}\right)$. 
\end{cor}

\begin{prop}
We have $\hat{\delta}^2=0$ in $\mca{A}$. 
\end{prop}

\begin{proof}
This follows from a direct computation as in the proof of Proposition \ref{squarezero} using Lemma \ref{ellvvisl}. We omit the proof. 
\end{proof}

\subsection{Hodge decomposition of $C^*\left(\mf{an};\bb{C}\otimes\mf{an}\right)$}\label{trivialhoooo}
We take the following elements as a basis of $C^*\left(\mf{an};\bb{C}\otimes\mf{an}\right)$: 
\begin{gather*}
1\otimes H,\ \ \ 1\otimes E, \\
\theta_H\otimes H,\ \ \ \theta_H\otimes E,\ \ \ \theta_E\otimes H,\ \ \ \theta_E\otimes E, \\
\theta_H\wedge\theta_E\otimes H,\ \ \ \theta_H\wedge\theta_E\otimes E. 
\end{gather*}
Let $d_i\colon C^i\left(\mf{an};\bb{C}\otimes\mf{an}\right)\to C^{i+1}\left(\mf{an};\bb{C}\otimes\mf{an}\right)$ be the differential. Then we have 
\begin{equation*}
d_0=
\begin{pmatrix}
0&0\\
0&\frac{1}{\sqrt{2}}\\
0&0\\
-\frac{1}{\sqrt{2}}&0
\end{pmatrix}
,\quad d_1=
\begin{pmatrix}
0&0&-\frac{1}{\sqrt{2}}&0\\
\frac{1}{\sqrt{2}}&0&0&0
\end{pmatrix}
\end{equation*}
with respect to the basis. 

Let $\delta_i\colon C^i\left(\mf{an};\bb{C}\otimes\mf{an}\right)\to C^{i-1}\left(\mf{an};\bb{C}\otimes\mf{an}\right)$ be a linear map. Then 
\begin{equation*}
d\delta+\delta d=\id
\end{equation*}
if and only if 
\begin{equation*}
\delta_1=
\begin{pmatrix}
a&0&b&-\sqrt{2}\\
c&\sqrt{2}&e&0
\end{pmatrix}
,\quad\delta_2=
\begin{pmatrix}
0&\sqrt{2}\\
e&-c\\
-\sqrt{2}&0\\
-b&a
\end{pmatrix}
\end{equation*}
for some $a$, $b$, $c$, $e\in\bb{C}$. This $\delta$ satisfies $\delta^2=0$. Therefore 
\begin{align*}
C^*\left(\mf{an};\bb{C}\otimes\mf{an}\right)&=\im d\oplus\im\delta\\
\ker d&=\im d\\
\ker\delta&=\im\delta. 
\end{align*}
In particular, 
\begin{equation*}
H^*\left(\mf{an};\bb{C}\otimes\mf{an}\right)=0. 
\end{equation*}

\subsection{Hodge decomposition of $C^*\left(\mf{an};C^\infty(D_1^\pm)\otimes\mf{an}\right)$}\label{hood1pm}
In this section we construct a ``Hodge decomposition'' of $C^*\left(\mf{an};C^\infty(D_1^\pm)\otimes\mf{an}\right)$. 

The adjoint representation $\ad\colon\mf{an}\to\End(\mf{an})$ extends to an algebra homomorphism $\ad\colon U(\mf{an})\to\End(\mf{an})$. Let $I$ be the kernel of $\ad\colon U(\mf{an})\to\End(\mf{an})$. Consider the algebra 
\begin{equation*}
\mca{A}=C(\mf{an}\oplus\mf{an}^*)\otimes\End\left(C^\infty(D_1^\pm)\right)\otimes U(\mf{an})/I, 
\end{equation*}
where $C(\mf{an}\oplus\mf{an}^*)$ is the Clifford algebra and $\End\left(C^\infty(D_1^\pm)\right)$ is the algebra of all $\bb{C}$-linear maps from $C^\infty(D_1^\pm)$ to itself. As before we have a natural representation $\mca{A}\curvearrowright C^*\left(\mf{an};C^\infty(D_1^\pm)\otimes\mf{an}\right)$. 

Let $\pi$ denote the representation $\sl(2,\bb{R})\curvearrowright C^\infty(D_1^\pm)$. By Proposition \ref{ffoorr} there exist a $\bb{C}$-linear map $\varphi\colon C^\infty(D_1^\pm)\to\bb{C}$ and an element $h\in C^\infty(D_1^\pm)$ such that 
\begin{equation*}
\ker\varphi=\pi(E)C^\infty(D_1^\pm),\quad\pi(H)\varphi=-\frac{1}{\sqrt{2}}\varphi,\quad\varphi(h)=1. 
\end{equation*}
As we explained in Sections \ref{geneconsss} and \ref{2generc}, we can define operators $S,T\colon C^\infty(D_1^\pm)\to C^\infty(D_1^\pm)$ by 
\begin{gather*}
S=\pi(E)^{-1}\left(\pi(H)-\frac{1}{\sqrt{2}}\right),\quad T=\pi(E)^{-1}\left(\id-\varphi h\right). 
\end{gather*}

Define elements $\hat{d}$, $\hat{\delta}\in\mca{A}$ by 
\begin{equation*}
\hat{d}=\hat{d}_\star\otimes\id\otimes1+\theta_H\otimes\pi(H)\otimes1+\theta_H\otimes\id\otimes H+\theta_E\otimes\pi(E)\otimes1+\theta_E\otimes\id\otimes E
\end{equation*}
and 
\begin{align*}
\hat{\delta}&=2\hat{d}_\star^\top\otimes T\otimes E+2H\otimes\id\otimes H-2E\otimes S\otimes H\\
&\quad+E\otimes T\otimes\left(1-\sqrt{2}H\right)+\sqrt{2}E\otimes ST\otimes E. 
\end{align*}
The element $\hat{d}$ operates as the differential $d$ on the complex $C^*\left(\mf{an};C^\infty(D_1^\pm)\otimes\mf{an}\right)$ and let $\delta$ be the operator determined by $\hat{\delta}$. 

\begin{thm}\label{ddddpppi}
We have 
\begin{equation*}
\hat{d}\hat{\delta}+\hat{\delta}\hat{d}=1\otimes\id\otimes1-\hat{p}
\end{equation*}
in $\mca{A}$, where $\hat{p}$ is an element of $\mca{A}$ defined by 
\begin{align*}
\hat{p}&=\left(\theta_E\pm\sqrt{2}i\theta_H\right)E\otimes\varphi h\otimes\left(1-\sqrt{2}H\pm2iE\right)\\
&\quad+\theta_H\theta_EHE\otimes\varphi h\otimes\left(\pm2iE\right). 
\end{align*}
Here we identify $\mca{A}$ with 
\begin{equation*}
\left(C(\mf{an}\oplus\mf{an}^*)\otimes\bb{C}\right)\otimes_\bb{C}\End\left(C^\infty(D_1^\pm)\right)\otimes_\bb{C}\left(U(\mf{an})/I\otimes\bb{C}\right). 
\end{equation*}
\end{thm}

\begin{proof}
The theorem is obtained by a direct computation as in the proof of Theorems \ref{mtrm} and  \ref{5555} using Proposition \ref{yuuui}, Lemma \ref{ellvvisl} and the following formulas from Sections \ref{1-n2} and \ref{1-n22}: 
\begin{gather*}
[\pi(H),S]=-\frac{1}{\sqrt{2}}S,\quad[\pi(E),S]=-\frac{1}{\sqrt{2}}\\
[\pi(H),T]=-\varphi Sh-\frac{1}{\sqrt{2}}T,\quad[\pi(E),T]=-\varphi h\\
S\pi(E)=\pi(H),\quad\pi(E)S=\pi(H)-\frac{1}{\sqrt{2}}\\
T\pi(H)=\pi(E)^{-1}\left(\pi(H)-\frac{1}{\sqrt{2}}\varphi h\right),\quad T\pi(E)=\id\\
Sh=\pm\sqrt{2}ih. 
\end{gather*}
It may take three to four hours to complete this computation. 
\end{proof}

\begin{prop}\label{dee2000}
We have $\hat{\delta}^2=0$ in $\mca{A}$. 
\end{prop}

\begin{proof}
This follows from an easy computation. We omit a proof. 
\end{proof}

Let $p\colon C^*\left(\mf{an};C^\infty(D_1^\pm)\otimes\mf{an}\right)\to C^*\left(\mf{an};C^\infty(D_1^\pm)\otimes\mf{an}\right)$ be the operator obtained from $\hat{p}$ through the representation of $\mca{A}$. 

\begin{prop}\label{unipt}
We have $\hat{p}^2=\hat{p}$ in $\mca{A}$. Hence $p$ is a projection. 
\end{prop}

\begin{proof}
This is proved by an easy computation as well. 
\end{proof}

\begin{prop}\label{imageofp}
The image of the projection $p$ is given by 
\begin{align*}
(\im p)^0&=0\\
(\im p)^1&=\bb{C}\left(\theta_E\pm\sqrt{2}i\theta_H\right)\otimes h\otimes\left(H\mp\sqrt{2}iE\right)\\
(\im p)^2&=\bb{C}\left(\theta_H\wedge\theta_E\right)\otimes h\otimes H, 
\end{align*}
where 
\begin{equation*}
(\im p)^i=\im p\cap C^i\left(\mf{an};C^\infty(D_1^\pm)\otimes\mf{an}\right)
\end{equation*}
and we identify $C^*\left(\mf{an};C^\infty(D_1^\pm)\otimes\mf{an}\right)$ with 
\begin{equation*}
\left(\bigwedge^*\mf{an}^*\otimes\bb{C}\right)\otimes_\bb{C}C^\infty(D_1^\pm)\otimes_\bb{C}(\mf{an}\otimes\bb{C}). 
\end{equation*}
\end{prop}

\begin{proof}
By easy computations we get 
\begin{equation*}
p\left(1\otimes f\otimes H+1\otimes g\otimes E\right)=0, 
\end{equation*}
\begin{align*}
&p\left(\theta_H\otimes f\otimes H+\theta_H\otimes g\otimes E+\theta_E\otimes k\otimes H+\theta_E\otimes l\otimes E\right)\\
&=\left(\theta_E\pm\sqrt{2}i\theta_H\right)\otimes\varphi(k)h\otimes\left(H\mp\sqrt{2}iE\right)
\end{align*}
and 
\begin{align*}
&p\left(\left(\theta_H\wedge\theta_E\right)\otimes f\otimes H+\left(\theta_H\wedge\theta_E\right)\otimes g\otimes E\right)\\
&=\left(\theta_H\wedge\theta_E\right)\otimes\varphi(f)h\otimes H, 
\end{align*}
from which the proposition follows. 
\end{proof}

\begin{lem}\label{popopo}
We have $dp=0$ and $\delta p=0$. 
\end{lem}

\begin{proof}
Following Proposition \ref{imageofp}, it suffices to show 
\begin{align*}
d\left(\left(\theta_E\pm\sqrt{2}i\theta_H\right)\otimes h\otimes\left(H\mp\sqrt{2}iE\right)\right)&=0\\
\delta\left(\left(\theta_E\pm\sqrt{2}i\theta_H\right)\otimes h\otimes\left(H\mp\sqrt{2}iE\right)\right)&=0\\
\delta\left(\left(\theta_H\wedge\theta_E\right)\otimes h\otimes H\right)&=0. 
\end{align*}
These equations are proved by direct computations. For the first equation, we use the definition of $\hat{d}$, Equations \eqref{fffooo} and \eqref{2ihheh}. For the second and third ones, we use the definition of $\delta$, $Th=0$ and $Sh=\pm\sqrt{2}ih$. 
\end{proof}

\begin{lem}\label{pdpd0aa}
We have $\hat{p}\hat{d}=0$ and $\hat{p}\hat{\delta}=0$ in $\mca{A}$. 
\end{lem}

\begin{proof}
Recall the definitions: 
\begin{equation*}
\hat{p}=\left(\theta_E\pm\sqrt{2}i\theta_H\right)E\otimes\varphi h\otimes\left(1-\sqrt{2}H\pm2iE\right)+\theta_H\theta_EHE\otimes\varphi h\otimes\left(\pm2iE\right), 
\end{equation*}
\begin{equation*}
\hat{d}=\hat{d}_\star\otimes\id\otimes1+\theta_H\otimes\pi(H)\otimes1+\theta_H\otimes\id\otimes H+\theta_E\otimes\pi(E)\otimes1+\theta_E\otimes\id\otimes E, 
\end{equation*}
\begin{align*}
\hat{\delta}&=2\hat{d}_\star^\top\otimes T\otimes E+2H\otimes\id\otimes H-2E\otimes S\otimes H\\
&\quad+E\otimes T\otimes\left(1-\sqrt{2}H\right)+\sqrt{2}E\otimes ST\otimes E. 
\end{align*}
We have $\hat{p}\hat{\delta}=0$ since 
\begin{gather*}
\left(1-\sqrt{2}H\pm2iE\right)H=0,\quad\left(1-\sqrt{2}H\pm2iE\right)E=0\\
EH=0,\quad E^2=0
\end{gather*}
in $U(\mf{an})/I\otimes\bb{C}$ and $E^2=0$ in $C(\mf{an}\oplus\mf{an}^*)\otimes\bb{C}$. By these equations and $\varphi\circ\pi(E)=0$, we get 
\begin{align*}
&\hat{p}\hat{d}\\
&=\left(\left(\theta_E\pm\sqrt{2}i\theta_H\right)E\otimes\varphi h\otimes\left(1-\sqrt{2}H\pm2iE\right)+\theta_H\theta_EHE\otimes\varphi h\otimes\left(\pm2iE\right)\right)\\
&\quad\left(\hat{d}_\star\otimes\id\otimes1+\theta_H\otimes\pi(H)\otimes1\right). 
\end{align*}
By a computation using $\hat{d}_\star=-\frac{1}{\sqrt{2}}\theta_H\theta_EE$ and $\varphi\circ\pi(H)=\frac{1}{\sqrt{2}}\varphi$, we see $\hat{p}\hat{d}=0$. 
\end{proof}

\begin{thm}
We have 
\begin{align*}
C^*\left(\mf{an};C^\infty(D_1^\pm)\otimes\mf{an}\right)&=\im p\oplus\im d\oplus\im\delta\\
\ker d&=\im p\oplus\im d\\
\ker\delta&=\im p\oplus\im\delta\\
\ker d\cap\ker\delta&=\im p\\
\im d\oplus\im\delta&=\ker p. 
\end{align*}
An element $x$ of $C^*\left(\mf{an};C^\infty(D_1^\pm)\otimes\mf{an}\right)$ is decomposed as 
\begin{equation*}
x=p(x)+d\delta x+\delta dx
\end{equation*}
with respect to the first decomposition. The second decomposition implies 
\begin{equation*}
H^*\left(\mf{an};C^\infty(D_1^\pm)\otimes\mf{an}\right)\simeq\im p. 
\end{equation*}
\end{thm}

\begin{proof}
This follows from Theorem \ref{ddddpppi}, Proposition \ref{dee2000}, Proposition \ref{unipt}, Lemma \ref{popopo} and Lemma \ref{pdpd0aa}. 
\end{proof}

\subsection{Computation of $H^*(\mca{F};\mf{an})$}\label{comphfan}
In this section we give a ``Hodge decomposition'' of $C^*\left(\mf{an};C^\infty(P,\bb{C})\otimes\mf{an}\right)$ and determine $H^*(\mca{F};\mf{an})$ as a consequence. 

Since $C^\infty(P,\bb{R})\otimes(\mf{an}\otimes\bb{C})\simeq C^\infty(P,\bb{C})\otimes\mf{an}$ as representations of $\mf{an}$, 
\begin{align*}
\Gamma\left(\bigwedge^*T^*\mca{F}\otimes(\mf{an}\otimes\bb{C})\right)&\simeq C^*\left(\mf{an};C^\infty(P,\bb{R})\otimes(\mf{an}\otimes\bb{C})\right)\\
&\simeq C^*\left(\mf{an};C^\infty(P,\bb{C})\otimes\mf{an}\right). 
\end{align*}
The cohomology of the left hand side is $H^*(\mca{F};\mf{an}\otimes\bb{C})$. 

By Section \ref{beginning} we have an equivalence 
\begin{equation*}
C^\infty(P,\bb{C})\simeq\bb{C}\oplus g_\Sigma\left(C^\infty(D_1^+)\oplus C^\infty(D_1^-)\right)\oplus\mca{H}
\end{equation*}
of representations of $\sl(2,\bb{R})$, where the Casimir element $\Omega$ of $\sl(2,\bb{R})$ acts as a bijection on $\mca{H}$. The cochain complex decomposes as well: 
\begin{align*}
C^*\left(\mf{an};C^\infty(P,\bb{C})\otimes\mf{an}\right)&\simeq C^*\left(\mf{an};\bb{C}\otimes\mf{an}\right)\\
&\quad\oplus g_\Sigma\left(C^*\left(\mf{an};C^\infty(D_1^+)\otimes\mf{an}\right)\oplus C^*\left(\mf{an};C^\infty(D_1^-)\otimes\mf{an}\right)\right)\\
&\quad\oplus C^*\left(\mf{an};\mca{H}\otimes\mf{an}\right). 
\end{align*}
``Hodge decompositions'' of $C^*\left(\mf{an};\bb{C}\otimes\mf{an}\right)$ and $C^*\left(\mf{an};C^\infty(D_1^\pm)\otimes\mf{an}\right)$ are given in Sections \ref{trivialhoooo} and \ref{hood1pm}. Only the second one has a nontrivial cohomology. 

To give a ``Hodge decomposition'' of $C^*\left(\mf{an};\mca{H}\otimes\mf{an}\right)$, let $\delta$ be the operator in Section \ref{an1h} for the complex $C^*\left(\mf{an};\mca{H}\otimes\mf{an}\right)$. It satisfies 
\begin{equation*}
d\delta+\delta d=\id\otimes\pi(\Omega)^2\otimes\id, 
\end{equation*}
where $\pi$ denotes the representation $\sl(2,\bb{R})\curvearrowright\mca{H}$. We can define 
\begin{equation*}
\delta^\prime=(\id\otimes\pi(\Omega)^{-2}\otimes\id)\delta
\end{equation*}
since $\pi(\Omega)$ is bijective. Then $d\delta^\prime+\delta^\prime d=\id$ and $(\delta^\prime)^2=0$. Hence we get the following ``Hodge decomposition'': 
\begin{gather*}
C^*\left(\mf{an};\mca{H}\otimes\mf{an}\right)=\im d\oplus\im\delta^\prime, \\
\ker d=\im d,\quad\ker\delta^\prime=\im\delta^\prime, \\
H^*(\mf{an};\mca{H}\otimes\mf{an})=0. 
\end{gather*}
Therefore we obtained a ``Hodge decomposition'' of the entire cochain complex $C^*\left(\mf{an};C^\infty(P,\bb{C})\otimes\mf{an}\right)$ by combining the four decompositions. As a corollary we get a computation of the cohomology. 

\begin{cor}\label{ancohrrreee}
We have 
\begin{align*}
H^i(\mca{F};\mf{an}\otimes\bb{C})\simeq
\begin{cases}
\bb{C}^{2g_\Sigma}&i=1,2\\
0&\text{otherwise}
\end{cases}
\end{align*}
and 
\begin{align*}
H^i(\mca{F};\mf{an})\simeq
\begin{cases}
\bb{R}^{2g_\Sigma}&i=1,2\\
0&\text{otherwise}. 
\end{cases}
\end{align*}
\end{cor}

\begin{proof}
The second isomorphism follows from the first since we have 
\begin{equation*}
H^*(\mca{F};\mf{an}\otimes\bb{C})\simeq H^*(\mca{F};\mf{an})\otimes\bb{C}
\end{equation*}
by \eqref{realcomplex} in Section \ref{acofpcf}. 
\end{proof}

\section{Computation of $H^*\left(\mca{F};\mf{sl}(2,\bb{R})\right)$}\label{slslsls2rr}
In this section we compute $H^*\left(\mca{F};\mf{sl}(2,\bb{R})\right)$ for the adjoint representation $\mf{an}\stackrel{\ad}{\curvearrowright}\sl(2,\bb{R})$. 

By \eqref{realcomplex} in Section \ref{acofpcf} and Proposition \ref{drce}, we have 
\begin{align}
H^*\left(\mca{F};\sl(2,\bb{R})\right)\otimes\bb{C}&\simeq H^*\left(\mca{F};\mf{sl}(2,\bb{R})\otimes\bb{C}\right)\nonumber\\
&\simeq H^*\left(\mf{an};C^\infty(P,\bb{R})\otimes\left(\mf{sl}(2,\bb{R})\otimes\bb{C}\right)\right)\nonumber\\
&\simeq H^*\left(\mf{an};C^\infty(P,\bb{C})\otimes\mf{sl}(2,\bb{R})\right). \label{lxifne}
\end{align}
From the short exact sequence 
\begin{equation*}
0\to\mf{an}\to\sl(2,\bb{R})\to\sl(2,\bb{R})/\mf{an}\to0
\end{equation*}
and an isomorphism $\sl(2,\bb{R})/\mf{an}\simeq\bb{R}_{-\frac{1}{\sqrt{2}}}$ of representations of $\mf{an}$, we get a long exact sequence 
\begin{align*}
\cdots\to H^i\left(\mf{an};C^\infty(P,\bb{C})\otimes\mf{an}\right)&\to H^i\left(\mf{an};C^\infty(P,\bb{C})\otimes\sl(2,\bb{R})\right)\\
&\to H^i\left(\mf{an};C^\infty(P,\bb{C})\otimes\bb{R}_{-\frac{1}{\sqrt{2}}}\right)\to\cdots. 
\end{align*}

\begin{lem}
The connecting homomorphisms 
\begin{equation}
H^i\left(\mf{an};C^\infty(P,\bb{C})\otimes\bb{R}_{-\frac{1}{\sqrt{2}}}\right)\to H^{i+1}\left(\mf{an};C^\infty(P,\bb{C})\otimes\mf{an}\right)\label{connectingh}
\end{equation}
are zero. 
\end{lem}

\begin{proof}
Recall that by Theorem \ref{36kwlxisee} and \eqref{ywolxi} in Section \ref{beginning}, 
\begin{align*}
H^*\left(\mf{an};C^\infty(P,\bb{C})\otimes\bb{R}_{-\frac{1}{\sqrt{2}}}\right)&\simeq(3g_\Sigma-3)H^*\left(\mf{an};C^\infty(D_2^+)\otimes\bb{R}_{-\frac{1}{\sqrt{2}}}\right)\\
&\quad\oplus(3g_\Sigma-3)H^*\left(\mf{an};C^\infty(D_2^-)\otimes\bb{R}_{-\frac{1}{\sqrt{2}}}\right)
\end{align*}
and by Section \ref{comphfan}, 
\begin{equation*}
H^*\left(\mf{an};C^\infty(P,\bb{C})\otimes\mf{an}\right)\simeq g_\Sigma H^*\left(\mf{an};C^\infty(D_1^+)\otimes\mf{an}\right)\oplus g_\Sigma H^*\left(\mf{an};C^\infty(D_1^-)\otimes\mf{an}\right). 
\end{equation*}
By Theorem \ref{unidecompkk} and an argument in Section \ref{beginning}, we have a decomposition 
\begin{align}
C^\infty(P,\bb{C})&=g_\Sigma C^\infty(D_1^+)\oplus g_\Sigma C^\infty(D_1^-)\nonumber\\
&\quad\oplus(3g_\Sigma-3)C^\infty(D_2^+)\oplus(3g_\Sigma-3)C^\infty(D_2^-)\oplus W\label{deccco}
\end{align}
of representation of $\sl(2,\bb{R})$, where $W$ is a subrepresentation of $C^\infty(P,\bb{C})$. By the definition of the connecting homomorphism, the map \eqref{connectingh} decomposes with respect to the decompositions of the cohomologies obtained by \eqref{deccco}. The homomorphisms 
\begin{equation*}
H^i\left(\mf{an};C^\infty(D_n^\pm)\otimes\bb{R}_{-\frac{1}{\sqrt{2}}}\right)\to H^{i+1}\left(\mf{an};C^\infty(D_n^\pm)\otimes\mf{an}\right)
\end{equation*}
for $n=1,2$ and 
\begin{equation*}
H^i\left(\mf{an};W\otimes\bb{R}_{-\frac{1}{\sqrt{2}}}\right)\to H^{i+1}\left(\mf{an};W\otimes\mf{an}\right)
\end{equation*}
are zero, since the domain or the codomain is zero in each case. Hence the map \eqref{connectingh} is zero as well. 
\end{proof}

Therefore, by Section \ref{comprrrr} and \ref{hodgean}, 
\begin{align*}
H^i\left(\mca{F};\sl(2,\bb{R})\otimes\bb{C}\right)&\simeq H^i\left(\mf{an};C^\infty(P,\bb{C})\otimes\mf{sl}(2,\bb{R})\right)\\
&\simeq H^i\left(\mf{an};C^\infty(P,\bb{C})\otimes\mf{an}\right)\oplus H^i\left(\mf{an};C^\infty(P,\bb{C})\otimes\bb{R}_{-\frac{1}{\sqrt{2}}}\right)\\
&\simeq
\begin{cases}
\bb{C}^{8g_\Sigma-6}&i=1,2\\
0&\text{otherwise}. 
\end{cases}
\end{align*}
Hence by \eqref{lxifne}, 
\begin{align*}
H^i\left(\mca{F};\sl(2,\bb{R})\right)\simeq
\begin{cases}
\bb{R}^{8g_\Sigma-6}&i=1,2\\
0&\text{otherwise}. 
\end{cases}
\end{align*}

\section*{Acknowledgements}
The authors thank Yu Nishimura for attending all the Skype (or Zoom) seminars for the preparation of this work. They also thank Shigenori Matsumoto for letting them know that total minimizability follows from their computation. Most of the results in this paper were obtained when the first named author was a Research Fellow of Japan Society for the Promotion of Science. He also thanks Masahiko Kanai and Masayuki Asaoka for some discussions related to this paper.

\bibliography{maruhashi}
\end{document}